\documentclass[a4paper, 12pt]{amsart}
\usepackage{amscd,amssymb,epsfig,amsxtra,amsfonts}
\usepackage{amsmath}
\usepackage{stmaryrd, mathrsfs}
\usepackage{ upgreek }
\usepackage{verbatim}
\usepackage{diagrams}
\usepackage[toc,page]{appendix}

\usepackage{epsfig} 
\newcommand{\ms}[1]{\mbox{\tiny$#1$}}
\newcommand{\epsh}[2]
          {\begin{array}{c} \hspace{-1.3mm}
         \raisebox{-4pt}{\epsfig{figure=#1,height=#2}}
         \hspace{-1.9mm}\end{array}}

\newcommand{\lk}{\operatorname{lk}}
\newcommand{\coev}{\operatorname{coev}}
\newcommand{\ev}{\operatorname{ev}}

\newcommand{\Id}{\operatorname{Id}}

\newcommand{\End}{\operatorname{End}}
\newcommand{\Hom}{\operatorname{Hom}}

\newcommand{\tr}{\operatorname{tr}}

\newcommand{\Vect}{\operatorname{Vect}}
\newtheorem{theo}{Theorem}[section]
\newtheorem{Def}[theo]{Definition}
\newtheorem{The}[theo]{Theorem}
\newtheorem{Lem}[theo]{Lemma}
\newtheorem{Exem}[theo]{Example}
\newtheorem{Pro}[theo]{Proposition}
\newtheorem{HQ}[theo]{Corollary}
\newtheorem{rmq}[theo]{Remark}

\begin{document}
\title[A Hennings type invariant of $3$-manifolds]{A Hennings type invariant of $3$-manifolds from a topological Hopf superalgebra}        


\author{Ngoc Phu HA}        
\address{Universit\'e de Bretagne Sud, Laboratoire de Math\'ematiques de Bretagne Atlantique, UMR CNRS 6205, Centre de Recherche, Campus de Tohannic, BP 573     
F-56017 Vannes, France and Hung Vuong university, Department of mathematics-informatics, Nong Trang, Viet Tri, Phu Tho, VietNam}
\email{ngoc-phu.ha@univ-ubs.fr, ngocphu.ha@hvu.edu.vn}

\maketitle

\begin{abstract}
We prove the unrolled superalgebra $\mathcal{U}_{\xi}^{H}\mathfrak{sl}(2|1)$ has a completion which is a ribbon superalgebra in a topological sense where $\xi$ is a root of unity of odd order. Using this ribbon superalgebra we construct its universal invariant of links. We use it to construct an invariant of $3$-manifolds of Hennings type.

\end{abstract}
\vspace{25pt}



\section{Introduction}
The notion of an unrolled quantum group is introduced in \cite{CGP15} by Costantino, Geer and Patureau-Mirand. Then an unrolled quantum group is a quantum group with some additional generators which should be thought of the logarithms of the some other generators, for example in $\mathcal{U}_{q}^{H}\mathfrak{sl}(2)$ the additional generator is an element $H$ with the relation $q^{H}=K$ (see \cite{CGP15}). This element $H$ is a tool to construct a ribbon structure on representations of $\mathcal{U}_{q}^{H}\mathfrak{sl}(2)$. The category of weight modules of $\mathcal{U}_{q}^{H}\mathfrak{sl}(2)$ is ribbon and not semi-simple but the Hopf algebra is not ribbon. With this category $\mathcal{U}_{q}^{H}\mathfrak{sl}(2)$-mod one constructed the invariants of links and of $3$-manifolds (see \cite{CGP15, BePM12}). For the Lie superalgebra $\mathfrak{sl}(2|1)$, the associated unrolled quantum group is denoted by $\mathcal{U}_{\xi}^{H}\mathfrak{sl}(2|1)$ with two additional generators $h_1,\ h_2$ from the quantum group $\mathcal{U}_{\xi}\mathfrak{sl}(2|1)$. Using this unrolled quantum group in \cite{Ha16} one has shown that the category $\mathscr{C}^H$ of nilpotent weight modules over $\mathcal{U}_{\xi}^{H}\mathfrak{sl}(2|1)$ is ribbon and relative $G$-(pre)modular and lead to an invariant of links and of $3$-manifolds. The category $\mathscr{C}^H$ is ribbon thanks to the role of the additional elements $h_1, h_2$ which should be thought as the logarithms of $k_1,\ k_2$, i.e. $\xi^{h_i}=k_i$ for $i=1,2$. They help to construct quasitriangular ribbon structure in $\mathscr{C}^H$. The relations $\xi^{h_i}=k_i \ i=1,2$ also suggest that $k_1,\ k_2$ can be consider as elements of $\mathscr{C}^{\omega}(h_1, h_2)$, the vector space of holomorphic functions on $\mathbb{C}^{2}$. Following this idea we extend the superalgebra $\mathcal{U}_{\xi}^{H}\mathfrak{sl}(2|1)$ to a ribbon superalgebra in a topological sense, the topology determined by the norm of uniform convergence on compact sets. Its bosonization is a ribbon algebra (see in Section 2).

It is known that for each ribbon Hopf algebra one can construct an universal link invariant (all links are framed and oriented) (see  \cite{habiro2006bottom}, \cite{Ohtsuki02}). From some universal link invariants one could construct a $3$-manifold invariant. There is many way to do this. In \cite{Henning96}, Hennings introduced a method of building an invariant of $3$-manifolds by using an universal link invariant and a right integral. He worked with a finite dimensional ribbon algebra and this condition guarantees the existence of a right integral. In other way, Virelizier and Turaev constructed the invariants which called invariants of $\pi$-links and invariants of $\pi$-manifolds. They began with a ribbon Hopf $\pi$-coalgebra of finite type to construct the invariants of $\pi$-links, after that they renormalized the invariant to invariant of $\pi$-manifolds by using $\pi$-integrals (see \cite{Vire01}). Note that the $\pi$-integrals exist if and only if the group-coalgebra is finite type (see \cite{Vire02}). 
We show that for $\mathcal{U}_{\xi}\mathfrak{sl}(2|1)$ there is an associated Hopf group-coalgebra (see in Section 4). 
The ribbon superalgebra $\mathcal{U}_{\xi}^{H}\mathfrak{sl}(2|1)$ has infinite dimension and the family of quotients $\mathcal{U}_{\overline{\alpha}}^{H}$ of $\mathcal{U}_{\xi}^{H}\mathfrak{sl}(2|1)$ by the ideals $(k_{i}^{\ell}-\xi^{\ell\alpha_i})$ for $i=1, 2$ has a structure of Hopf group-coalgebra in the sense of Turaev-Virelizier which is ribbon but is not finite type, i.e. the method of construction the invariant of $3$-manifolds in \cite{Vire01} does not work here. 
We will present an another approach to construct an invariant of $3$-manifolds. We will use first the ribbon structure of $\mathcal{U}_{\xi}^{H}\mathfrak{sl}(2|1)$ to construct an universal invariant of links. The value of this invariant is represented by a product of a part containing $h_1,\ h_2$ which is a holomorphic function of variables $h_1,\ h_2$ and a part of the elements in copies of $\mathcal{U}_{\xi}\mathfrak{sl}(2|1)$. Assume the link is a surgery link in $S^3$ that produces a closed $3$-manifolds $M$.
Next we use a cohomology class $\omega\in H^{1}(M, G)$ and the discrete Fourier transform to reduce this element. This universal invariant of links
allows to construct an invariant of $3$-manifolds $(M,\omega)$ of Hennings type. 

The paper contains four sections. In section 2 we construct the topological ribbon structure of $\mathcal{U}_{\xi}^{H}\mathfrak{sl}(2|1)$ whose bosonization is a topological ribbon algebra. Section 3 builds the universal invariant of links from the topological ribbon superalgebra $\mathcal{U}_{\xi}^{H}\mathfrak{sl}(2|1)$. Finally, in section 4 we define the discrete Fourier transform from the topological ribbon superalgebra to a finite type Hopf $G$-coalgebra. This leads to definition in Theorem \ref{dl chinh1} of an invariant of pair $(M, \omega)$ as above.

\subsection*{Acknowledgments}
I would like to thank B. Patureau-Mirand, my thesis advisor, who helped me with this work, and who gave me the motivation to study mathematics. The author would also like to thank the professors and friends in the laboratory LMBA of the Universit\'e de Bretagne Sud and the Centre Henri Lebesgue ANR-11-LABX-0020-01 for creating an attractive mathematical environment.
\section{Topological ribbon Hopf superalgebra $\widehat{\mathcal{U}^H}$} 
In this section we recall the definition of Hopf superalgebra $\mathcal{U}_{\xi}^{H}\mathfrak{sl}(2|1)$ and we construct a topological ribbon Hopf superalgebra $\widehat{\mathcal{U}^H}$ which is a completion of $\mathcal{U}_{\xi}^{H}\mathfrak{sl}(2|1)$.
\subsection{Hopf superalgebra $\widehat{\mathcal{U}^H}$}
\subsubsection{Hopf superalgebra $\mathcal{U}_{\xi}^{H}\mathfrak{sl}(2|1)$}
\begin{Def}[\cite{Ha16}]
Let $\ell\geq 3$ be an odd integer and $\xi=\exp(\frac{2\pi i}{\ell})$.
The superalgebra $\mathcal{U}_\xi\mathfrak{sl}(2|1)$ is an associative superalgebra on $\mathbb{C}$ generated by the elements $k_1,k_2,k_1^{-1},k_2^{-1}, e_1,e_2,f_1,f_2$ 
and the relations
\begin{align}
&k_1k_2=k_2k_1,\\
&k_{i}k_{i}^{-1}=1, \ i=1,2,\\
&k_ie_jk_i^{-1}=\xi^{a_{ij}}e_j,\ k_if_jk_i^{-1}=\xi^{-a_{ij}}f_j \ i,j=1,2,\\
&e_1f_1-f_1e_1=\frac{k_1-k_1^{-1}}{\xi-\xi^{-1}},\ e_2f_2+f_2e_2=\frac{k_2-k_2^{-1}}{\xi-\xi^{-1}},\\
&[e_1, f_2]=0,\ [e_2, f_1]=0,\\
&e_2^{2}=f_2^{2}=0,\\
&e_1^{2}e_2-(\xi+\xi^{-1})e_1e_2e_1+e_2e_1^{2}=0,\label{relation serre 1}\\
&f_1^{2}f_2-(\xi+\xi^{-1})f_1f_2f_1+f_2f_1^{2}=0.
\end{align}
The last two relations are called the Serre relations. The matrix $(a_{ij})$ is given by $a_{11}=2,\ a_{12}=a_{21}=-1,\ a_{22}=0$. The odd generators are $e_2,\ f_2$. 
\end{Def}
 We define $\xi^{x}:=\exp(\frac{2\pi i x}{\ell})$, afterwards we will use the notation
 $$\{x\}= \xi^{x}-\xi^{-x}.$$
According to \cite{SMkVNt91}, $\mathcal{U}_\xi\mathfrak{sl}(2|1)$ is a Hopf superalgebra with the coproduct, counit and the antipode as below
\begin{align*}
&\Delta(e_i)=e_i \otimes 1 + k_i^{-1} \otimes e_i \quad i=1,2,\\ 
&\Delta(f_i)=f_i \otimes k_i + 1 \otimes f_i \quad i=1,2,\\
&\Delta(k_i)=k_i \otimes k_i \quad i=1,2,\\
&S(e_i)=- k_ie_i,\ S(f_i)=-f_ik_i^{-1},\ S(k_i)=k_i^{-1} \quad i=1,2,\\
&\varepsilon(k_i)=1,\ \varepsilon(e_i)=\varepsilon(f_i)=0 \quad i=1,2.
\end{align*}

The $\mathbb{C}$-superalgebra $\mathcal{U}_{\xi}^{H}\mathfrak{sl}(2|1)$ as $\mathcal{U}_{\xi}^{H}\mathfrak{sl}(2|1)=\left< \mathcal{U}_\xi\mathfrak{sl}(2|1),\ h_{i}, i=1,2\right>$ with the relations in $\mathcal{U}_\xi\mathfrak{sl}(2|1)$ and $[h_{i},e_{j}]=a_{ij}e_{j},\ [h_{i},f_{j}]=-a_{ij}f_{j}$, $[h_{i},h_{j}]=0,\ [h_{i},k_{j}]=0 \quad i, j = 1, 2$.\\
The superalgebra $\mathcal{U}_{\xi}^{H}\mathfrak{sl}(2|1)$ is a Hopf superalgebra where $\Delta,\ S$ and $\varepsilon$ are determined as in $\mathcal{U}_\xi\mathfrak{sl}(2|1)$ and by 
	$$\Delta(h_i)=h_i \otimes 1 + 1 \otimes h_i,\ S(h_i)=-h_i,\  \varepsilon(h_i)=0 \quad i=1,2.$$
Define the odd elements $e_3=e_1e_2-\xi^{-1}e_2e_1,\ f_3=f_2f_1-\xi f_1f_2$.
Denote by
\begin{align*}
\mathfrak{B}_{+}&=\{e_{1}^{p}e_{3}^{\rho}e_{2}^{\sigma},\ p \in \{0,1,...,\ell-1\},\ \rho,\sigma \in \{ 0,1 \}\},\\
\mathfrak{B}_{-}&=\{f_{1}^{p'}f_{3}^{\rho'}f_{2}^{\sigma'},\ p' \in \{0,1,...,\ell-1\},\ \rho', \sigma' \in \{ 0,1 \}\},\\
\mathfrak{B}_{0}&=\{k_{1}^{s_{1}}k_{2}^{s_{2}},\ s_{1}, s_{2} \in \mathbb{Z}  \}\ \text{and} \  \mathfrak{B}_{h}=\{h_{1}^{t_{1}}h_{2}^{t_{2}},\ t_{1}, t_{2} \in \mathbb{N}  \}.
\end{align*}
We consider the quotient $\mathcal{U}=\mathcal{U}_\xi\mathfrak{sl}(2|1)/(e_{1}^{\ell}, f_{1}^{\ell})$ which is the Hopf superalgebra having a Poincar\'e-Birkhoff-Witt basis $\mathfrak{B}_{+}\mathfrak{B}_{0}\mathfrak{B}_{-}$ and  $\mathcal{U}^{H}=\mathcal{U}_{\xi}^{H}\mathfrak{sl}(2|1)/(e_{1}^{\ell}, f_{1}^{\ell})$ which is the Hopf superalgebra with a Poincar\'e-Birkhoff-Witt basis $\mathfrak{B}_{+}\mathfrak{B}_{0}\mathfrak{B}_{h}\mathfrak{B}_{-}$.
\subsubsection{Topological Hopf superalgebra $\widehat{\mathcal{U}^{H}}$}
We recall some notions of topological tensor product and nuclear spaces in \cite{Markus13, Gro52}. A locally convex space $E$ is called nuclear, if all the compatible topologies on $E\otimes F$ agree for all locally convex spaces $F$, i.e. the topology on $E\otimes F$ compatible with $\otimes$ is unique. A topology is compatible with $\otimes$ if: 1) $\otimes:\ E \times F \rightarrow E \otimes F$ is continuous and 2) for all $(e, f) \in E' \times F'$ the linear form $e \otimes f:\ E \otimes F \rightarrow \mathbb{C},\ x \otimes y \mapsto e(x)f(y)$ is continuous \cite{Markus13}.
For two nuclear spaces $E$ and $F$ the completion of the tensor product $E \otimes F$ endowed with its compatible topology is denoted $E \widehat{\otimes} F$. A finite dimensional space is nuclear, the tensor product of two nuclear spaces is nuclear space and a space is nuclear if only if its completion is nuclear \cite{Gro52}. If $V$ is a finite dimensional $\mathbb{C}$-vector space we denote by $\mathscr{C}^{\omega}(V)$ the space of holomorphic functions on $V$ endowed with the topology of uniform convergence on compact sets, it is nuclear space. Remark that we have $\mathscr{C}^{\omega}(V_1) \widehat{\otimes} \mathscr{C}^{\omega}(V_2)\simeq \mathscr{C}^{\omega}(V_1 \times V_2)$ (Theorem 51.6 \cite{Fran67}) where $V_1,\ V_2$ are finite dimensional $\mathbb{C}$-vector spaces. For a quantum group, if $\mathfrak{H}$ is generated by Cartan generators and $V$ is a finite dimensional vector space then elements of $V \widehat{\otimes} \mathscr{C}^{\omega}(\mathfrak{H}^{*})$ can be seen as $V$-valued holomorphic functions.

Since the nuclear spaces form a symmetric monoidal category with the product $\widehat{\otimes}$ (\cite{Markus13}), we have the proposition. 
\begin{Pro} \label{completion}  
Let $\mathfrak{H}_i$ be $\mathbb{C}$-vector spaces of dimension $n_i$ and let $V_i$ be finite dimensional vector spaces on $\mathbb{C}$ for $i=1, 2$. Then
\begin{equation*}
(V_1 \otimes \mathscr{C}^{\omega}(\mathfrak{H}_{1}^{*})) \otimes (V_2 \otimes \mathscr{C}^{\omega}(\mathfrak{H}_{2}^{*})) \simeq (V_1 \otimes V_2) \otimes \mathscr{C}^{\omega}(\mathfrak{H}_{1}^{*}\times \mathfrak{H}_{2}^{*}).
\end{equation*}
\end{Pro}
\begin{proof}
By the symmetric and associative properties of $\widehat{\otimes}$ we have
\begin{equation*}
(V_1 \widehat{\otimes} \mathscr{C}^{\omega}(\mathfrak{H}_{1}^{*})) \widehat{\otimes} (V_2 \widehat{\otimes} \mathscr{C}^{\omega}(\mathfrak{H}_{2}^{*})) \simeq (V_1 \widehat{\otimes} V_2) \widehat{\otimes} \mathscr{C}^{\omega}(\mathfrak{H}_{1}^{*}) \widehat{\otimes} \mathscr{C}^{\omega}(\mathfrak{H}_{2}^{*}).
\end{equation*}
Furthermore, by Theorem 51.6 \cite{Fran67} $\mathscr{C}^{\omega}(\mathfrak{H}_{1}^{*}) \widehat{\otimes} \mathscr{C}^{\omega}(\mathfrak{H}_{2}^{*})\simeq \mathscr{C}^{\omega}(\mathfrak{H}_{1}^{*}\times \mathfrak{H}_{2}^{*})$. It implies 
\begin{equation*}
(V_1 \widehat{\otimes} \mathscr{C}^{\omega}(\mathfrak{H}_{1}^{*})) \widehat{\otimes} (V_2 \widehat{\otimes} \mathscr{C}^{\omega}(\mathfrak{H}_{2}^{*})) \simeq (V_1 \widehat{\otimes} V_2) \widehat{\otimes} \mathscr{C}^{\omega}(\mathfrak{H}_{1}^{*}\times \mathfrak{H}_{2}^{*}).
\end{equation*}
Since the spaces $V_i, \mathscr{C}^{\omega}(\mathfrak{H}_{i}^{*})$ for $i=1, 2$ are complete then $V_i \widehat{\otimes} \mathscr{C}^{\omega}(\mathfrak{H}_{i}^{*}) \simeq V_i \otimes \mathscr{C}^{\omega}(\mathfrak{H}_{i}^{*})$. Thus we get 
\begin{equation*}
(V_1 \otimes \mathscr{C}^{\omega}(\mathfrak{H}_{1}^{*})) \otimes (V_2 \otimes \mathscr{C}^{\omega}(\mathfrak{H}_{2}^{*})) \simeq (V_1 \otimes V_2) \otimes \mathscr{C}^{\omega}(\mathfrak{H}_{1}^{*}\times \mathfrak{H}_{2}^{*}).
\end{equation*}
\end{proof}
The space of entire functions is a nuclear Fr\'echet space obtained as the completion of polynomial functions for the topology of uniform convergence on compact sets. We use a similar completion to define a topological ribbon Hopf superalgebra from $\mathcal{U}_{\xi}^{H}\mathfrak{sl}(2|1)$. That is a topological ribbon Hopf superalgebra   $\widehat{\mathcal{U}_{\xi}^{H}\mathfrak{sl}(2|1)}$ where the topology is constructed as follow.
We consider $\mathcal{U}^{H}\simeq W \otimes_{\mathbb{C}} \mathbb{C}[h_1, h_2, k_{1}^{\pm 1}, k_{2}^{\pm 1}]$ as a vector space on $\mathbb{C}$ where $W$ is a finite dimensional vector space on $\mathbb{C}$ with the basis $$\mathfrak{B}=\mathfrak{B}_{+}\mathfrak{B}_{-}.$$ 
Let $\mathfrak{H}$ be $\mathbb{C}$-vector space with basis $\{h_1, h_2\}$ and $\mathfrak{H}^{*}$ be its dual, let $\mathscr{C}^{\omega}(h_1, h_2)$ be the vector space of holomorphic functions on $\mathbb{C}^{2}\simeq \mathfrak{H}^{*}$.
Now $\mathbb{C}[h_1, h_2, k_{1}^{\pm 1}, k_{2}^{\pm 1}]$ embeds in $\mathscr{C}^{\omega}(h_1, h_2)$ by sending $k_i$ to $\xi^{h_i}=\exp{(\frac{2i\pi}{\ell}h_i)}$. Thus $\mathcal{U}^{H}$ is embedded in $W \widehat{\otimes}_{\mathbb{C}} \mathscr{C}^{\omega}(h_1, h_2):=W {\otimes}_{\mathbb{C}} \mathscr{C}^{\omega}(h_1, h_2)$, in particular $k_i=1\otimes \xi^{h_i}\in W {\otimes}_{\mathbb{C}} \mathscr{C}^{\omega}(h_1, h_2)\ i=1,2$, this space is nuclear. In the following, we show that the completion $\widehat{\mathcal{U}^{H}}:=\widehat{\mathcal{U}_{\xi}^{H}\mathfrak{sl}(2|1)}$ of $\mathcal{U}^{H}$, which is $W {\otimes}_{\mathbb{C}} \mathscr{C}^{\omega}(h_1, h_2)$ has the topological Hopf algebraic structure continuously extended from $\mathcal{U}^{H}$ with the coproduct $\Delta:\ \mathcal{U}^{H} \rightarrow \mathcal{U}^{H} \widehat{\otimes}\ \mathcal{U}^{H}$. 
\begin{rmq}\label{cs}
For each $w_i \in \mathfrak{B}$ there exists $\vert w_i\vert= (\vert w_i\vert_1, \vert w_i\vert_2) \in \mathbb{Z}^2$ such that
\begin{align*}
&h_k w_i=w_i (h_k+\vert w_i \vert_k)\  \text{and}\\
&\forall w_i, w_j \in \mathfrak{B} \quad w_iw_j=\sum_{m}w_kc_{ij}^{m}(h_1, h_2), 
\end{align*} 
here $\vert w_i \vert_k\in \mathbb{Z}$ is the weight of $w_i$ for $h_k$ with $k=1,2$.
\end{rmq}
\begin{Lem}[PBW topological basis]
For any $u \in \widehat{\mathcal{U}^{H}}$ there exists unique $Q_{ij}(h_1, h_2) \in \mathscr{C}^{\omega}(h_1, h_2)\ 1\leq i, j \leq 4\ell$ such that
$$u=\sum_{i,j}u_iQ_{ij}(h_1, h_2)v_j $$ 
where $u_i \in \mathfrak{B}_-,\ v_j \in \mathfrak{B}_+$.
\end{Lem}
By Remark \ref{cs} each $x \in \mathcal{U}^{H}$ can be written 
\begin{equation}\label{x}
x=\sum_{k}w_{k}x_{k}(h_1,h_2)
\end{equation}
where $w_{k} \in \mathfrak{B},\ x_{k}(h_1,h_2) \in \mathscr{C}^{\omega}(h_1, h_2)$.\\
Let $K$ be a compact set in $\mathfrak{H}^{*}$. If $\phi \in K$ and $x(h_1, h_2)\in \mathscr{C}^{\omega}(h_1, h_2)$ then $\phi_*x(h_1, h_2)$ is the evaluation of $x$ at $\phi$, that is $\phi_*x(h_1, h_2)=x(\phi(h_1), \phi(h_2))\in \mathbb{C}$.
Define a norm associated to $K$ on $\widehat{\mathcal{U}^{H}}$ as follow
\begin{align}\label{norm of x}
\| x \|_{K}&=\| \sum_{k}w_{k}x_{k}(h_1,h_2) \|_{K} = \sup_{k}\sup_{\phi \in K}\vert \phi_{*}(x_{k}(h_1,h_2))\vert\\
&=\sup_{k}\sup_{\phi \in K}\vert x_{k}(\phi(h_1),\phi(h_2))\vert \nonumber \ \text{for}\ x\in \widehat{\mathcal{U}^{H}}.
\end{align}
In particular, $\| x \|_{K}=1$ when $x \in \mathfrak{B}$.  
The $\|.\|_{K}$ induces the topology of uniform convergence on compact sets. As $W \otimes \mathscr{C}^{\omega}(h_1, h_2)$ is complete and $\mathcal{U}^H$ is dense in it, we have the proposition below.
\begin{Pro}
The completion of $\mathcal{U}^{H}$ is $W \otimes_{\mathbb{C}} \mathscr{C}^{\omega}(h_1, h_2)$.
\end{Pro}
\begin{Pro}\label{bdt}
For each compact set $K \subset \mathfrak{H}^{*}$, there exists a compact set $K'$ and a $\lambda_K              \in \mathbb{R}$ such that
$\forall\ x,y \in \mathcal{U}^{H}$, we have
\begin{equation*} 
\| xy \|_{K} \leq \lambda_{K} \| x \|_{K'} \| y \|_{K}.
\end{equation*}
\end{Pro}
\begin{proof}
Given $x=\sum_{i}w_{i}x_{i}(h_1,h_2),\ y=\sum_{j}w_{j}y_{j}(h_1,h_2)$ then 
\begin{align*}
xy &=\sum_{i}w_{i}x_{i}(h_1,h_2)\sum_{j}w_{j}y_{j}(h_1,h_2)\\
&=\sum_{i,j}w_{i}w_{j}x_{i}(h_1+\vert w_{j} \vert_{1},h_2+\vert w_{j} \vert_{2})y_{j}(h_1,h_2)\\
&=\sum_{i,j,k}w_{k}c_{i,j}^{k}(h_1,h_2)x_{i}(h_1+\vert w_{j} \vert_{1},h_2+\vert w_{j} \vert_{2})y_{j}(h_1,h_2).
\end{align*}
\begin{align*}
&\| xy \|_{K}\\
&=\sup_{k}\sup_{\phi \in K}\vert\sum_{i,j} c_{i,j}^{k}(\phi(h_1),\phi(h_2))x_{i}(\phi(h_1)+\vert w_{j} \vert_{1},\phi(h_2)+\vert w_{j} \vert_{2})y_{j}(\phi(h_1),\phi(h_2))\vert\\
&\leq \sup_{k}\sup_{\phi \in K}\vert \sum_{i,j}c_{i,j}^{k}(\phi(h_1),\phi(h_2))\vert \sup_{i}\sup_{\phi \in K} \vert x_{i}(\phi(h_1)+\vert w_{j} \vert_{1},\phi(h_2)+\vert w_{j} \vert_{2})\vert\\
& \qquad \sup_{j}\sup_{\phi \in K}\vert y_{j}(\phi(h_1),\phi(h_2))\vert\\
&= \lambda_{K}\| x \|_{K+C} \| y \|_{K}
\end{align*} 
where $\lambda_{K}=\sup_{k}\sup_{\phi \in K}\vert \sum_{i,j}c_{i,j}^{k}(\phi(h_1),\phi(h_2))\vert$ and $C \subset \mathfrak{H}^{*}$ is the convex hull of weights of elements of $\mathfrak{B}$.
\end{proof}
Proposition \ref{bdt} implies that the product on $\mathcal{U}^{H}$ is continuous but there does not seem to exist multiplicative seminorms on $\mathcal{U}^H$. \\
By Proposition \ref{completion} we have $\mathcal{U}^H\widehat{\otimes} \mathcal{U}^H \simeq W^{\otimes 2}\otimes \mathscr{C}^{\omega}(h_{i,j})$ where $h_{i, 1}= h_i \otimes 1,\ h_{i, 2}=1\otimes h_i$ for $i=1, 2$ and the $h_{i,j}$ are seen as coordinates functions on $\mathfrak{H}^*\times \mathfrak{H}^*$. Thus we can write each $x \in \mathcal{U}^H \widehat{\otimes} \mathcal{U}^H$ form $x=\sum_{k}w_k x_{k}(h_{i,j})$ where $w_k \in W^{\otimes 2}$ and $x_{k}(h_{i,j}) \in \mathscr{C}^{\omega}(h_{i,j})$. We can define a norm of $x\in \mathcal{U}^H \widehat{\otimes} \mathcal{U}^H$ associated to a compact set $K_2 \subset \mathfrak{H}^* \times \mathfrak{H}^*$ by
\begin{equation}\label{norm of x, dn 2}
\| x\|_{K_2}=\sup_{k}\sup_{\phi \in K_2}\vert \phi_{*}x_{k}(h_{i,j})\vert= \sup_{k}\sup_{\phi \in K_2}\vert x_{k}(\phi(h_{i,j}))\vert.
\end{equation}
\begin{Pro}
For each compact set $K_2\subset \mathbb{C}^4$, there exists a compact set $K\subset \mathbb{C}^2$ and a $\lambda_{K_2} \in \mathbb{R}$ such that
$\forall\ x \in \mathcal{U}^{H}$, we have $$\|\Delta x\|_{K_2}\leq \lambda_{K_2}\| x\|_{K}.$$
\end{Pro}
\begin{proof}
Let $U$ be a compact set, $U=U_1\times U_2\subset \mathfrak{H}^{*}\times \mathfrak{H}^{*} \simeq \mathbb{C}^4$. 
First there exists $\lambda_U \in \mathbb{R}$ such that for any $a, a', b, b' \in \mathcal{U}^{H}$ we have 
\begin{align} \label{bdt phu 1}
\|(a \otimes b)(a' \otimes b') \|_{U_1\times U_2}&=\|aa' \otimes bb' \|_{U_1\times U_2}=\|aa'\|_{U_1}\|bb'\|_{U_2}\\
&\leq \lambda_{U_1}\|a\|_{U_{1}+C_{1}}\|a'\|_{U_{1}}\lambda_{U_2}\|b\|_{U_{2}+C_{2}}\|b'\|_{U_{2}}\nonumber\\
&=\lambda_{U_1}\lambda_{U_2}\|a\|_{U_{1}+C_{1}}\|b\|_{U_{2}+C_{2}}\|a'\|_{U_{1}}\|b'\|_{U_{2}}\nonumber\\
&=\lambda_U \|a\otimes b\|_{U+C_{1} \times C_{2}}\|a' \otimes b'\|_{U}\nonumber\\
&=\lambda_U \|a\otimes b\|_{U'}\|a' \otimes b'\|_{U} \nonumber
\end{align}
where $\lambda_U=\lambda_{U_1}\lambda_{U_2}$ and ${U'}=U+C_{1} \times C_{2}$.
Second let a compact set $K_2\subset \mathfrak{H}^{*}\times \mathfrak{H}^{*}$ and let $K\subset \mathfrak{H}^{*}$ be the compact set $\{\varphi+\psi|\ (\varphi, \psi) \in K_2 \}$. For $x\in \mathcal{U}^{H},\ x=\sum_{j}w_{j}x_{j}(h_1,h_2)$, we have
\begin{align*}
\|\Delta x\|_{K_2}&=\|\sum_{j}\Delta w_{j}\Delta x_{j}(h_1,h_2)\|_{K_2}
\leq \sum_{j}\|\Delta w_{j}\Delta x_{j}(h_1,h_2)\|_{K_2}\\
&=\sum_{j}\|\sum_{s}w_{j}^{1,s}\otimes w_{j}^{2,s}x_{j}(h_{1,1}+h_{1,2}, h_{2,1}+h_{2,2})\|_{K_2}\\
&\leq \sum_{j}\sum_{s}\|w_{j}^{1,s}\otimes w_{j}^{2,s}x_{j}(h_{1,1}+h_{1,2}, h_{2,1}+h_{2,2})\|_{K_2}\\
&\leq \sum_{j}\sum_{s} \lambda_{K_2,j,s}\|w_{j}^{1,s}\otimes w_{j}^{2,s}\|_{K_{2}+(C_1, C_2)}\|x_{j}(h_{1,1}+h_{1,2}, h_{2,1}+h_{2,2})\|_{K_2}\\
&\leq \sum_{j} \lambda_{K_2, j}\|x_{j}(h_{1,1}+h_{1,2}, h_{2,1}+h_{2,2})\|_{K_2}
\end{align*}
where the sums are finite and $ \lambda_{K_2,j,s},\  \lambda_{K_2, j}$ are constants and in the fifth inequality one used Inequality \eqref{bdt phu 1}.
Furthermore, let $\mathfrak{H}$ be vector space on $\mathbb{C}$ with basis $\{h_1, h_2\}$. The symmetric algebra $S(\mathfrak{H}\times \mathfrak{H})\simeq S(\mathfrak{H}\oplus \mathfrak{H})\simeq S\mathfrak{H} \otimes S\mathfrak{H}$ (see \cite{ChKa95}), it is a commutative algebra on $\mathbb{C}$ generated by $h_1\otimes 1,\ h_2 \otimes1,\ 1\otimes h_1,\ 1\otimes h_2$ and $\Hom_{Alg}(S\mathfrak{H} \otimes S\mathfrak{H}, \mathbb{C})\simeq \Hom_{Alg}(S(\mathfrak{H} \times \mathfrak{H}), \mathbb{C})\simeq \Hom_{Vect}(\mathfrak{H} \times \mathfrak{H}, \mathbb{C})\simeq (\mathfrak{H} \times \mathfrak{H})^*\simeq \mathfrak{H}^* \times \mathfrak{H}^*$. This isomorphism allows that for $(\varphi, \psi)\in \mathfrak{H}^* \times \mathfrak{H}^*$ one has $(\varphi, \psi)(h_i\otimes 1)=\varphi(h_i)$ and $(\varphi, \psi)(1\otimes h_i)=\psi(h_i)$ for $i=1, 2$.
It implies that
\begin{align*}
 \|x_{j}(h_{1,1}+&h_{1,2},\ h_{2,1}+h_{2,2})\|_{K_2}\\
&=\sup_{(\varphi, \psi)\in K_2} |(\varphi, \psi)_{*}x_{j}\left(h_{1,1}+h_{1,2}, h_{2,1}+h_{2,2}\right)|\\
&=\sup_{(\varphi, \psi)\in K_2} |(\varphi+ \psi)_{*}x_{j}\left(h_{1}, h_{2}\right)|=\|x_{j}(h_{1}, h_{2})\|_{K}.
\end{align*}
Hence $$\|\Delta x\|_{K_2} \leq \sum_{j} \lambda_{K_2, j}\|x_{j}(h_{1}, h_{2})\|_{K} \leq \lambda_{K_2}\|x\|_{K}$$
where $\lambda_{K_2}$ is a constant.
\end{proof}
This proposition implies that the coproduct is continuous. The antipode $S$ is also continuous by proposition below.
\begin{Pro}
For each compact set $K \subset \mathfrak{H}^{*}$ there exists a compact set $K'' \subset \mathfrak{H}^{*}$ and a constant $\lambda_K$ such that
$$\|S(x)\|_{K}\leq \lambda_K \|x\|_{K''} \ \text{for} \ x \in \mathcal{U}^{H}.$$
\end{Pro}
\begin{proof}
For $x=\sum_{j}w_{j}x_{j}(h_1,h_2)\in \mathcal{U}^{H}$ we have
\begin{align*}
\|S(x)\|_{K}&=\|\sum_{j}S(x_{j}(h_1,h_2)) S(w_{j})\|_{K}=\|\sum_{j}x_{j}(-h_1,-h_2) S(w_{j})\|_{K}\\
&\leq \sum_{j}\|x_{j}(-h_1,-h_2) S(w_{j})\|_{K}\\
&\leq \sum_{j} \lambda_{K,j}\|x_{j}(-h_1,-h_2)\|_{K'}\|S(w_{j})\|_{K}\\
&\leq \sum_{j} \lambda_{K,j}^{'}\|x_{j}(h_1,h_2)\|_{-K'}
\leq \lambda_K \|x\|_{-K'}
\end{align*}
where $\lambda_{K,j},\ \lambda_{K,j}^{'}$ and $\lambda_{K}$ are  constants.
\end{proof}
It is clear that the unit and counit are continuous. Hence the maps product, coproduct, unit, counit and the antipode of $\mathcal{U}^{H}$ are continuous (with the topology of uniform convergence on compact sets). Thus the topology of uniform convergence on compact sets of $\mathcal{U}^H$ is compatible with its algebraic structure. The maps product, coproduct, unit, counit and the antipode of $\mathcal{U}^{H}$ continuously extend to the completion $\widehat{\mathcal{U}^{H}}$. Note that the coproduct $\mathcal{U}^{H} \rightarrow \mathcal{U}^{H}\otimes \mathcal{U}^{H}$ extends to $\widehat{\mathcal{U}^{H}} \rightarrow \mathcal{U}^{H}\widehat{\otimes} \mathcal{U}^{H}$. The space $\widehat{\mathcal{U}^{H}}$ endows with these continuous maps is a topological Hopf superalgebra.  

Similarly, for $n\geq 2$ denote 
\begin{equation}\label{Eq hij}
h_{i,j}=1 \otimes ... \otimes h_{i} \otimes ... \otimes 1
\end{equation} 
where $h_{i}$ is in $j$-th position for $1\leq i \leq 2$ and $1\leq j \leq n$. Then the completion of $\mathcal{U}^{H\otimes n}$ is topological vector space $\mathcal{U}^{H\widehat{\otimes} n}\simeq W^{\otimes n} \otimes \mathscr{C}^{\omega}(h_{i,j})$ with the topology of uniform convergence on compact sets. Here $W^{\otimes n}$ is the tensor product of $n$ copies of $W$ and $\mathscr{C}^{\omega}(h_{i,j})$ is the vector space of holomorphic functions of $2n$ variables $\{h_{i,j}\}_{i=1,2}^{j=1,\ ...,n}$ in $\mathbb{C}^{2n}$.
Note also that the maps $\Delta_{i}^{[n]}:\  \mathcal{U}^{H\otimes n} \rightarrow \mathcal{U}^{H\otimes (n+1)}$ and $\varepsilon_{i}^{[n]}:\  \mathcal{U}^{H\otimes n} \rightarrow \mathcal{U}^{H\otimes (n-1)}$ continuously extend to $\mathcal{U}^{H\widehat{\otimes} n}$, 
here $\Delta_{i}^{[n]}$ and $\varepsilon_{i}^{[n]}$ determined by
\begin{equation*}
\Delta_{i}^{[n]}= \underbrace{\Id \otimes ... \otimes \Id }_{i-1}\otimes \Delta \otimes \underbrace{\Id \otimes ... \otimes \Id}_{n-i}
\end{equation*} 
and
\begin{equation*}
\varepsilon_{i}^{[n]}= \underbrace{\Id \otimes ... \otimes \Id }_{i-1}\otimes \varepsilon \otimes \underbrace{\Id \otimes ... \otimes \Id}_{n-i}
\end{equation*}
where $\Delta,\ \varepsilon$ are in $i$-th position. It follows that 
\begin{align*}
&\Id \otimes \Delta_{i}^{[n]} =\Delta_{i+1}^{[n+1]}, \Delta_{i}^{[n]} \otimes \Id = \Delta_{i}^{[n+1]}, \\
& \Delta_{i}^{[n+1]} \circ \Delta_{i}^{[n]} = \Delta_{i+1}^{[n+1]} \circ \Delta_{i}^{[n]},\\
&\Delta_{j}^{[n+1]} \circ \Delta_{i}^{[n]}=\Delta_{i}^{[n+1]} \circ \Delta_{j}^{[n]} \quad i\neq j
\end{align*}
and we denote $\Delta^{[n]}(x)=\sum x_{(1)}\otimes ... \otimes x_{(n)}$ for $x\in \mathcal{U}^{H}$.
Hence, each element $x$ of $\mathcal{U}^{H\widehat{\otimes} n}$ can be written $x=\sum_{k} w_{k} x_{k}(h_{i,j})$ where $w_{k} \in \mathfrak{B}^{\otimes n},\ x_{k}(h_{i,j}) \in \mathscr{C}^{\omega}(h_{i,j}):=\mathscr{C}^{\omega}(\mathfrak{H}^{* n})$. In particular, the element $k_{i,j}:=1\otimes ...\otimes k_i\otimes ...\otimes 1$ where $k_i$ is in $j$-th position is equal to $\xi^{h_{i,j}}=1\otimes ...\otimes \xi^{h_i}\otimes ...\otimes 1$ for $i=1, 2\ j=1,...,n$.
Let $K$ be a compact set in $\mathbb{C}^{2n} \simeq Vect_{\mathbb{C}}(h_{i,j})^{*}$. As in Definition \eqref{norm of x} we define 
$$\| x \|_{K}=\sup_{k} \sup_{\phi \in K} \vert \phi_{*}(x_{k}(h_{i,j}))\vert=\sup_{k} \sup_{\phi \in K} \vert x_{k}(\phi(h_{i,j}))\vert.$$
Recall that $\mathscr{C}^{H}$ is the even category of finite dimensional nilpotent modules over $\mathcal{U}^{H}$ (see in \cite{Ha16}).
\begin{Pro}
For any $V_1, ..., V_n \in \mathscr{C}^{H}$ the representation $\rho_{V_1\otimes ...\otimes V_n}:\ \mathcal{U}^{H\otimes n} \rightarrow \End_{\mathbb{C}}(V_1\otimes ...\otimes V_n)$ continuously extends to a representation $\mathcal{U}^{H\widehat{\otimes} n} \rightarrow \End_{\mathbb{C}}(V_1\otimes ...\otimes V_n)$.
\end{Pro}
\begin{proof}
Let $K$ be the compact set containing the weights of $V=V_1\otimes ...\otimes V_n$. We have $\rho_{V}:\ \mathcal{U}^{H\otimes n} \rightarrow \End(V)$ be continuous on compact set $K$. Indeed, let $x\in \mathcal{U}^{H \otimes n}$ and write $x=\sum_{k} w_{k} x_{k}(h_{i,j})$.
On the subspace of weights $\phi \in K$, $\rho_V\left(\sum_{k}w_k x_k(h_{i, j})\right)$ acts as $\|\rho_V\left(\sum_{k}w_k x_k(\phi(h_{i, j}))\right) \|\leq \sum_{k}\|w_k\|_{K}\|x_k\|_{K}\leq \lambda_{K}\|x\|_{K}$ with $\lambda_K$ is a constant. It implies that $\rho_V$ is continuous. This prove that it exists a continuous representation $\widehat{\rho_V}:\  \mathcal{U}^{H\widehat{\otimes} n} \rightarrow \End_{\mathbb{C}}(V)$.
\end{proof}
\subsection{Topological ribbon superalgebra $\widehat{\mathcal{U}^{H}}$}	
It is known in \cite{Ha16} that the operator $\mathcal{R}=\check{\mathcal{R}}\mathcal{K}$ on $\mathscr{C}^{H}$ where
\begin{align}\label{Eq K}
&\check{\mathcal{R}}=\sum_{i=0}^{\ell-1}\frac{\{1\}^{i}e_1^{i} \otimes f_1^{i}}{(i)_{\xi}!} \sum_{\rho=0}^{1}\frac{(-\{1\})^{\rho}e_3^{\rho} \otimes f_3^{\rho}}{(\rho)_{\xi}!}\sum_{\delta=0}^{1}\frac{(-\{1\})^{\delta}e_2^{\delta} \otimes f_2^{\delta}}{(\delta)_{\xi}!} \in \mathcal{U}^{H} \otimes \mathcal{U}^{H}, \nonumber\\
& (0)_{\xi}!=1,\ (i)_{\xi}!=(1)_{\xi}(2)_{\xi}\cdots (i)_{\xi},\ (k)_{\xi}=\frac{1-\xi^k}{1-\xi}\quad \text{and} \nonumber  \\
& \mathcal{K}=\xi^{-h_1 \otimes h_2 -h_2 \otimes h_1 - 2h_2 \otimes h_2} \in \mathcal{U}^{H\widehat{\otimes} 2}
\end{align}
satisfies these conditions below
\begin{align*} \label{dkcm}
&\Delta \otimes \Id(\mathcal{R})=\mathcal{R}_{13}\mathcal{R}_{23},\\
&\Id \otimes \Delta(\mathcal{R})=\mathcal{R}_{13}\mathcal{R}_{12},\\
&\mathcal{R}\Delta^{op}(x)=\Delta(x)\mathcal{R} \ \text{for all} \ x \in \mathcal{U}^{H}. 
\end{align*}
This operator is given by action of an element $\mathcal{R}$ is in the completion $\mathcal{U}^{H\widehat{\otimes} 2}$, so the proof of the lemma below follows the line of Theorem VIII.2.4 \cite{ChKa95}.
\begin{Lem}
The element $\mathcal{R}=\check{\mathcal{R}}\mathcal{K}$ is a topological universal $R$-matrix of the topological Hopf superalgebra $\widehat{\mathcal{U}^{H}}$.
\end{Lem}
The element $\mathcal{R}$ satisfies the properties 
\begin{align*}
&\mathcal{R}_{12}\mathcal{R}_{13}\mathcal{R}_{23}=\mathcal{R}_{23}\mathcal{R}_{13}\mathcal{R}_{12},\\
&(\varepsilon \otimes \Id_{\widehat{\mathcal{U}^{H}}})(\mathcal{R})=1=(\Id_{\widehat{\mathcal{U}^{H}}} \otimes \varepsilon)(\mathcal{R}),\\
&(S \otimes \Id_{\widehat{\mathcal{U}^{H}}})(\mathcal{R})=\mathcal{R}^{-1}=(\Id_{\widehat{\mathcal{U}^{H}}} \otimes S^{-1})(\mathcal{R}),\\
&(S \otimes S)(\mathcal{R})=\mathcal{R}.
\end{align*}
The completion $\widehat{\mathcal{U}^{H}}$ of $\mathcal{U}^{H}$ is a Hopf $\mathbb{C}$-superalgebra which has a pivotal element $\phi_{0}=k_{1}^{-\ell}k_{2}^{-2}$ (see Proposition 3.3 \cite{Ha16}). We define an even element $\theta$, invertible and in the center of $\widehat{\mathcal{U}^{H}}$ by
\begin{equation}\label{twist not sigma}
\theta =\phi_{0}.(m \circ \tau^{s} \circ (\Id \otimes S)(\mathcal{R}))^{-1}
\end{equation}
where $\tau^{s}:\ \mathcal{U}^{H}\widehat{\otimes}\mathcal{U}^{H}  \rightarrow \mathcal{U}^{H}\widehat{\otimes}\mathcal{U}^{H},\ x \otimes y \mapsto (-1)^{\deg x \deg y}y \otimes x$ is super-flip of $\mathcal{U}^{H}\widehat{\otimes}\mathcal{U}^{H}$. \\
We now show that the completion $\widehat{\mathcal{U}^{H}}$ with the element $\theta$ will be a ribbon Hopf superalgebra.
\begin{Pro}\label{steta}
The $\theta$ is a twist, i.e. the element $\theta$ satisfies
\begin{enumerate}
\item $\varepsilon(\theta)=1$,
\item $\Delta(\theta)=\tau^{s}(\mathcal{R}).\mathcal{R}.(\theta \otimes \theta)$,
\item $S(\theta)=\theta$.
\end{enumerate}
\end{Pro}
Equalities (1) and (2) follow from the definition of $\theta$. To prove (3), we need the following lemmas.\\
Let $\mathcal{U}^{h}$ be the sub-superalgebra of $\widehat{\mathcal{U}^{H}}$ of all elements commuting with $h_1,\ h_2$ we have the lemma.
\begin{Lem}
For $u \in \widehat{\mathcal{U}^{H}},\ u \in \mathcal{U}^{h}$ if and only if $u$ has the form 
\begin{equation}
u=\sum_{0\leq\rho,\sigma \leq1,\ 0\leq p\leq \ell -1}y_{\rho,\sigma, p}Q_{\rho,\sigma, p}(h_1, h_2)e_2^{\rho}e_3^{\sigma}e_1^{p}
\end{equation}
where $\text{weight}(y_{\rho,\sigma, p})+\text{weight}(e_2^{\rho}e_3^{\sigma}e_1^{p})=0$ and $Q_{\rho,\sigma, p}(h_1, h_2) \in \mathscr{C}^{\omega}(h_1, h_2)$.
\end{Lem}
Let $\mathcal{I}^{+}$ be a left ideal of $\widehat{\mathcal{U}^{H}}$ generated by $e_1,\ e_2$ and $e_3$, set $\mathcal{I}= \mathcal{I}^{+} \cap \mathcal{U}^{h}$.
\begin{Lem}
We have $\mathcal{I}= \mathcal{I}^{+} \cap\ \mathcal{U}^{h}=\mathcal{I}^{-} \cap\ \mathcal{U}^{h}$ and $\mathcal{U}^{h}=\mathscr{C}^{\omega}(h_1,h_2)\bigoplus \mathcal{I}$ where $\mathcal{I}^{-}$ is right ideal generated by $f_1,\ f_2$ and $f_3$.
\end{Lem}
Hence, $\mathcal{I}$ is a two-side ideal and the projection $\varphi:\ \mathcal{U}^{h} \rightarrow \mathscr{C}^{\omega}(h_1,h_2)$ is a homomorphism of algebra called the Harish-Chandra homomorphism.
\begin{Pro} \label{z=phiz}
Let $V_{\mu}$ be a simple highest weight $\mathcal{U}^{H}$-module with highest weight $\mu=(\mu_1, \mu_2)$. Then for any $z \in Z(\widehat{\mathcal{U}^{H}})$ and $v$ the highest weight vector of weight $\mu$ of $V_{\mu}$  $$ zv=\varphi(z)(\mu)v$$ where $\varphi(z)$ is in $\mathscr{C}^{\omega}(\mathfrak{H}^{*})$ and $\varphi(z)(\mu)$ is its value at $\mu=(\mu_1, \mu_2)$.
\end{Pro}
\begin{proof}
Let $v_{0}$ be a highest weight vector generating $V_{\mu}$ and $z$ a central element of $\widehat{\mathcal{U}^{H}}$. Following the lemmas above, $z$ can be written $$z=\varphi(z)+ \sum_{(\rho,\sigma, p) \neq (0,0,0)}y_{\rho,\sigma, p}Q_{\rho,\sigma, p}(h_1, h_2)e_2^{\rho}e_3^{\sigma}e_1^{p}.$$
Since $e_2^{\rho}e_3^{\sigma}e_1^{p}v_{0}=0$ for $(\rho,\sigma, p) \neq (0,0,0)$ and $h_iv_0=\mu_iv_0 \ i=1,2$, we get $zv_0=\varphi(z)(\mu_1, \mu_2)v_0$. If $v$ is an arbitrary vector of $V$, we have $v=xv_0$ for some $x$ in $\widehat{\mathcal{U}^{H}}$. It implies that $zv=zxv_0=xzv_0=\varphi(z)(\mu_1, \mu_2)xv_0=\varphi(z)(\mu_1, \mu_2)v$.
\end{proof}
By using this proposition, we have
\begin{Pro}\label{u=0}
Let $u$ be a central element of $\widehat{\mathcal{U}^{H}}$. If $\varphi(u)=0$ then $u=0$ where $\varphi$ is Harish-Chandra homomorphism. 
\end{Pro}
\begin{proof}
Let $u$ be a central element of $\widehat{\mathcal{U}^{H}}$ such that $\varphi(u)=0$. Assume $u$ is non-zero can be written as $$u=\sum_{(\rho,\sigma, p) \neq (0,0,0)}y_{\rho,\sigma, p}Q_{\rho,\sigma, p}(h_1, h_2)e_2^{\rho}e_3^{\sigma}e_1^{p}$$ where $Q_{\rho,\sigma, p}(h_1, h_2)$ are non-zero functions in $\mathscr{C}^{\omega}(h_1,h_2)$,   $0\leq \rho, \sigma \leq 1,\ 0\leq p \leq \ell -1$ and $(\rho,\sigma, p) \neq (0,0,0)$.\\
Consider a typical highest weight $\mathcal{U}^{H}$-module $V_\mu$ generated by highest weight vector $w_0$. It is known that the set of $4r$ vectors $B^{*}=\{S^{-1}(e_2^{\rho}e_3^{\sigma}e_1^{p})w_{0,0,0}^{*}\}$ forms a basis of $V_{\mu}^{*}$ where $0\leq \rho, \sigma \leq 1,\ 0\leq p \leq \ell -1, \ \{w_{\rho, \sigma,p}^{*}\}$ is the dual basis of $\{w_{\rho, \sigma,p}\}$ of $V_\mu$. In fact, the elements $S^{-1}(e_2^{\rho}e_3^{\sigma}e_1^{p})$ form up to multiplication by $k_1^{a}k_2^{b} \ a,b \in \mathbb{Z}$ a basis of the subalgebra $\mathcal{U}^{+}$ of $\mathcal{U}^{H}$ generated by $e_2^{\rho}e_3^{\sigma}e_1^{p} \ 0\leq \rho, \sigma \leq 1,\ 0\leq p \leq \ell -1$. Since $\mathcal{U}^{-}w_{0,0,0}^{*}=\mathbb{C}w_{0,0,0}^{*}$ where $\mathcal{U}^{-}$ is subalgebra of $\widehat{\mathcal{U}^{H}}$ generated by $f_2^{\rho}f_3^{\sigma}f_1^{p} \ 0\leq \rho, \sigma \leq 1,\ 0\leq p \leq \ell -1$, we have $\Vect(B^{*})\simeq \mathcal{U}^{+}w_{0,0,0}^{*}\simeq \mathcal{U}^{+}\mathcal{U}^{0}\mathcal{U}^{-}w_{0,0,0}^{*} \simeq \mathcal{U}^{H}w_{0,0,0}^{*}\simeq V_{\mu}^{*}$ where $\mathcal{U}^{0}$ is subalgebra of $\widehat{\mathcal{U}^{H}}$ topologically generated by $h_1, h_2$. Furthermore $\text{card}(B^{*})=\dim{V_{\mu}^{*}}$, hence $B^{*}$ is a basis of $V_{\mu}^{*}$. It exists in $V_\mu$ a dual basis $B=\{\widetilde{w}_{\rho,\sigma, p}\ 0\leq \rho, \sigma \leq 1,\ 0\leq p \leq \ell -1\}$ of $B^{*}$ in $V_{\mu}^{*}$, i.e. given $(\rho,\sigma, p)$, for any $e_2^{\rho^{'}}e_3^{\sigma^{'}}e_1^{p^{'}},\ w_{0,0,0}^{*}(e_2^{\rho^{'}}e_3^{\sigma^{'}}e_1^{p^{'}}\widetilde{w}_{\rho,\sigma, p})=\delta_{\rho}^{\rho^{'}}\delta_{\sigma}^{\sigma^{'}}\delta_{p}^{p^{'}}$.\\
On the one hand, Proposition \ref{z=phiz} implies that $u\widetilde{w}_{\rho,\sigma, p}=0$ for all $0\leq \rho, \sigma \leq 1,\ 0\leq p \leq \ell -1$. On the other, we have that $e_2^{\rho_0}e_3^{\sigma_0}e_1^{p_0}$ is an element having minimal weight of ones in the items of sum $\sum_{(\rho,\sigma, p) \neq (0,0,0)}y_{\rho,\sigma, p}Q_{\rho,\sigma, p}(h_1, h_2)e_2^{\rho}e_3^{\sigma}e_1^{p}$ such that $Q_{\rho_0,\sigma_0, p_0}(h_1, h_2) \neq 0$. It is clear that $e_2^{\rho}e_3^{\sigma}e_1^{p}\widetilde{w}_{\rho_0,\sigma_0, p_0}=0$ for $e_2^{\rho}e_3^{\sigma}e_1^{p}$ having the weight higher than one of $e_2^{\rho_0}e_3^{\sigma_0}e_1^{p_0}$ and $e_2^{\rho}e_3^{\sigma}e_1^{p}\widetilde{w}_{\rho_0,\sigma_0, p_0}=\delta_{\rho}^{\rho_0}\delta_{\sigma}^{\sigma_0}\delta_{p}^{p_0}w_{0,0,0}$ for $e_2^{\rho}e_3^{\sigma}e_1^{p}$ having the weight equal one of $e_2^{\rho_0}e_3^{\sigma_0}e_1^{p_0}$. Hence we have
\begin{align*}
&\sum_{(\rho,\sigma, p) \neq (0,0,0)}y_{\rho,\sigma, p}Q_{\rho,\sigma, p}(h_1, h_2)e_2^{\rho}e_3^{\sigma}e_1^{p}\widetilde{w}_{\rho_0,\sigma_0, p_0}\\
&=\sum_{\text{weight}(e_2^{\rho}e_3^{\sigma}e_1^{p})=\text{weight}(e_2^{\rho_0}e_3^{\sigma_0}e_1^{p_0})}y_{\rho,\sigma, p}Q_{\rho,\sigma, p}(h_1, h_2)\delta_{\rho}^{\rho_0}\delta_{\sigma}^{\sigma_0}\delta_{p}^{p_0}w_{0,0,0}\\
&=y_{\rho_0,\sigma_0, p_0}Q_{\rho_0,\sigma_0, p_0}(h_1, h_2)w_{0,0,0}\\
&=Q_{\rho_0,\sigma_0, p_0}(\mu)w_{\rho_0,\sigma_0, p_0}=0.
\end{align*} 
This result prove that $Q_{\rho_0,\sigma_0, p_0}(h_1, h_2)=0$. Thus $u=0$.
\end{proof}
\begin{Lem}
Let $\rho_{V}:\ \mathcal{U}^{H} \rightarrow \End(V)$ be a nilpotent finite dimensional representation of $\mathcal{U}^{H}$. We have
$$\rho_{V}(S(\theta))=\rho_{V}(\theta).$$
\end{Lem}
\begin{proof}
Recall that the category $\mathscr{C}^{H}$ of nilpotent representations of $\mathcal{U}_{\xi}^{H}\mathfrak{sl}(2|1)$ is a ribbon category having the twist is the family of isomorphisms $\theta_{V}:\ V \rightarrow V,\ \forall V \in \mathscr{C}^{H}, \ \theta_{V}=\rho_{V}(\theta)$ where $\rho_{V}:\ \mathcal{U}^{H} \rightarrow \End(V)$ is a representation of $\mathcal{U}^{H}$ (see \cite{Ha16}). It follows that $(\theta_{V})^{*}=\theta_{V^{*}}\ \forall V \in \mathscr{C}^{H}$. In fact $(\theta_{V})^{*}=(\rho_{V}(\theta))^{*}=(\ev_{V} \otimes \Id_{V^{*}})(\Id_{V^{*}} \otimes \theta_{V} \otimes \Id_{V^{*}})(\Id_{V^{*}} \otimes \coev_{V}):\ V^{*} \rightarrow V^{*}$ has matrix $(\rho_{V}(\theta))^t$ where $(\rho_{V}(\theta))$ is the matrix of the endomorphism $\rho_{V}(\theta)$. Furthermore $\theta_{V^{*}}=\rho_{V^{*}}(\theta)$ has matrix $(\rho_{V}(S(\theta)))^t$, so we have  
\begin{equation} \label{ptsteta}
\rho_{V}(\theta)=\rho_{V}(S(\theta)).
\end{equation}
\end{proof}
\begin{proof}[Proof of Proposition \ref{steta}]
Set $z=S(\theta)-\theta$, $z$ is in the center of $\widehat{\mathcal{U}^{H}}$. Let a weight module $V_{\mu}$ in $\mathscr{C}^{H}$ of weight $\mu$ and $v$ is a weight vector of $V_{\mu}$.
By Proposition \ref{z=phiz} and Equality \eqref{ptsteta} we have $\varphi(z)(\mu)v=zv=0$. It implies that $\varphi(z)(\mu)=0$. Furthermore $\varphi(z)\in \mathscr{C}^{\omega}(\mathfrak{H}^{*})$, this deduces that $\varphi(z)=0$, so $z=0$ by Proposition \ref{u=0}, i.e. $S(\theta)=\theta$. 
\end{proof}
Hence the results above give us the theorem.
\begin{The}\label{Main theorem 1}
The completion $\widehat{\mathcal{U}^{H}}$ of $\mathcal{U}^{H}$ is a topological ribbon superalgebra. 
\end{The}
\subsection{Bosonization of $\widehat{\mathcal{U}^{H}}$}\label{bosonization}
It is known that each ribbon superalgebra has an associated ribbon algebra, namely its bosonization (see  \cite{Majid94}). For the ribbon superalgebra $\widehat{\mathcal{U}^{H}}$, its bosonization denoted by $\widehat{\mathcal{U}^{H}}^{\sigma}$, is a topological ribbon algebra by adding an element $\sigma$ from $\widehat{\mathcal{U}^{H}}$, i.e. as an algebra, $\widehat{\mathcal{U}^{H}}^{\sigma}$ is the semi-direct product of $\widehat{\mathcal{U}^{H}}$ with $\mathbb{Z}/2\mathbb{Z}=\{1, \sigma\}$ where the
action of $\sigma$ is given by 
\begin{equation}\label{relation with sigma}
\sigma x=(-1)^{\deg x}x\sigma \quad \text{for}\ x \in \widehat{\mathcal{U}^{H}}.
\end{equation}
The coproduct $\Delta^{\sigma}$, the counity $\varepsilon^{\sigma}$ and the antipode $S^{\sigma}$ on $\widehat{\mathcal{U}^{H}}^{\sigma}$ given by 
\begin{itemize}
\item $\Delta^{\sigma}\sigma=\sigma \otimes \sigma,\ \Delta^{\sigma}(x)=\sum_{i}x_{i}\sigma^{\deg x_{i}^{'}}\otimes x_{i}^{'}$ where $\Delta(x)=\sum_{i}x_i \otimes x_{i}^{'}$ for $x \in \widehat{\mathcal{U}^{H}}$,
\item $\varepsilon^{\sigma}(\sigma)=1,\ \varepsilon^{\sigma}(x)=\varepsilon(x)$ for $x \in \widehat{\mathcal{U}^{H}}$ and 
\item $S^{\sigma}(\sigma)=\sigma, \ S^{\sigma}(x)=\sigma^{\deg x}S(x)$ for $x \in \widehat{\mathcal{U}^{H}}$.
\end{itemize}
The universal $\mathcal{R}$-matrix $\mathcal{R}^{\sigma}$ in $\widehat{\mathcal{U}^{H}}^{\sigma}$ determined by 
\begin{equation*}
\mathcal{R}^{\sigma}=R_1\sum_{i}R_{i}^{1}\sigma^{\deg R_{i}^2}\otimes R_{i}^2
\end{equation*}
where $R_1=\dfrac{1}{2}\left( 1\otimes 1+\sigma \otimes 1 + 1 \otimes\sigma-\sigma \otimes \sigma \right)$ and $\mathcal{R}=\sum_{i}R_{i}^{1}\otimes R_{i}^2$ is the universal $\mathcal{R}$-matrix in $\widehat{\mathcal{U}^{H}}$. 
Note that the universal $\mathcal{R}$-matrix $\mathcal{R}^{\sigma}$ can be written by 
\begin{equation}\label{R matrix}
\mathcal{R}^{\sigma}=\sum_{i}a_i\otimes b_i \sum_{j}\mathcal{K}_{j}^{1}\otimes \mathcal{K}_{j}^{2}
\end{equation}
where the terms $a_i, b_i$ do not contain $h_1, h_2$ for all $i$ and $\mathcal{K}=\sum_{j}\mathcal{K}_{j}^{1}\otimes \mathcal{K}_{j}^{2}$ is the Cartan part which contains only $h_1, h_2$ (see Equation \eqref{Eq K}). Its inverse denotes 
\begin{equation}\label{R matrix inverse}
\left(\mathcal{R}^{\sigma}\right)^{-1}=\sum_{j}\overline{\mathcal{K}}_{j}^{1}\otimes \overline{\mathcal{K}}_{j}^{2}\sum_{i} \overline{a}_i\otimes \overline{b}_i. 
\end{equation}
The pivotal element of the ribbon algebra $\widehat{\mathcal{U}^{H}}^{\sigma}$ is $\phi_{0}^{\sigma}=\sigma \phi_0$. We denote $\mathcal{U}^{\sigma}$ the Hopf subalgebra of $\widehat{\mathcal{U}^{H}}^{\sigma}$ generated by elements $e_i, f_i, k_i, k_i^{-1}$ for $i=1, 2$ and $\sigma$.
It is a pivotal Hopf algebra with a pivotal element $\phi_{0}^{\sigma}$.

\section{Universal invariant of link diagrams}
It is well known that from a ribbon algebra one can construct an universal invariant of oriented framed links, for exemple one can see these constructions which presented by Habiro (see \cite{habiro2006bottom}), Hennings (see \cite{Henning96}), Kauffman and Radford (see \cite{kauffman2001oriented}), Ohtsuki (see \cite{Ohtsuki02}), ... In previous section we proved that $\widehat{\mathcal{U}^{H}}$ is a ribbon superalgebra in the topological sense so its bosonization is a ribbon algebra. This ribbon algebra allows to construct an universal invariant of oriented framed links. In this section we apply the methods above to reconstruct an universal invariant of oriented framed links associated with the unrolled quantum group $\mathcal{U}^{H}$. Then we will use this invariant to construct an invariant of $3$-manifolds in the next section.
\subsection{Category of tangles}
We recall the category $\mathcal{T}$ of framed, oriented tangles (see \cite{habiro2006bottom}, \cite{ChKa95}). The objets are the tensor words of symbols $\downarrow$ and $\uparrow$, i.e. each word forms $x_1 \otimes ... \otimes x_n$ with $x_1, ..., x_n \in \{\downarrow, \uparrow\}, n\geq 0$. The tensor word of length $0$ is denoted by $1=1_{\mathcal{T}}$. The morphisms $T:\ w \rightarrow w'$ between $w, w' \in Ob(\mathcal{T})$ are the isotopy classes of framed, oriented tangles in a cube $[0,1]^{3}$ such that the endpoints at the bottom are descriped by $w$ and those at the top by $w'$.\\
The composition $gf$ of a composable pair $(f, g)$ of morphisms in $\mathcal{T}$ is obtained by placing $g$ above $f$, and the tensor product $f \otimes g$ of two morphisms $f$ and $g$ is obtained by placing $g$ on the right of $f$.\\
The braiding $c_{w,w'}:\ w \otimes w' \rightarrow w' \otimes w$ for $w, w' \in Ob(\mathcal{T})$ is the positve braiding of parallel of strings. The dual $w^{*} \in Ob(\mathcal{T})$ of $w\in Ob(\mathcal{T})$ is defined by $1^{*}=1,\ \downarrow ^{*}\ =\ \uparrow,\ \uparrow^{*}\ =\ \downarrow$ and
$$(x_1 \otimes ... \otimes x_n)^{*}=x_n^{*} \otimes ... \otimes x_1^{*}\quad\ \text{for}\ x_1, ..., x_n\in \{\downarrow,\uparrow\},\ n\geq 2.$$
For $w \in Ob(\mathcal{T})$, let
$$\ev_w:\ w^{*}\otimes w \rightarrow 1, \ \coev_w:\ 1 \rightarrow w \otimes w^{*}$$
denote the duality morphisms. For each object $w$ in $\mathcal{T}$, let $t_w:\ w \rightarrow w$ denote the positive full twist defined by
$$t_w=(w\otimes \ev_{w^{*}})(c_{w,w}\otimes w^{*})(w \otimes \coev_w).$$
It is well known that $\mathcal{T}$ is generated as a monoidal category by the objects $\downarrow, \uparrow$ and the morphisms 
$$c_{\downarrow,\downarrow},\ c_{\downarrow,\downarrow}^{-1},\  \ev_\downarrow,\ \coev_\downarrow,\ \ev_\uparrow,\ \coev_\uparrow$$ which are represented in Figure \ref{Fig0}.
\begin{figure}
$$\epsh{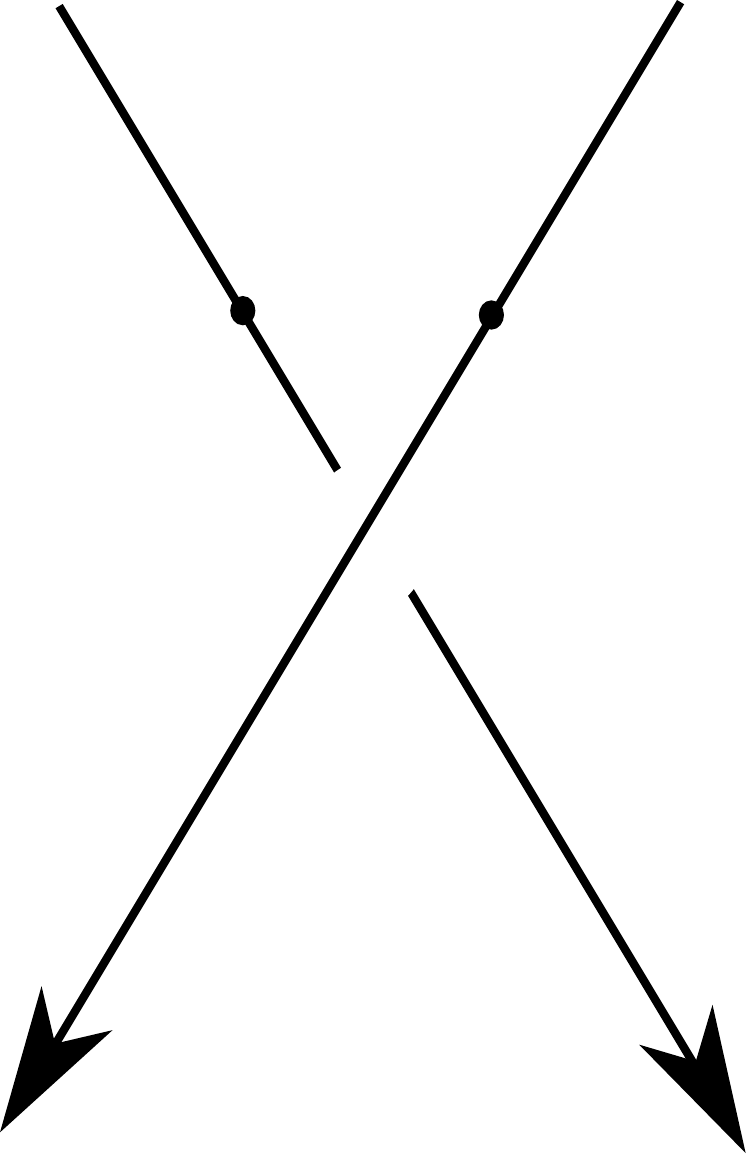}{8ex} \qquad 
\epsh{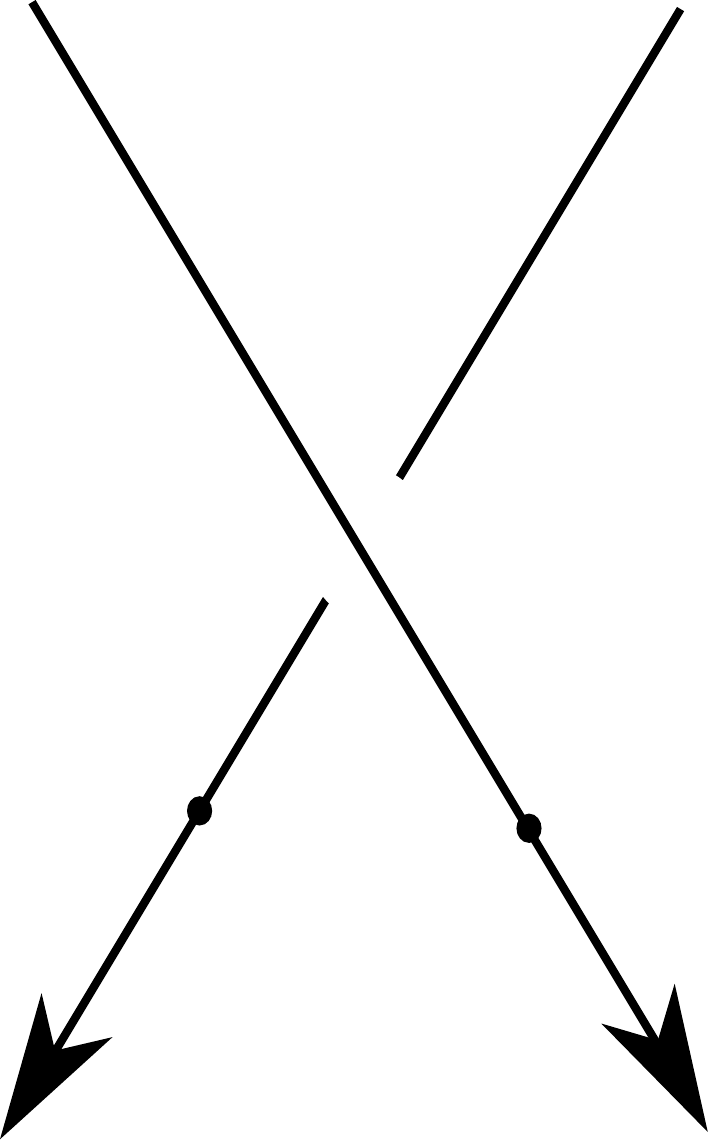}{8ex}  \qquad
\epsh{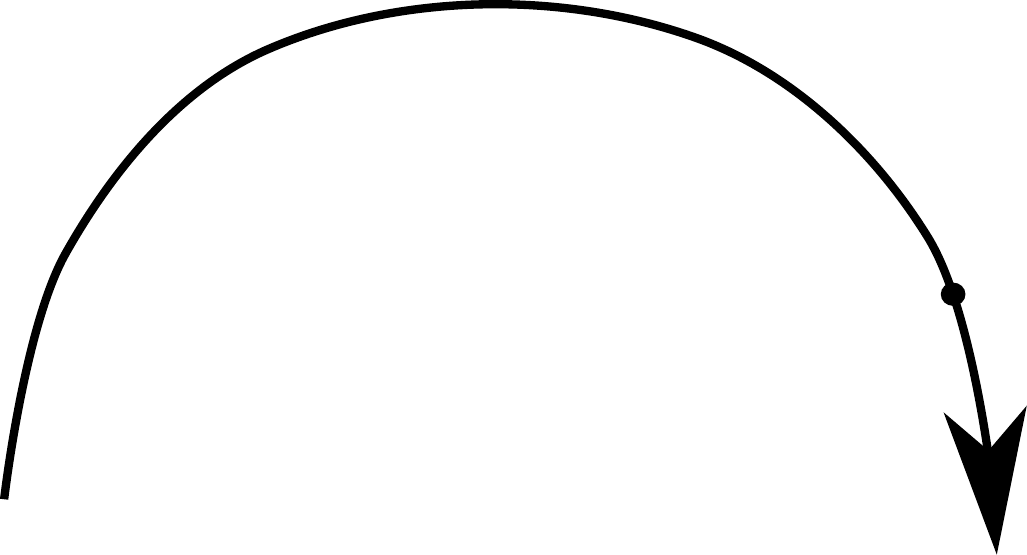}{4ex} \qquad
\epsh{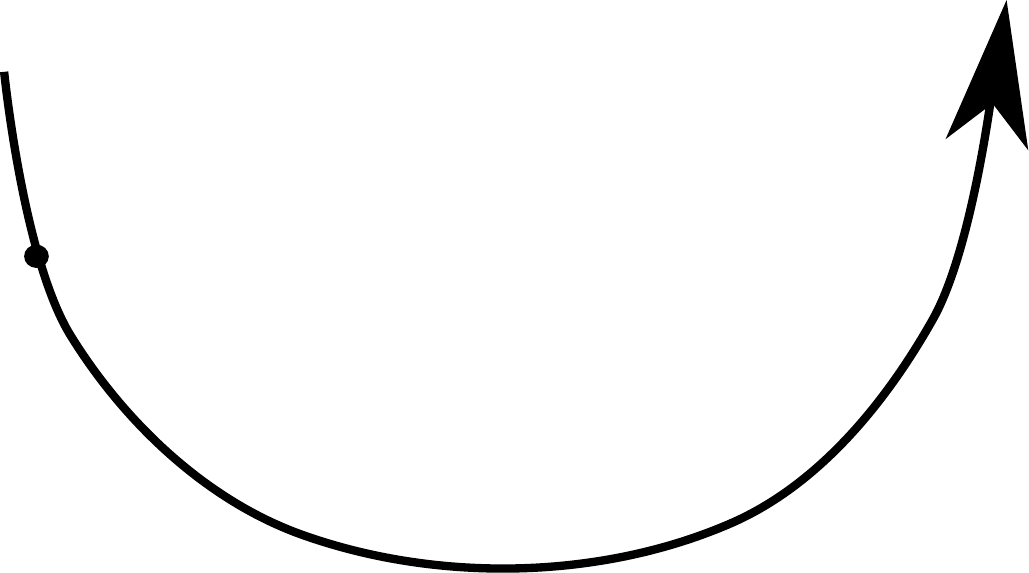}{4ex}  \qquad
\epsh{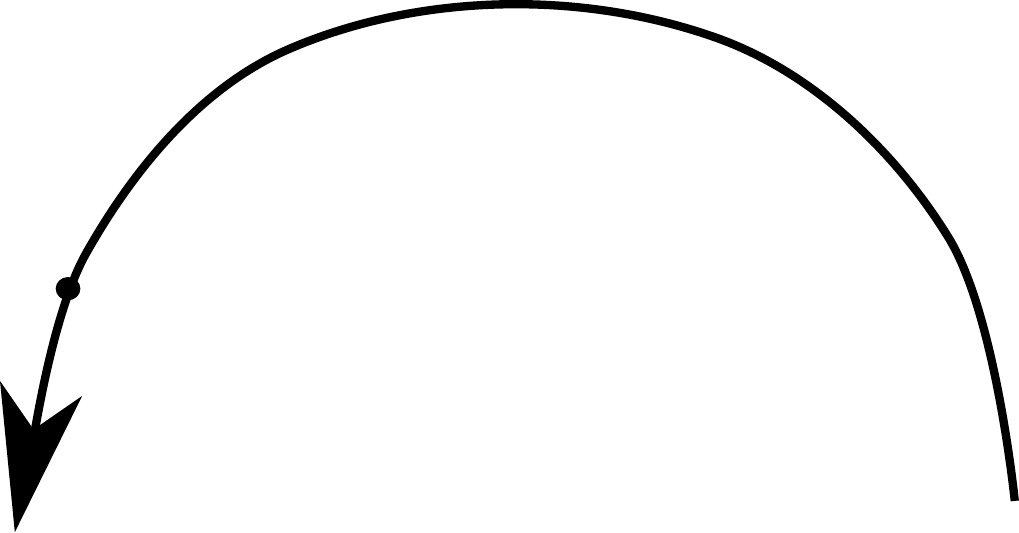}{4ex}  \qquad
\epsh{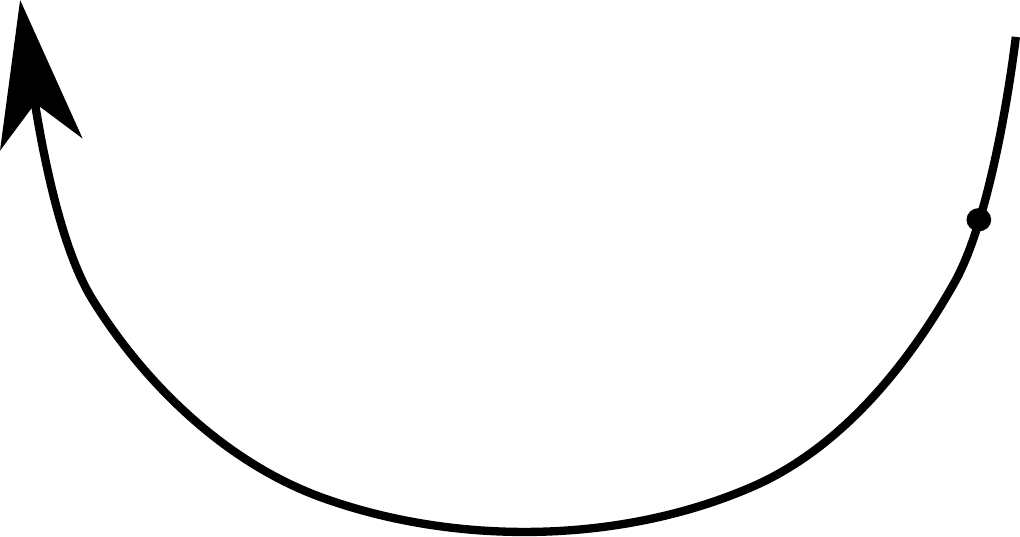}{4ex}  \qquad
$$
\caption{The morphisms 
$c_{\downarrow,\downarrow},\ c_{\downarrow,\downarrow}^{-1},\ \ev_\downarrow,\ \coev_\downarrow,\ \ev_\uparrow,\ \coev_\uparrow$}
	\label{Fig0}
\end{figure}

A {\em string link} is a tangle without closed component whose arcs end at the same order as they start, with downwards orientation. 
\subsection{Universal invariant of link diagrams}
We recall the notion of the $0^{\text{th}}$-Hochschild homology for an algebra $A$, that is $\text{H\!H}_{0}(A):=A/[A,A]$ where $[A,A]=\text{Span}\{xy-yx:\ x,y \in A\}$.
Let $L=L_1 \cup ... \cup L_{n}$ be a (framed, oriented) link diagram consisting of $n$ ordered circle components $L_1, ..., L_n$ with $n \geq 0$.

\begin{figure}
$$\epsh{fig1_1}{8ex} \put(-27,10){\ms{\alpha}} \put(-6,10){\ms{\beta}} \qquad 
\epsh{fig1_2}{8ex} \put(-27,-8){\ms{\overline{\alpha}}} \put(-4,-8){\ms{\overline{\beta}}} \qquad
\epsh{fig1_12}{4ex} \put(0,2){\ms{1}} \qquad
\epsh{fig1_10}{4ex} \put(-47,4){\ms{1}} \qquad
\epsh{fig1_11}{4ex} \put(-48,0){\ms{\phi_{0}}} \qquad
\epsh{fig1_9}{4ex} \put(0,2){\ms{\phi_{0}^{-1}}} \qquad
$$
\caption{Place elements on the strings}
	\label{Fig1}
\end{figure}

\begin{figure}
$$\epsh{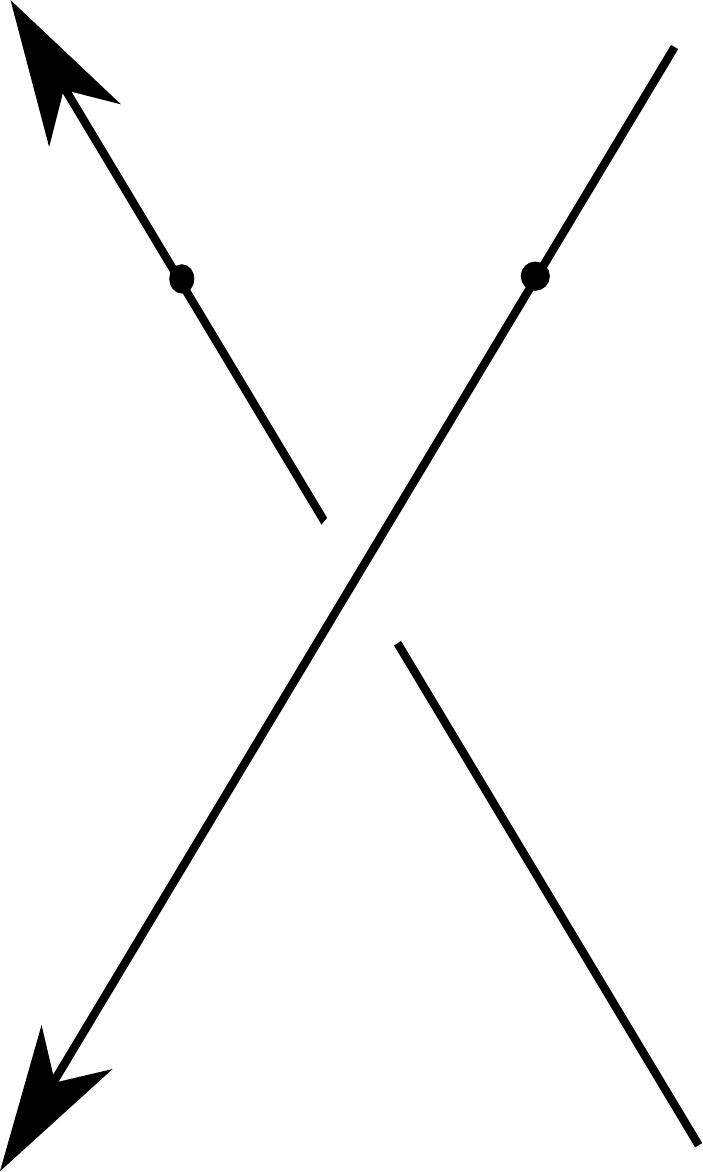}{8ex} \put(-37,10){\ms{S(\alpha)}} \put(-6,10){\ms{\beta}} \qquad 
\epsh{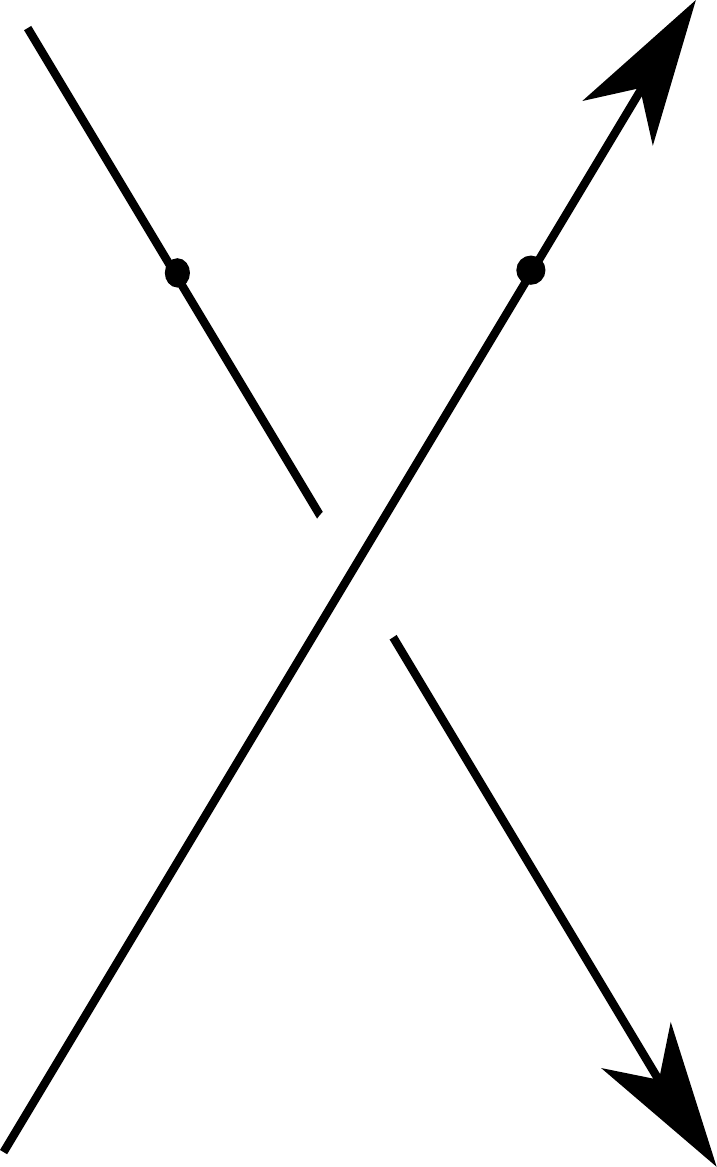}{8ex} \put(-27,10){\ms{\alpha}} \put(-5,9){\ms{S(\beta)}} \qquad \quad
\epsh{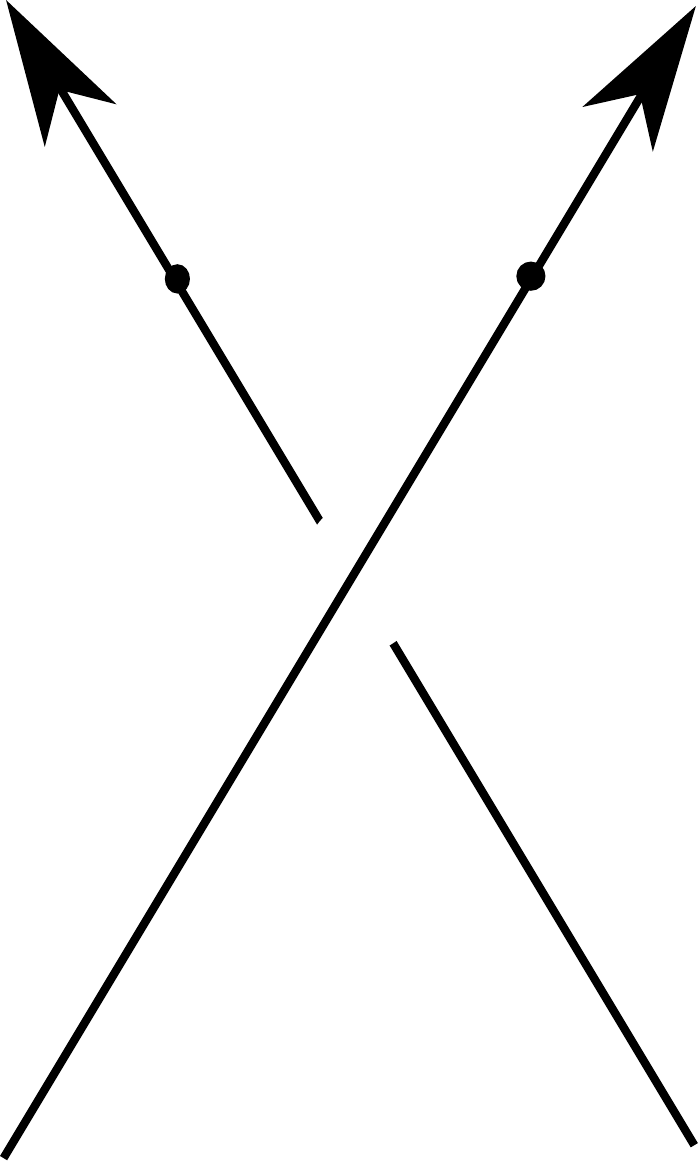}{8ex} \put(-37,10){\ms{S(\alpha)}} \put(-5,9){\ms{S(\beta)}} \qquad \quad
 \epsh{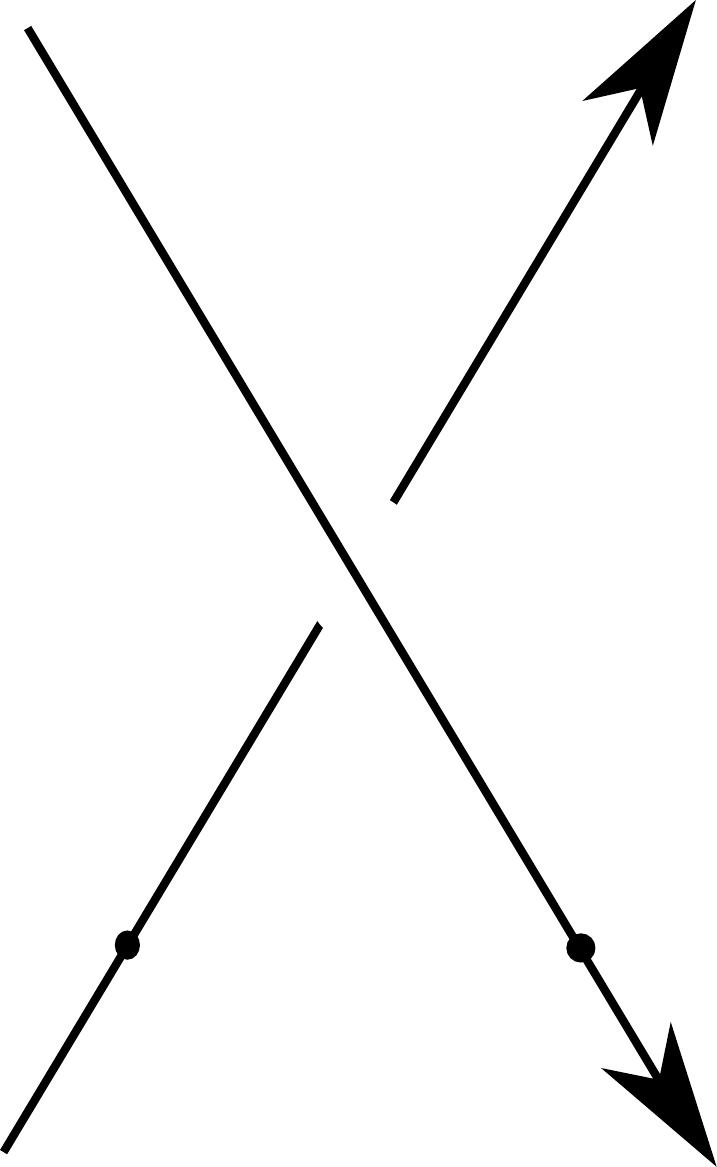}{8ex} \put(-37,-7){\ms{S(\overline{\alpha})}} \put(-4,-7){\ms{\overline{\beta}}} \qquad
 \epsh{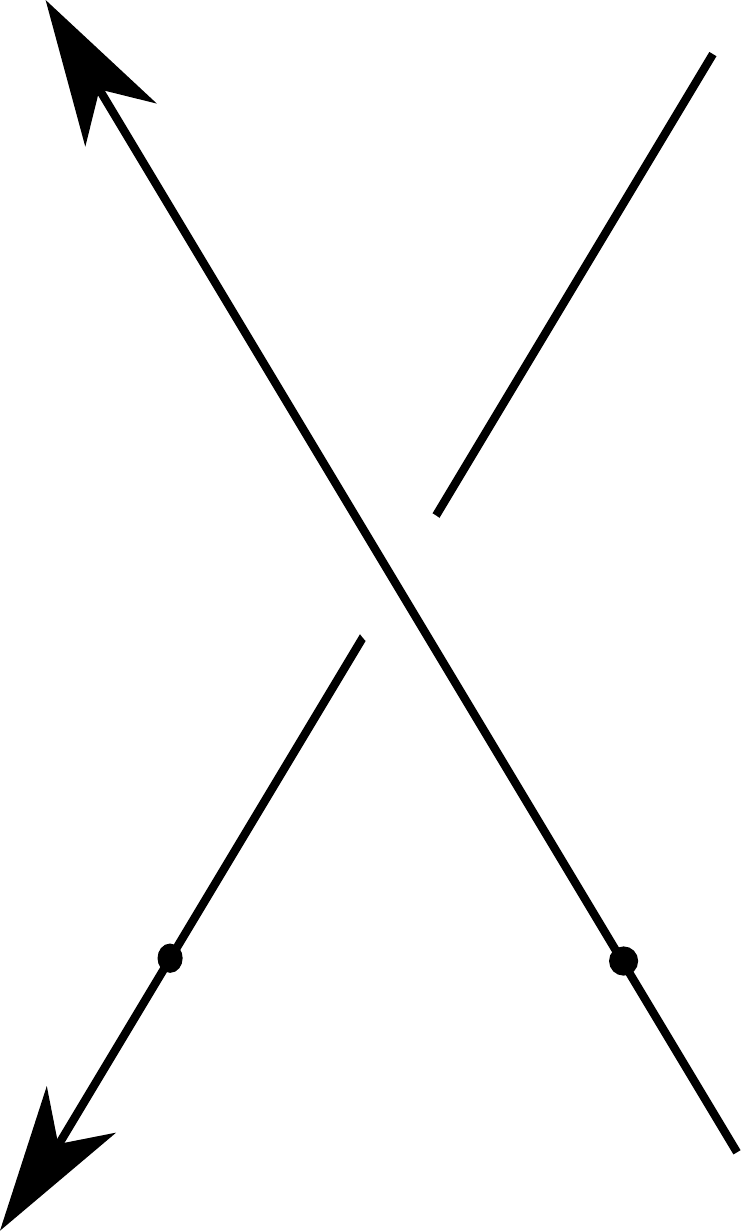}{8ex} \put(-27,-8){\ms{\overline{\alpha}}} \put(0,-8){\ms{S(\overline{\beta})}} \qquad \qquad
 \epsh{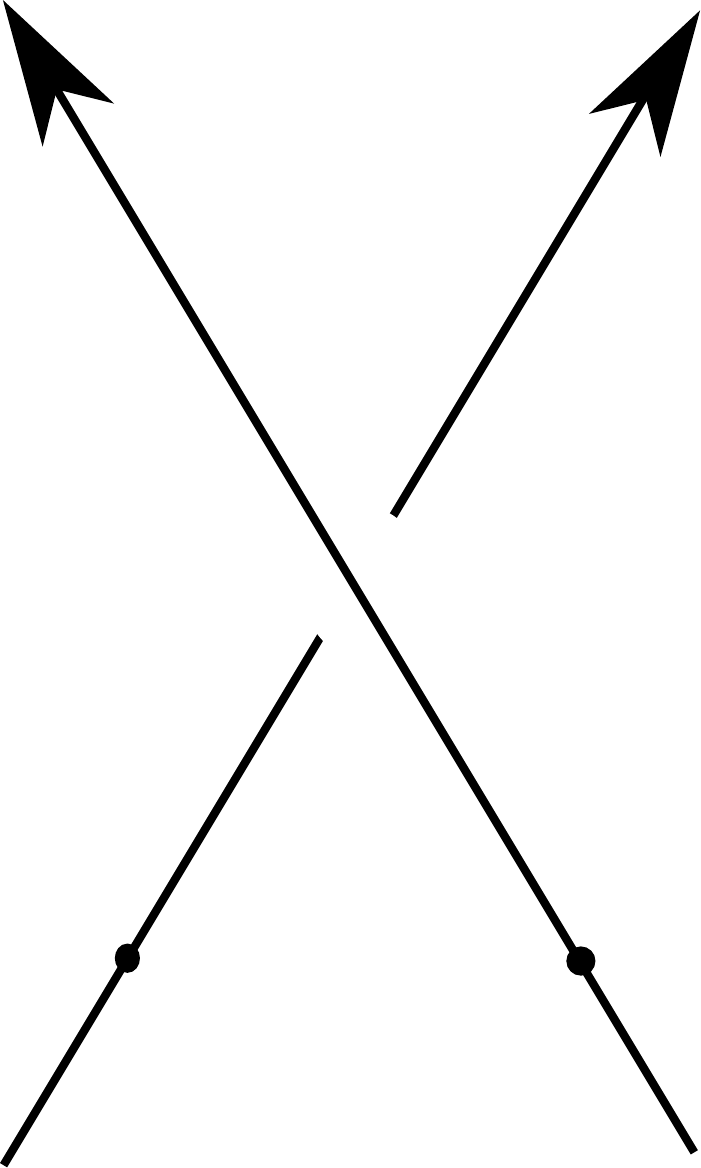}{8ex} \put(-37,-7){\ms{S(\overline{\alpha})}} \put(0,-8){\ms{S(\overline{\beta})}} \qquad
$$
\caption{The cases of crossings with upwards strings where $\mathcal{R}=\sum\alpha\otimes\beta$ and $\mathcal{R}^{-1}=\sum\overline{\alpha}\otimes\overline{\beta}$.}
	\label{Fig2}
\end{figure}
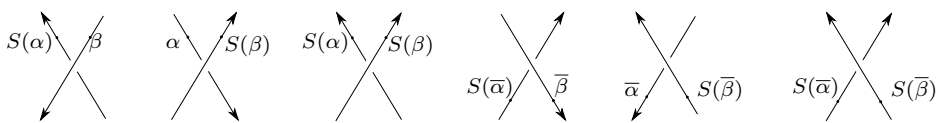

We can put elements of $\widehat{\mathcal{U}^{H}}^{\sigma}$ on the strings of $L$ according to the rule depicted in Figure \ref{Fig1} or in two Figures \ref{Fig1} and \ref{Fig2}. For each $j=1, ..., n$, we define $\mathcal{J}_{L_j}$ by first obtaining a word $w_j$ to be the product of the elements put on the component $L_j$ where these elements are read along the orientation of $L_j$ starting from any point (point basis) in $L_j$. Then set $\mathcal{J}_{L_j}=\tr_{q}(w_j)$ where $\tr_{q}:\ \widehat{\mathcal{U}^{H}}^{\sigma} \rightarrow \text{H\!H}_{0}\left(\widehat{\mathcal{U}^{H}}^{\sigma}\right)$, here $\text{H\!H}_{0}\left(\widehat{\mathcal{U}^{H}}^{\sigma}\right)$ is the $0^{\text{th}}$-Hochschild homology for the algebra $\widehat{\mathcal{U}^{H}}^{\sigma}$. We define
\begin{equation}\label{invariant of link}
\mathcal{J}_{L}=\sum \mathcal{J}_{L_1}\otimes ...\otimes \mathcal{J}_{L_n} \in \text{H\!H}_{0}\left(\widehat{\mathcal{U}^{H}}^{\sigma\otimes n}\right).
\end{equation}
\begin{The}[see also Theorem 4.5 \cite{Ohtsuki02}]
$\mathcal{J}_{L}$ is a topological invariant of framed links.
\end{The}
\begin{proof}
The proof in the finite dimensional setting apply without change. One can show that $\mathcal{J}_{L_j}$ does not depend on where we start reading the element on the closed components, and $\mathcal{J}_{L}$ is invariant under the Reidemeister moves for oriented links. This proves $\mathcal{J}_{L}$ is an invariant of framed links.
\end{proof}
We can similarly define the invariant of the string links by
\begin{equation}
\mathcal{J}_{T}=\sum \mathcal{J}_{T_1}\otimes ...\otimes \mathcal{J}_{T_n} \in \widehat{\mathcal{U}^{H}}^{\sigma \otimes n}
\end{equation}
where $T$ is a string link consisting of $n$ components $T_i,\ 1\leq i \leq n$, $\mathcal{J}_{T_i}$ is determined as above.
The relation between the invariant of tangles and of links is similar as Proposition 7.3 in \cite{habiro2006bottom}:
\begin{Pro}
If $T$ is a string link, then we have $$\mathcal{J}_{cl(T)}=\tr_{q}^{\otimes n}\left((\phi_0 \otimes ... \otimes \phi_0)(\mathcal{J}_{T})\right)$$
where $cl(T)$ is the closure of $T$.
\end{Pro}
\subsection{Value of universal invariant of link diagrams}
For $x, y \in \mathbb{C}^2\times \mathbb{C}^2$, call $Q(x,y)$ the polarization of the quadratic form determined by the matrix $B=(b_{ij})$ which is given by $b_{11}=0,\ b_{12}=b_{21}=-1,\ b_{22}=-2$.
Recall that $h_{i,j}=1\otimes ...\otimes h_i \otimes ... \otimes 1$ where $h_i$ is in $j$-th position for $i=1, 2$ and $j = 1, ..., n$. Let $\mathfrak{H}^{(n)}=\Vect_{\mathbb{C}}\{h_{ij}\}\subset \mathcal{U}^{H\otimes n}$ and $Q_{ij}$ be the quadratic form on $\mathfrak{H}^{(n)^*}$ defined by
$$Q_{ij}(h)=Q(h_{[i]}, h_{[j]})=h_{[i]}^{t}Bh_{[j]}$$ where $h_{[i]}$ is the column matrix $ \begin{pmatrix} h_{1,i}\\ h_{2,i} \end{pmatrix}$ for $i=1,..., n$. 
Recall also the formula for the universal $\mathcal{R}$-matrix and its inverse in Equations \eqref{R matrix} and \eqref{R matrix inverse}.

Let $L$ be a link diagram consisting of $n$ ordered circle components $L_1, ..., L_n$.
Denote by $\lk = (\lk_{ij})$ the linking matrix of link diagrams $L$ and
set $Q_{L}(h)=\sum_{1\leq i, j \leq n}\lk_{ij}Q_{ij}(h)$. We consider the algebraic automorphisms $\varphi_{ij}, \varphi_{Q_L}$ of $\widehat{\mathcal{U}^{H}}^{\sigma \otimes n}$ given by 
\begin{equation*}
\varphi_{ij}(x)= \xi^{-Q_{ij}(h)}x\xi^{Q_{ij}(h)}, \quad \varphi_{Q_L}(x)=\xi^{-Q_L(h)}x\xi^{Q_L(h)}\ \text{for}\ x \in \widehat{\mathcal{U}^{H}}^{\sigma \otimes n}.
\end{equation*}
Remark that $\varphi_{ij}$ and $\varphi_{Q_L}$ restrict to an automorphism of $\mathcal{U}^{\sigma\otimes n}$. Indeed, we denote the weight of an element $x \in \mathcal{U}^{H}$ for $h_i$ by 
$\vert x\vert_{i},\ i=1,2$, we have that $\vert x\vert_{i} \in \mathbb{Z}$.
We also recall that $$h_i x=x(h_i+\vert x\vert_{i}), \ x h_i = (h_i-\vert x\vert_{i})x \ \text{for} \ x \in \mathcal{U}^{H}.$$
These equalities imply that for $x=\bigotimes_{k=1}^{n}x_k \in \mathcal{U}^{\sigma\otimes n}$ we have $$\bigotimes_{k=1}^{n}x_k\xi^{h_{1,i}h_{2,j}}=\xi^{1\otimes ...\otimes (h_1-\vert x_i\vert_{1})\otimes ... \otimes (h_2-\vert x_j\vert_{2})\otimes ... \otimes 1}\bigotimes_{k=1}^{n}x_k.$$
Then $\xi^{h_i}=k_i \in \mathcal{U}^{\sigma}$ implies that $x\xi^{h_{1,i}h_{2,j}}=\xi^{h_{1,i}h_{2,j}}x'$ with $x'\in \mathcal{U}^{\sigma\otimes n}$. This deduces that $\varphi_{ij}(\mathcal{U}^{\sigma\otimes n}) = \mathcal{U}^{\sigma\otimes n}$ for $1\leq i, j \leq n$ and $\varphi_{Q_L}(\mathcal{U}^{\sigma\otimes n}) = \mathcal{U}^{\sigma\otimes n}$.
We have the theorem.
\begin{The}\label{gtri J}
The value of the invariant of $L$ satisfies
\begin{equation}
\xi^{-Q_{L}(h)}\mathcal{J}_{L}\in\mathcal{U}^{\sigma\otimes n}/N_{Q_L}
\end{equation}
where $N_{Q_L}= \xi^{-Q_{L}(h)}[\widehat{\mathcal{U}^{H}}^{\sigma\otimes n}, \ \widehat{\mathcal{U}^{H}}^{\sigma\otimes n}]\cap \mathcal{U}^{\sigma \otimes n}$.

\end{The}
\begin{proof}
We represent the value of $\mathcal{J}_{L}$ by the product of two parts, the first one is in $\mathscr{C}^{\omega}(\mathfrak{H}^*)$ and the second one is in the tensor product of copies of $\mathcal{U}^\sigma$ as follow. For $j=1,..., n$ we first put the element of $\widehat{\mathcal{U}^{H}}^\sigma$ on the strands $L_j$ following the rule depicted in Figure \ref{Fig1}. Second, we fix the Cartan parts of the elements at the cross points and then push the rest of the elements to the point basis of strand (along the orientation of $L_j$), the product of this part gives an element $w_j\in \mathcal{U}^\sigma$ for $j=1,..., n$. At each point of crossing $(i, j)$ between the $i$-strand and $j$-strand of $L$, its Cartan part gives us the element $$\xi^{\varepsilon_{ij} (-h_{1,i}h_{2,j}-h_{2,i}h_{1,j}-2h_{2,i}h_{2,j})}=\xi^{\varepsilon_{ij} Q_{ij}(h)}$$
where $\varepsilon_{ij}=\pm 1$ is the sign of the crossing $(i, j)$. Hence the value of $\mathcal{J}_{L}$ can be written as a product of $\xi^{Q_L(h)}$ and an element of the quotient of $\mathcal{U}^{\sigma\otimes n}$ by subspace $N_{Q_L}:=\xi^{-Q_{L}(h)}[\widehat{\mathcal{U}^{H}}^{\sigma\otimes n}, \ \widehat{\mathcal{U}^{H}}^{\sigma\otimes n}]\cap \mathcal{U}^{\sigma \otimes n}$.
This means that $\xi^{-Q_{L}(h)}\mathcal{J}_{L}\in\mathcal{U}^{\sigma\otimes n}/N_{Q_L}$.
\end{proof}
\begin{rmq}
As $[\widehat{\mathcal{U}^{H}}^{\sigma\otimes n}, \ \widehat{\mathcal{U}^{H}}^{\sigma\otimes n}]=\text{Span}_{\mathbb{C}}\{xy-yx| \ x, y \in \widehat{\mathcal{U}^H}^{\sigma\otimes n}\}$ then $\xi^{-Q_{L}(h)}[\widehat{\mathcal{U}^{H}}^{\sigma\otimes n}, \ \widehat{\mathcal{U}^{H}}^{\sigma\otimes n}]=\text{Span}_{\mathbb{C}}\{\xi^{-Q_{L}(h)}(xy-yx)| \ x, y \in \widehat{\mathcal{U}^H}^{\sigma\otimes n}\}=\text{Span}_{\mathbb{C}}\{xy-\varphi_{Q_L}(y)x\ | \ x, y \in \widehat{\mathcal{U}^H}^{\sigma\otimes n}\}$. I do not know if the following is true: is  $N_{Q_L}$ equal to $\text{Span}_{\mathbb{C}}\{xy-\varphi_{Q_L}(y)x\ | \ x, y \in \mathcal{U}^{\sigma\otimes n}\}$.
\end{rmq}

Note also that $\mathcal{J}_{L}$ belongs in $\left(\widehat{\mathcal{U}^{H}}^{\sigma\otimes n}\right)^{\widehat{\mathcal{U}^{H}}^{\sigma}}$ where\\
 $\left(\widehat{\mathcal{U}^{H}}^{\sigma\otimes n}\right)^{\widehat{\mathcal{U}^{H}}^{\sigma}}=\left\{u\in \widehat{\mathcal{U}^{H}}^{\sigma\otimes n}| u\Delta^{[n]}(x)=\Delta^{[n]}(x) u\right\}$ for all $x \in \widehat{\mathcal{U}^{H}}^{\sigma}$. A proof of this assertion can be seen in Lemma 6 \cite{BBGe17}.

\section{Invariant of $3$-manifolds of Hennings type}
In the article \cite{Henning96}, Hennings proposed a method to construct an invariant of $3$-manifolds from an universal invariant of links by using a finite dimensional ribbon algebra with its right integral. The invariant of $3$-manifolds is computed from the universal invariant of links. The key point of the construction is the role of a right integral of the Hopf algebra \cite{Henning96}. It is well known that it always exists a right integral on a finite dimensional Hopf algebra. Virelizier generalised this fact by using the notions of a finite type unimodular ribbon Hopf $\pi$-coalgebra and the right $\pi$-integral to construct an invariant of $3$-manifolds with $\pi$-structure. Here $\pi$ is a group and the structure is given by representation of the fundamental group in $\pi$ (see \cite{Vire01}). When $\pi=G$ is commutative a $G$-structure reduces to a $G$-valued cohomology class. 
In the case of the unrolled algebra $\mathcal{U}_{\xi}^{H}\mathfrak{sl}(2|1)$, the associated Hopf $G$-coalgebra can be ribbon but not finite type. However, we show that the associated Hopf $G$-coalgebra induces a finite type Hopf $G$-coalgebra by forgetting $h_1, h_2$.
We show that we can still construct an invariant of $3$-manifolds of Hennings type by working on the pairs $(M, \omega)$ in which $M$ is a $3$-manifold and $\omega$ is a cohomology class in $H^{1}(M, G)$. The construction of the invariant uses the discrete Fourier transform and the $G$-integrals for the finite type Hopf $G$-coalgebra associated with $\mathcal{U}_{\xi}\mathfrak{sl}(2|1)$. This invariant is a generalisation of the one in \cite{Vire02} that apply to $\mathcal{U}_{\xi}^H\mathfrak{sl}(2|1)$.
We recall some definitions from \cite{Oh93, Vire02}.
\subsection{Hopf $G$-coalgebra from the pivotal Hopf algebra $\mathcal{U}^{\sigma}$}
\begin{Def}
Let $\pi$ be a group. A $\pi$-coalgebra over $\mathbb{C}$ is a family $C=\{C_{\alpha}\}_{\alpha \in \pi}$ of $\mathbb{C}$-spaces endowed with a family $\Delta=\{\Delta_{\alpha, \beta}:\  C_{\alpha\beta}\rightarrow C_{\alpha}\otimes C_{\beta}\}_{\alpha, \beta \in \pi}$ of $\mathbb{C}$-linear maps (the coproduct) and a $\mathbb{C}$-linear map $\varepsilon:\ C_1 \rightarrow \mathbb{C}$ (the counit) such that
\begin{itemize}
\item $\Delta$ is coassociative, i.e. for any $\alpha, \beta, \gamma \in \pi$,
\begin{equation*}
(\Delta_{\alpha, \beta}\otimes \Id_{C_{\gamma}})\Delta_{\alpha\beta, \gamma}=(\Id_{C_{\alpha}}\otimes \Delta_{\beta, \gamma})\Delta_{\alpha, \beta\gamma},
\end{equation*} 
\item for all $\alpha \in \pi, \ (\Id_{C_{\alpha}}\otimes \varepsilon)\Delta_{\alpha,1}=\Id_{C_{\alpha}}=(\varepsilon \otimes \Id_{C_{\alpha}})\Delta_{1,\alpha}$.
\end{itemize}
\end{Def}
A {\em Hopf $\pi$-coalgebra} is a $\pi$-coalgebra $H=(\{H_{\alpha}\}_{\alpha \in \pi}, \Delta, \varepsilon)$ endowed with a family $S=\{S_{\alpha}:\ H_{\alpha}\rightarrow H_{\alpha^{-1}}\}_{\alpha\in \pi}$ of $\mathbb{C}$-linear maps (the antipode) such that
\begin{itemize}
\item each $H_{\alpha}$ is an algebra with product $m_{\alpha}$ and unit element $1_{\alpha}\in H_{\alpha}$,
\item $\varepsilon:\ H_{1} \rightarrow \mathbb{C}$ and $\Delta_{\alpha, \beta}:\ H_{\alpha\beta}\rightarrow H_{\alpha}\otimes H_{\beta}$ are algebra homomorphisms for all $\alpha, \beta \in \pi$,
\item for any $\alpha \in \pi$, 
\begin{equation*}
m_{\alpha}(S_{\alpha^{-1}}\otimes \Id_{H_{\alpha}})\Delta_{\alpha^{-1}, \alpha}=\varepsilon 1_{\alpha}=m_{\alpha}(\Id_{H_{\alpha}}\otimes S_{\alpha^{-1}})\Delta_{\alpha, \alpha^{-1}}.
\end{equation*}
\end{itemize}

A Hopf $\pi$-coalgebra is of {\em finite type} if $H_{\alpha}$ is finite dimensional algebra for any $\alpha\in \pi$.

Recall that 
$C=\mathbb{C}[k_{1}^{\pm \ell}, k_{2}^{\pm \ell}]$ is the commutative Hopf subalgebra in the center of $\mathcal{U}^{\sigma}$. Let $G=(\mathbb{C}/\mathbb{Z}\times \mathbb{C}/\mathbb{Z}, +) \xrightarrow{\sim} \Hom_{Alg}(C,\mathbb{C}),\ (\overline{\alpha}_1, \overline{\alpha}_2) \mapsto \left(k_{i}^\ell \mapsto \xi^{\ell\alpha_i}\right)$ for $i=1, 2$ and let $\mathcal{U}_{\overline{\alpha}}$ be the algebra $\mathcal{U}^{\sigma}$ modulo the
relations $k_{i}^{\ell}=\xi^{\ell\alpha_i}$ for $\overline{\alpha}\in G$.
\begin{Pro}
The family $\mathcal{U}^\sigma=\{\mathcal{U}_{\overline{\alpha}}\}_{\overline{\alpha} \in G}$ is a finite type Hopf $G$-coalgebra.
\end{Pro}
\begin{proof}
 By applying \cite[Example 2.3]{Ha18} it follows that $\{\mathcal{U}_{\overline{\alpha}}\}_{{\overline{\alpha}}\in G}$ is the Hopf $G$-coalgebra with the coproduct and the antipode determined by the commutative diagrams:
\begin{equation*}
\begin{diagram}
\mathcal{U}^{\sigma} &\rTo^{\Delta^{\sigma}} &\mathcal{U}^{\sigma} \otimes \mathcal{U}^{\sigma}\\
\dTo_{p_{\overline{\alpha}+\overline{\beta}}} & &\dTo_{p_{\overline{\alpha}}\otimes p_{\overline{\beta}}}\\
\mathcal{U}_{\overline{\alpha}+\overline{\beta}} &\rTo^{\Delta_{\overline{\alpha},\overline{\beta}}} &\mathcal{U}_{\overline{\alpha}} \otimes \mathcal{U}_{\overline{\beta}}
\end{diagram} \qquad \qquad \qquad
\begin{diagram}
\mathcal{U}^{\sigma} &\rTo^{S^{\sigma}} &\mathcal{U}^{\sigma} \\
\dTo_{p_{\overline{\alpha}}} & &\dTo_{p_{-\overline{\alpha}}}\\
\mathcal{U}_{\overline{\alpha}} &\rTo^{S_{\overline{\alpha}}} &\mathcal{U}_{-\overline{\alpha}}
\end{diagram}
\end{equation*}
where $p_{\overline{\alpha}}: \ \mathcal{U}^{\sigma} \rightarrow \mathcal{U}_{\overline{\alpha}}, \ x\mapsto [x]$ is the projection from $\mathcal{U}^{\sigma}$ to $\mathcal{U}_{\overline{\alpha}}$.
For $\overline{\alpha}=\overline{0}$ the Hopf algebra $\mathcal{U}_{\overline{0}}$ is called the restricted quantum $\mathfrak{sl}(2|1)$, i.e. the algebra $\mathcal{U}^{\sigma}$ modulo the relations $k_{i}^\ell=1$ for $i=1, 2$. Furthermore $\dim(\mathcal{U}_{\overline{\alpha}})=32\ell^4$ for $\overline{\alpha}\in G$. This finished the proof.
\end{proof}
\begin{Pro}\label{U_0 modular}
The small quantum group $\mathcal{U}_{\overline{0}}$ is unimodular.
\end{Pro}
\begin{proof}
Call $\mathscr{C}$ the even category of finite dimensional nipotent representations of $\mathcal{U}_{\xi}\mathfrak{sl}(2|1)$. 
We claim that the projective cover $P_{\mathbb{C}}$ of the trivial module is self dual: $P_{\mathbb{C}}\simeq P_{\mathbb{C}}^{*}$. The proof is analogous to Theorem 4.1 \cite{Ha16}. Furthermore $P_{\mathbb{C}}\in \mathcal{U}_{\overline{0}}\text{-mod}$ so the category  $\mathcal{U}_{\overline{0}}$-mod is unimodular. By Lemma 4.2.1 \cite{NgJkBp13} confirms that $\mathcal{U}_{\overline{0}}$ is unimodular.  
\end{proof}
A consequence of the proposition above is that the Hopf $G$-coalgebra $\mathcal{U}^{\sigma}$ is unimodular. 
\begin{Def}
A $\pi$-trace for a Hopf $\pi$-coalgebra $H=\{H_{\alpha}\}_{\alpha\in \pi}$ is a family of $\mathbb{C}$-linear forms $\tr =\{\tr^{\alpha}:\ H_{\alpha} \rightarrow \mathbb{C}\}_{\alpha\in \pi}$ which verifies
\begin{equation*}
\tr^{\alpha}(xy) = \tr^{\alpha}(yx),\ \tr^{\alpha^{-1}} (S_{\alpha}(x)) = \tr^{\alpha}(x)
\end{equation*}
for all $\alpha\in \pi$ and $x, y \in H_{\alpha}$.
\end{Def}
It is known that for each finite type Hopf $\pi$-coalgebra, there exists a family of linear forms called a family of the right $\pi$-integrals (\cite{Vire02}).
Call $\{\lambda_{\overline{\alpha}}\}_{\overline{\alpha} \in G}$ the family of right $G$-integral for the finite type Hopf $G$-coalgebra $\mathcal{U}^{\sigma}=\{\mathcal{U}_{\overline{\alpha}}\}_{\overline{\alpha} \in G}$. This means that the family of $\mathbb{C}$-linear forms $\lambda=(\lambda_{\overline{\alpha}})_{\overline{\alpha} \in G} \in \prod_{\overline{\alpha} \in G}\mathcal{U}_{\overline{\alpha}}^{*}$ satisfies 
\begin{equation}\label{definition of the integral}
(\lambda_{\overline{\alpha}}\otimes \Id_{\mathcal{U}_{\overline{\beta}}})\Delta_{\overline{\alpha}, \overline{\beta}}=\lambda_{\overline{\alpha}+\overline{\beta}}1_{\overline{\beta}}
\end{equation}
for all $\overline{\alpha}, \overline{\beta} \in G$ (see in Section 3 \cite{Vire02}). 
Note that $\lambda_{\overline{0}}$ is an usual right integral for the Hopf algebra $\mathcal{U}_{\overline{0}}$. 
We define a family of $\mathbb{C}$-linear forms $\{\tr^{\overline{\alpha}}\}_{\overline{\alpha} \in G}$ on $\mathcal{U}^{\sigma}$ determined by 
\begin{equation*}
\tr^{\overline{\alpha}}(x):=\lambda_{\overline{\alpha}}(G_{\overline{\alpha}}x)\quad \text{for}\ x \in \mathcal{U}_{\overline{\alpha}}
\end{equation*}
where $G_{\overline{\alpha}}=\sigma\phi_{0}|_{k_{i}^{\ell}=\xi^{\ell\alpha_i}}$ for $i=1, 2$, i.e. $G_{\overline{\alpha}}=\xi^{-\ell \alpha_1}\sigma k_{2}^{-2} \mod k_{i}^{\ell}-\xi^{\ell\alpha_i}$. 
This family determines a $G$-trace by proposition below.
\begin{Pro}\label{xd trace}
The family $\{\tr^{\overline{\alpha}}\}_{\overline{\alpha} \in G}$ above is a $G$-trace for the unimodular finite type Hopf $G$-coalgebra $\mathcal{U}^{\sigma}=\{\mathcal{U}_{\overline{\alpha}}\}_{\overline{\alpha} \in G}$.
\end{Pro}
\begin{proof}
We see that $\mathcal{U}^{\sigma}=\{\mathcal{U}_{\overline{\alpha}}\}_{\overline{\alpha} \in G}$ is an unimodular finite type Hopf $G$-coalgebra. By Theorem 4.2 \cite{Vire02} for $\mathcal{U}^{\sigma}$ one gets
\begin{align*}
\lambda_{\overline{\alpha}}(xy)&=\lambda_{\overline{\alpha}}\left(S_{-\overline{\alpha}}S_{\overline{\alpha}}(y)x\right),\\
\lambda_{-\overline{\alpha}}\left(S_{\overline{\alpha}}(x)\right)&=\lambda_{\overline{\alpha}}\left(G_{\overline{\alpha}}^{2}x\right)\ \text{and}\\
S_{-\overline{\alpha}}S_{\overline{\alpha}}(x)&=G_{\overline{\alpha}}xG_{\overline{\alpha}}^{-1} \quad \text{for}\ x, y \in \mathcal{U}_{\overline{\alpha}}.
\end{align*}
By the definition of $\{\tr^{\overline{\alpha}}\}_{\overline{\alpha} \in G}$ we have 
\begin{align*}
\tr^{\overline{\alpha}}(yx)&=\lambda_{\overline{\alpha}}(G_{\overline{\alpha}}yx)=\lambda_{\overline{\alpha}}\left(S_{-\overline{\alpha}}S_{\overline{\alpha}}(x)G_{\overline{\alpha}}y\right)\\
&=\lambda_{\overline{\alpha}}(G_{\overline{\alpha}}xy)=\tr^{\overline{\alpha}}(xy).
\end{align*}
Furthermore, for $x\in \mathcal{U}_{\overline{\alpha}}$
\begin{align*}
\lambda_{-\overline{\alpha}}\left(S_{\overline{\alpha}}(x)\right)&=\lambda_{-\overline{\alpha}}\left(S_{\overline{\alpha}}(x)S_{\overline{\alpha}}(G_{\overline{\alpha}})G_{-\overline{\alpha}}\right)\\
&=\lambda_{-\overline{\alpha}}\left(S_{\overline{\alpha}}(G_{\overline{\alpha}}x)G_{-\overline{\alpha}}\right)\\
&=\lambda_{-\overline{\alpha}}\left(S_{\overline{\alpha}}S_{-\overline{\alpha}}(G_{-\overline{\alpha}})S_{\overline{\alpha}}(G_{\overline{\alpha}}x)\right)\\
&=\lambda_{-\overline{\alpha}}\left(G_{-\overline{\alpha}}S_{\overline{\alpha}}(G_{\overline{\alpha}}x)\right)\\
&=\tr^{-\overline{\alpha}}\left(S_{\overline{\alpha}}(G_{\overline{\alpha}}x)\right)
\end{align*}
and $\lambda_{\overline{\alpha}}\left(G_{\overline{\alpha}}^{2}x\right)=\tr^{\overline{\alpha}}(G_{\overline{\alpha}}x)$ so $\tr^{-\overline{\alpha}} (S_{\overline{\alpha}}(x)) = \tr^{\overline{\alpha}}(x)$.
This implies that the family $\tr =(\tr^{\overline{\alpha}})_{ \overline{\alpha}\in G}$ is a $G$-trace for $\mathcal{U}^{\sigma}$.
\end{proof}
Note that, since $S_{-\overline{\alpha}}S_{\overline{\alpha}}(G_{\overline{\alpha}})=G_{\overline{\alpha}}$ for $x \in \mathcal{U}_{\overline{\alpha}}$ then 
\begin{equation*}
\lambda_{\overline{\alpha}}(G_{\overline{\alpha}}x)= \lambda_{\overline{\alpha}}(S_{-\overline{\alpha}}S_{\overline{\alpha}}(G_{\overline{\alpha}}) x)= \lambda_{\overline{\alpha}}(x G_{\overline{\alpha}}).
\end{equation*}
Thus we also have $\tr^{\overline{\alpha}}(x)=\lambda_{\overline{\alpha}}(x G_{\overline{\alpha}})$ for $x \in \mathcal{U}_{\overline{\alpha}}$.
\subsection{Discrete Fourier transform}
For a (partial) map $f:\ \mathbb{C}^{n} \rightarrow \mathbb{C}$ we define $t_i(f)$ by $t_i(f)(h_1, ..., h_n)=f(h_1, ..., h_i+1, ..., h_n)$ for $1\leq i\leq n$.
Let $\mathcal{L}_{\overline{\alpha}}=\{(\alpha_1, ..., \alpha_n)+\mathbb{Z}^n \}$ be the lattice of $\mathbb{C}^n$ corresponding to $\overrightarrow{\alpha}=\overline{\alpha}=(\overline{\alpha}_1, ..., \overline{\alpha}_n)\in (\mathbb{C}/\mathbb{Z})^n$.
A function $f(h_1, ..., h_n)\in \mathscr{C}^{\omega}(h_1,...,h_n)$ is called $\ell$-periodic in $h_i$ on the lattice $\mathcal{L}_{\overline{\alpha}}$ if it satisfies $f_{\mid_{\overline{\alpha}}}=t_{i}^{\ell}\left(f_{\mid_{\overline{\alpha}}}\right)$ where $f_{\mid_{\overline{\alpha}}}:=f_{\mid_{\mathcal{L}_{\overline{\alpha}}}}$. A function $f(h_1, ..., h_n)\in \mathscr{C}^{\omega}(h_1,...,h_n)$ is $\ell$-periodic on $\mathcal{L}_{\overline{\alpha}}$ if it is in all variables on $\mathcal{L}_{\overline{\alpha}}$.
The functions $\{\xi^{mh_i}\}_{m\in \mathbb{Z}}^{i=1, ..., n}$ are $\ell$-periodic and $\xi^{\ell h_i}-\xi^{\ell \alpha_i}$ are zero on $\overline{\alpha}$. Let $I$ be the ideal in the ring $R=\mathbb{C}[\xi^{\pm h_1}, ..., \xi^{\pm h_n}]$ generated by $\xi^{\ell h_i}-\xi^{\ell \alpha_i}$ for $1 \leq i \leq n$. Then an element of $R/I$ define a $\ell$-periodic map in all variable on $\mathcal{L}_{\overline{\alpha}}$.
\begin{Pro}[Discrete Fourier transform] \label{key}
Let $f=f(h_1, ..., h_n)\in \mathscr{C}^{\omega}(h_1,...,h_n)$ be a $\ell$-periodic function on $\mathcal{L}_{\overline{\alpha}}$. Then there is an unique element $\mathcal{F}_{\overrightarrow{\alpha}}(f)\in R/I$ which coincides with $f$ on $\mathcal{L}_{\overline{\alpha}}$ and is given by
\begin{equation*}
\mathcal{F}_{\overrightarrow{\alpha}}(f)=\sum_{m_1, ..., m_n=0}^{\ell-1}a_{m_1 ... m_n}\xi^{m_1h_1+...+m_nh_n}.
\end{equation*}
The coefficients $a_{m_1 ... m_n}$ (Fourier coefficients) are determined by
\begin{equation*}
a_{m_1 ... m_n}=\frac{1}{\ell^n}\sum_{i_1, ..., i_n=0}^{\ell-1}\xi^{-m_1(\alpha_1+i_1)-...-m_n(\alpha_n+i_n)}f(\alpha_1+i_1,..., \alpha_n+i_n).
\end{equation*}
\end{Pro}
\begin{proof}
We consider first the function $f(h_1) \in \mathscr{C}^{\omega}(h_1)$ is $\ell$-periodic on $\mathcal{L}_{\overline{\alpha}_1}$ for $\overline{\alpha}_1\in \mathbb{C}/\mathbb{Z}$ which is denoted by $f_{|_{\overline{\alpha}_1}}$.  
The set of such functions is a $\ell$-dimensional vector space.
The family $\{\xi^{m_1 h_1}\}_{m_1=0}^{\ell -1}$ of linearly independent $\ell$-periodic functions on $\mathcal{L}_{\overline{\alpha}_1}$ is a basis of this space, so we can write 
\begin{equation*}
f_{|_{\overline{\alpha}_1}}=\sum_{m=0}^{\ell-1}a_{m_1}\xi^{m_1h_1} .
\end{equation*}
To determine $(a_{m_1})_{m_1}$ we evaluate the function at $\alpha_1+i_1$ for $i_1=0,..., \ell -1$, we have a linear system of $\ell$ variables $a_{m_1}$ with $m_1=0, .., \ell-1$
\begin{equation*}
\sum_{m_1=0}^{\ell-1}a_{m_1}\xi^{m_1(\alpha_1+i_1)}=f(\alpha_1+i_1) \ \text{for}\ i_1=0,..., \ell-1.
\end{equation*}
The matrix of this linear system is 
$$A=\begin{bmatrix} 
1 & \xi^{\alpha_1} & \cdots & \xi^{(\ell-1)\alpha_1} \\
1 & \xi^{\alpha_1+1} & \cdots & \xi^{(\ell-1)(\alpha_1+1)} \\ 
 & \cdots &\cdots & \cdots\\
1 & \xi^{\alpha_1+k} & \cdots & \xi^{(\ell-1)(\alpha_1+k)} \\ 
 & \cdots &\cdots & \cdots\\
1 & \xi^{\alpha_1+\ell-1} & \cdots & \xi^{(\ell-1)(\alpha_1+\ell-1)}
\end{bmatrix}.$$
Note that $\sum_{k=0}^{\ell-1}\xi^{k(i-j)}=\ell \delta_{j}^{i}$, so we have
$$A^{-1}=\dfrac{1}{\ell}\begin{bmatrix} 
1 & 1 & \cdots & 1 \\
\xi^{-\alpha_1} & \xi^{-(\alpha_1+1)} & \cdots & \xi^{-(\alpha_1+\ell-1)} \\ 
 \vdots& \vdots &\cdots & \vdots\\
\xi^{-(\ell-1)\alpha_1} & \xi^{-(\ell-1)(\alpha_1+1)} & \cdots & \xi^{-(\ell-1)(\alpha_1+\ell-1)} 
\end{bmatrix}.$$
This implies that $a_{m_1}=\frac{1}{\ell}\sum_{i_1=0}^{\ell-1}\xi^{-m_1(\alpha_1+i_1)}f(\alpha_1+k_1)$ for $m_1=0, .., \ell-1$.
Then by induction on $i$ for $1\leq i \leq n$ we have a similar affirmation for the $\ell$-periodic functions on $\mathcal{L}_{\overline{\alpha}}$ with $\overline{\alpha}\in \left(\mathbb{C}/\mathbb{Z}\right)^n$ .
\end{proof}
Denote $\mathcal{U}_{\otimes\overrightarrow{\alpha}}:=\mathcal{U}_{\overline{\alpha}_1}\otimes ... \otimes \mathcal{U}_{\overline{\alpha}_n}$ for $\overrightarrow{\alpha}\in \left((\mathbb{C}/\mathbb{Z})^{2}\right)^{n}$ in which $\overline{\alpha}_j=(\overline{\alpha}_{1j}, \overline{\alpha}_{2j})\in (\mathbb{C}/\mathbb{Z})^{2}$ and $\mathcal{U}_{\otimes\overrightarrow{\alpha}}^{0}$ the subalgebra of $\mathcal{U}_{\otimes \overline{\alpha}}$ generated by $k_{i,j}^{\pm 1}=\xi^{\pm h_{i,j}}$ for $i=1, 2$ and $j=1,...,n$ (see Equation \eqref{Eq hij}).
\begin{HQ}\label{key1}
Let $f=f(h_{i,j})\in \mathscr{C}^{\omega}(h_{i,j})$ be a $\ell$-periodic function on $\mathcal{L}_{\overline{\alpha}}$. Then the unique element of $\mathcal{U}_{\otimes \overrightarrow{\alpha}}^{0}$ which coincides with $f$ on $\mathcal{L}_{\overline{\alpha}}$ and is given by
\begin{equation*}
\mathcal{F}_{\overrightarrow{\alpha}}(f)=\sum_{i_1, ..., i_n,\ 
j_1, ..., j_n=0}^{\ell-1}a_{i_1 ... i_nj_1 ... j_n}\prod_{s=1}^{n}k_{1,s}^{i_s}k_{2,s}^{j_s} \in \mathcal{U}_{\otimes \overrightarrow{\alpha}}^{0}
\end{equation*}
where
\begin{multline*}
a_{i_1 ... i_nj_1 ... j_n}=\frac{1}{\ell^{2n}}\sum_{s_1, ..., s_n,\ 
t_1, ..., t_n=0}^{\ell-1}\xi^{-\sum_{m=1}^{n}i_m(\alpha_{1m}+s_m)+j_m(\alpha_{2m}+t_m)}.\\
f(\alpha_{11}+s_1,\alpha_{21}+t_1,...,\alpha_{1n}+s_n, \alpha_{2n}+t_n).
\end{multline*}
\end{HQ}
\begin{proof}
By Proposition \ref{key} we have
\begin{equation*}
\mathcal{F}_{\overrightarrow{\alpha}}(f)=\sum_{i_1, ..., i_n,\ 
j_1, ..., j_n=0}^{\ell-1}a_{i_1 ... i_nj_1 ... j_n}\xi^{\sum_{s=1}^{n}i_s h_{1,s}+j_s h_{2,s}}.
\end{equation*}
Since $\xi^{h_{i,j}}=k_{i,j}$ for $i=1,2$ and $j=1,..., n$ then
\begin{equation*}
\mathcal{F}_{\overrightarrow{\alpha}}(f)=\sum_{i_1, ..., i_n,\ 
j_1, ..., j_n=0}^{\ell-1}a_{i_1 ... i_nj_1 ... j_n}\prod_{s=1}^{n}k_{1,s}^{i_s}k_{2,s}^{j_s} \in \mathcal{U}_{\otimes \overrightarrow{\alpha}}^{0}.
\end{equation*}
Proposition \ref{key} gives the formula determining the coefficients $a_{i_1 ... i_nj_1 ... j_n}$.
\end{proof}
\begin{Exem}
The function $\mathcal{K}=\xi^{-h_1 \otimes h_2 -h_2 \otimes h_1 - 2h_2 \otimes h_2}$ is $\ell$-periodic on $\mathcal{L}_{\overline{0}}$ and we have
\begin{equation}\label{F(K)}
\mathcal{F}_{\overrightarrow{0}}(\mathcal{K})=\dfrac{1}{\ell^2}\sum_{i_1, i_2, j_1, j_2=0}^{\ell -1}\xi^{i_1j_2+i_2j_1-2i_1i_2}k_{1}^{i_1}k_{2}^{j_1}\otimes k_{1}^{i_2}k_{2}^{j_2}\in \mathcal{U}_{\overline{0}}\otimes \mathcal{U}_{\overline{0}}.
\end{equation}
Indeed, by Corollary \ref{key1} one has
\begin{equation*}
\mathcal{F}_{\overrightarrow{0}}(\mathcal{K})=\sum_{i_1, i_2, j_1, j_2=0}^{\ell -1}a_{i_1 i_2 j_1 j_2}k_{1}^{i_1}k_{2}^{j_1}\otimes k_{1}^{i_2}k_{2}^{j_2}.
\end{equation*}
The coefficients $a_{i_1 i_2 j_1 j_2}$ are computed as below
{\allowdisplaybreaks
\begin{align*}
a_{i_1 i_2 j_1 j_2}&=\dfrac{1}{\ell^4}\sum_{s_1, s_2, t_1, t_2=0}^{\ell -1}\xi^{-i_1s_1-j_1t_1-i_2s_2-j_2t_2}\xi^{-s_1t_2-t_1s_2-2t_1t_2}\\
&=\dfrac{1}{\ell^4}\sum_{t_1, t_2=0}^{\ell -1}\xi^{-j_1t_1-j_2t_2-2t_1t_2}\sum_{s_1, s_2=0}^{\ell -1}\xi^{-i_1s_1-i_2s_2-s_1t_2-t_1s_2}\\
&=\dfrac{1}{\ell^4}\sum_{t_1, t_2=0}^{\ell -1}\xi^{-j_1t_1-j_2t_2-2t_1t_2}\sum_{s_1, s_2=0}^{\ell -1}\xi^{-(i_1+t_2)s_1-(i_2+t_1)s_2}\\
&=\dfrac{1}{\ell^4}\sum_{t_1, t_2=0}^{\ell -1}\xi^{-j_1t_1-j_2t_2-2t_1t_2}\sum_{s_1=0}^{\ell -1}\xi^{-(i_1+t_2)s_1}\sum_{ s_2=0}^{\ell -1}\xi^{-(i_2+t_1)s_2}\\
&=\dfrac{1}{\ell^4}\sum_{t_1, t_2=0}^{\ell -1}\xi^{-j_1t_1-j_2t_2-2t_1t_2}\ell \delta_{i_1+t_2\, \text{mod}\, \ell\mathbb{Z}}^{0}\ell \delta_{i_2+t_1\, \text{mod}\, \ell\mathbb{Z}}^{0}\\
&=\dfrac{1}{\ell^2}\sum_{t_1=0}^{\ell -1}\xi^{-j_1t_1}\delta_{i_2+t_1\, \text{mod}\, \ell\mathbb{Z}}^{0}\sum_{t_2=0}^{\ell -1}\xi^{-j_2t_2-2t_1t_2}\delta_{i_1+t_2\, \text{mod}\, \ell\mathbb{Z}}^{0}\\
&=\dfrac{1}{\ell^2}\sum_{t_1=0}^{\ell -1}\xi^{-j_1t_1}\delta_{i_2+t_1\, \text{mod}\, \ell\mathbb{Z}}^{0}\xi^{-j_2(-i_1)-2t_1(-i_1)}\\
&=\dfrac{1}{\ell^2}\xi^{-j_1(-i_2)}\xi^{-j_2(-i_1)-2(-i_2)(-i_1)}\\
&=\dfrac{1}{\ell^2}\xi^{j_1i_2+j_2i_1-2i_1i_2}.
\end{align*}}
\end{Exem}
For $\overline{\alpha}_i\in (\mathbb{C}/\mathbb{Z})^{2}$ we call $\widehat{\mathcal{U}^{H}}_{\overline{\alpha}_i}^{ \text{per}}$ the subalgebra of $\widehat{\mathcal{U}^{H}}^{\sigma}$ generated by elements forms $u=\sum_{j} f_{ij}(h_1,h_2)w_j $
where $w_j \in \sigma^{m}\mathfrak{B}_+  \mathfrak{B}_-$ for $m=0,1$ and $f_{ij}(h_1,h_2) \in \mathscr{C}^{\omega}(h_1,h_2)$ are $\ell$-periodic on $\mathcal{L}_{\overline{\alpha}_i}$. Denote $\widehat{\mathcal{U}^{H}}_{\otimes\overrightarrow{\alpha}}^{ \text{per}}=\widehat{\mathcal{U}^{H}}_{\overline{\alpha}_1}^{ \text{per}}\otimes ...\otimes \widehat{\mathcal{U}^{H}}_{\overline{\alpha}_n}^{ \text{per}}$.
We extend linearly $\mathcal{F}_{\overrightarrow{\alpha}}$ to a map $\widehat{\mathcal{U}^{H}}_{\otimes\overrightarrow{\alpha}}^{\text{per}} \rightarrow \mathcal{U}_{\otimes \overrightarrow{\alpha}}$ by the rule $\sum_{m} f_{m}(h_{1,i},h_{2,j})w_m  \mapsto \sum\mathcal{F}_{\overrightarrow{\alpha}}(f_{m}(h_{1,i},h_{2,j}))w_m$.
\begin{Lem}
The map $\mathcal{F}_{\overrightarrow{\alpha}}:\ \widehat{\mathcal{U}^{H}}_{\otimes\overrightarrow{\alpha}}^{\text{per}}\ \rightarrow \mathcal{U}_{\otimes \overrightarrow{\alpha}}$ is an algebra map.
\end{Lem}
\begin{proof}
By the unicity in Proposition \ref{key}, as $fg\vert_{\mathcal{L}_{\overline{\alpha}}}=\mathcal{F}_{\overrightarrow{\alpha}}(f)\mathcal{F}_{\overrightarrow{\alpha}}(g)$ we have
\begin{equation*}
\mathcal{F}_{\overrightarrow{\alpha}}(fg)=\mathcal{F}_{\overrightarrow{\alpha}}(f)\mathcal{F}_{\overrightarrow{\alpha}}(g)
\end{equation*}
for the $\ell$-periodic functions $f, g$ on $\mathcal{L}_{\overline{\alpha}}$.\\
Consider the elements $f(h_1, h_2)w_1, g(h_1, h_2)w_2\in \widehat{\mathcal{U}^{H}}_{\overline{\alpha}_i}^{\text{per}}$ where $f,g$ are $\ell$-periodic on $\mathcal{L}_{\overline{\alpha}_i}$ and $w_1, w_2\in  \sigma^{m}\mathfrak{B}_+  \mathfrak{B}_-$ for $m=0, 1$. By Remark \ref{cs} one has
\begin{align*}
(f(h_1, h_2)w_1)(g(h_1, h_2) w_2)&=f(h_1, h_2)(w_1 g(h_1, h_2))w_2\\
&=f(h_1, h_2)g(h_1+\vert w_1\vert_1, h_2+\vert w_1\vert_2)w_2
\end{align*}
where $(\vert w_1\vert_1, \vert w_1\vert_2)$ is the weight of $w_1$ for $(h_1,h_2)$. So we have
\begin{align*}
\mathcal{F}_{\overrightarrow{\alpha}}(fw_1gw_2)&=\mathcal{F}_{\overrightarrow{\alpha}}(f(h_1, h_2)g(h_1+\vert w_1\vert_1, h_2+\vert w_1\vert_2)w_1w_2)\\
&=\mathcal{F}_{\overrightarrow{\alpha}}(fg(h_1+\vert w_1\vert_1, h_2+\vert w_1\vert_2))w_1w_2\\
&=\mathcal{F}_{\overrightarrow{\alpha}}(f)\mathcal{F}_{\overrightarrow{\alpha}}(g(h_1+\vert w_1\vert_1, h_2+\vert w_1\vert_2))w_1w_2\\
&=\mathcal{F}_{\overrightarrow{\alpha}}(f)w_1 \mathcal{F}_{\overrightarrow{\alpha}}(g(h_1+\vert w_1\vert_1-\vert w_1\vert_1, h_2+\vert w_1\vert_2-\vert w_1\vert_2))w_2\\
&=\mathcal{F}_{\overrightarrow{\alpha}}(f)w_1 \mathcal{F}_{\overrightarrow{\alpha}}(g)w_2\\
&=\mathcal{F}_{\overrightarrow{\alpha}}(fw_1) \mathcal{F}_{\overrightarrow{\alpha}}(gw_2).
\end{align*}
\end{proof}
\begin{Lem}\label{coproduct and Fourier transform}
Let $\overline{\beta}, \overline{\gamma}\in (\mathbb{C}/\mathbb{Z})^2$ and let $\overrightarrow{\alpha}=\overline{\alpha}=\overline{\beta}+ \overline{\gamma}$. Assume $f(h_1, h_2)$ is a $\ell$-periodic entire function on $\mathcal{L}_{\overline{\alpha}}$. Then $\Delta(f)$ is $\ell$-periodic on $\mathcal{L}_{(\overline{\beta}, \overline{\gamma})}$ and
\begin{equation*}
\Delta_{\overline{\beta}, \overline{\gamma}}\mathcal{F}_{\overrightarrow{\alpha}}(f)=\mathcal{F}_{(\overline{\beta}, \overline{\gamma})}\left(\Delta(f)\right).
\end{equation*}
\end{Lem}
\begin{proof}
First, by Proposition \ref{key} we have
\begin{equation*}
\mathcal{F}_{\overrightarrow{\alpha}}(f)=\sum_{m_1,m_2=0}^{\ell -1}a_{m_1 m_2}\xi^{m_1 h_1 +m_2 h_2}
\end{equation*}
where $a_{m_1 m_2}=\dfrac{1}{\ell^2}\sum_{i_1, i_2=0}^{\ell -1}\xi^{-m_1 (\alpha_1+i_1)-m_2 (\alpha_2+i_2)}f(\alpha_1+i_1, \alpha_2+i_2)$.
Then 
\begin{equation*}
\Delta_{\overline{\beta}, \overline{\gamma}}\mathcal{F}_{\overrightarrow{\alpha}}(f)=\sum_{m_1,m_2=0}^{\ell -1}a_{m_1 m_2}\xi^{m_1 (h_1\otimes 1+1\otimes h_1) +m_2 (h_2\otimes 1+1\otimes h_2)}.
\end{equation*}
Second, the algebra homomorphism $\Delta$ gives us $\Delta f(h_1, h_2)=f(h_1\otimes 1+1\otimes h_1, h_2\otimes 1+1\otimes h_2)$.
Applying the discrete Fourier transform one gets
\begin{equation*}
\mathcal{F}_{(\overline{\beta}, \overline{\gamma})}\left(\Delta(f)\right)=\sum_{n_1,n_2,n_3,n_4=0}^{\ell -1}b_{n_1n_2n_3n_4}\xi^{n_1(h_1\otimes 1)+n_2(h_2\otimes 1)+n_3(1\otimes h_1)+n_4(1\otimes h_2)}
\end{equation*}
where the Fourier coefficient 
\begin{multline*}
b_{n_1n_2n_3n_4}=\dfrac{1}{\ell^4}\sum_{i_1,i_2,j_1,j_2=0}^{\ell -1}\xi^{-n_1(\beta_1+i_1)-n_2(\beta_2+i_2)-n_3(\gamma_1+j_1)-n_4(\gamma_2+j_2)}.\\
f(\beta_1+i_1+\gamma_1+j_1, \beta_2+i_2+\gamma_2+j_2).
\end{multline*}
By $\overline{\alpha}=\overline{\beta}+ \overline{\gamma}$, one has
\begin{multline*}
b_{n_1n_2n_3n_4}=\dfrac{1}{\ell^4}\xi^{-n_1\beta_1-n_2\beta_2-n_3\gamma_1-n_4\gamma_2}\sum_{i_1,i_2,j_1,j_2=0}^{\ell -1}\xi^{-n_1i_1-n_2i_2-n_3j_1-n_4j_2}.\\
f(\alpha_1+i_1+j_1, \alpha_2+i_2+j_2).
\end{multline*}
Since $f(h_1, h_2)$ is $\ell$-periodic on $\mathcal{L}_{\overline{\alpha}}$, setting $s=i_1+j_1$ and $t=i_2+j_2$ then 
\begin{multline*}
b_{n_1n_2n_3n_4}=\dfrac{1}{\ell^4}\xi^{-n_1\beta_1-n_2\beta_2-n_3\gamma_1-n_4\gamma_2}\sum_{i_1,i_2,s,t=0}^{\ell -1}f(\alpha_1+s, \alpha_2+t).\\
\xi^{-n_1i_1-n_2i_2-n_3(s-i_1)-n_4(t-i_2)}.
\end{multline*}
\begin{multline*}
b_{n_1n_2n_3n_4}=\dfrac{1}{\ell^4}\xi^{-n_1\beta_1-n_2\beta_2-n_3\gamma_1-n_4\gamma_2}\sum_{s,t=0}^{\ell -1}f(\alpha_1+s, \alpha_2+t)\xi^{-n_3 s-n_4 t}.\\
\sum_{i_1,i_2=0}^{\ell -1}\xi^{(n_3-n_1)i_1+(n_4-n_2)i_2}.
\end{multline*}
Since 
\begin{align*}
\sum_{i_1,i_2=0}^{\ell -1}\xi^{(n_3-n_1)i_1+(n_4-n_2)i_2}&=\sum_{i_1=0}^{\ell -1}\xi^{(n_3-n_1)i_1}\sum_{i_2=0}^{\ell -1}\xi^{(n_4-n_2)i_2}\\
&=\ell^2\delta_{n_3}^{n_1}\delta_{n_4}^{n_2}
\end{align*}
then $b_{n_1n_2n_3n_4}=0$ if $(n_1, n_2) \neq (n_3, n_4)$ and when $(n_1, n_2) = (n_3, n_4)$ then $b_{n_1n_2n_1n_2}$ is computed
\begin{align*}
b_{n_1n_2n_1n_2}&=\dfrac{1}{\ell^2}\xi^{-n_1(\beta_1+\gamma_1)-n_2(\beta_2+\gamma_2)}\sum_{s,t=0}^{\ell -1}f(\alpha_1+s, \alpha_2+t)\xi^{-n_1 s-n_2 t}\\
&=\dfrac{1}{\ell^2}\xi^{-n_1\alpha_1-n_2\alpha_2}\sum_{s,t=0}^{\ell -1}f(\alpha_1+s, \alpha_2+t)\xi^{-n_1 s-n_2 t}\\
&=\dfrac{1}{\ell^2}\sum_{s,t=0}^{\ell -1}\xi^{-n_1(\alpha_1+ s)-n_2(\alpha_2+ t)}f(\alpha_1+s, \alpha_2+t)\\
&=a_{n_1 n_2}.
\end{align*}
Hence 
\begin{align*}
\mathcal{F}_{(\overline{\beta}, \overline{\gamma})}\left(\Delta(f)\right)&=\sum_{n_1,n_2=0}^{\ell -1}b_{n_1n_2n_1n_2}\xi^{n_1(h_1\otimes 1)+n_2(h_2\otimes 1)+n_1(1\otimes h_1)+n_2(1\otimes h_2)}\\
&=\sum_{n_1,n_2=0}^{\ell -1}a_{n_1n_2}\xi^{n_1(h_1\otimes 1+1\otimes h_1)+n_2(h_2\otimes 1+1\otimes h_2)}\\
&=\Delta_{\overline{\beta}, \overline{\gamma}}\mathcal{F}_{\overrightarrow{\alpha}}(f).
\end{align*}
\end{proof}
\begin{rmq}\label{antipode and Fourier transform}
As $S(h_i)=-h_i$ for $i=1,2$, by the similar calculations as in Lemma \ref{coproduct and Fourier transform} then
\begin{equation*}
\mathcal{F}_{-\overrightarrow{\alpha}}S(f)=S_{\overline{\alpha}}\mathcal{F}_{\overrightarrow{\alpha}}(f).
\end{equation*}
\end{rmq}
A consequence of Lemma \ref{coproduct and Fourier transform} is that $\mathcal{R}^{\overline{0}}=R_1\check{\mathcal{R}}\mathcal{F}_{\overrightarrow{0}}(\mathcal{K})$ is the universal $R$-matrix of $\mathcal{U}_{\overline{0}}$ with $\mathcal{R}^{\overline{0}}=\mathcal{F}_{\overrightarrow{0}}(\mathcal{R}_q)$ is given by 
\begin{multline}\label{R matrix of U_0}
\mathcal{R}^{\overline{0}}=\frac{1}{\ell^2}R_1\sum_{i, i_1, i_2, j_1, j_2=0}^{\ell-1}\sum_{\rho, \delta=0}^{1}\frac{\{1\}^i(-\{1\})^{\rho+\delta}}{(i)_{\xi}!(\rho)_{\xi}!(\delta)_{\xi}!}\xi^{i_1j_2+i_2j_1-2i_1i_2}.\\
e_1^{i}e_3^{\rho}e_2^{\delta}k_1^{i_1}k_2^{j_1}\sigma^{\rho+\delta}\otimes f_1^{i}f_3^{\rho}f_2^{\delta}k_1^{i_2}k_2^{j_2}
\end{multline}
where $R_1=\dfrac{1}{2}\left( 1\otimes 1+\sigma \otimes 1 + 1 \otimes\sigma-\sigma \otimes \sigma \right)$ (see Section \ref{bosonization}). Indeed the relations satisfied by the $\mathcal{R}$-matrix $\mathcal{R}_q$ (see \cite{SMkVNt91}, \cite{HYa94}) translate to the relations for $\mathcal{R}^{\overline{0}}$.

\subsection{Invariant of $3$-manifolds of Hennings type}
Let $L$ be a framed link in $S^3$ consisting of $n$ components (still denote by $L$ its link  diagram), $M$ be a $3$-manifold obtained by surgery along the link $L$. Let $\omega$ be an element of the cohomology group $H^{1}(M, G)$ (see Section 2 \cite{FcNgBp14}).
The value of the invariant of link $\mathcal{J}_{L}$ is in $\xi^{Q_{L}(h)}\mathcal{U}^{\sigma\otimes n}$. Let $\overline{\alpha}_j=\omega(m_j)=(\overline{\alpha}_{1j}, \overline{\alpha}_{2j})$ here $m_j$ is a meridian of the $j$-th component of $L$. Denote $\overline{\alpha}=(\overline{\alpha}_1, ..., \overline{\alpha}_n)$. Since $\omega$ is an element of the cohomology group $H^{1}(M, G)$ it vanishes on longitudes of $L$, this implies the relation $\sum_{j=1}^{n}\lk_{ij}\overline{\alpha}_j=\overline{0}, \ \forall i=1,...,n$. We have
\begin{Pro}
The function $f(h_{1,i},h_{2,j})=\xi^{Q_{L}(h)}$ is $\ell$-periodic on $\mathcal{L}_{\overline{\alpha}}$. 
\end{Pro}
\begin{proof}
We denote $h_{1,i}+\ell=1\otimes ...\otimes (h_1+\ell)\otimes ... \otimes 1$ where $h_1+\ell$ is in $i$-th position. We have
\begin{align*}
f(h_{1,i}+\ell,h_{2,j})&=\xi^{-\sum_{i,j=1}^{n}\lk_{ij}\left((h_1+\ell)_ih_{2,j}+h_{2,i}h_{1,j}+2h_{2,i}h_{2,j}\right)}\\
&=\xi^{\sum_{i,j=1}^{n}\lk_{ij}Q_{ij}(h)}\xi^{-\sum_{i,j=1}^{n}\lk_{ij}\ell h_{2,j}}\\
&=f(h_{1,i},h_{2,j})\xi^{-\sum_{i,j=1}^{n}\lk_{ij}\ell \alpha_{2j}}.
\end{align*}
The equalities $\sum_{j=1}^{n}\lk_{ij}\overline{\alpha}_j=\overline{0}$ imply that $\sum_{i,j=1}^{n}\lk_{ij} \alpha_{2j} \in \mathbb{Z}$. Hence we get $f(h_{1,i}+\ell,h_{2,j})=f(h_{1,i},h_{2,j})$. The computation is similar for the variables $h_{2,j}$.
\end{proof}
Lemma \ref{key} implies that $\mathcal{F}_{\overrightarrow{\alpha}}\left( \xi^{Q_{L}(h)}\right) \in \mathcal{U}_{\otimes \overrightarrow{\alpha}}$. We define 
\begin{equation}\label{gtri J sau bd fourier}
\mathcal{J}_{L}^{\omega}=\mathcal{F}_{\overrightarrow{\alpha}}\left(\mathcal{J}_{L}\right)\in \text{H\!H}_{0}(\mathcal{U}_{\otimes \overrightarrow{\alpha}})
\end{equation}
thanks to Theorem \ref{gtri J}.
Let $\theta_{\overline{0}}$ be the ribbon element of the small quantum group $\mathcal{U}_{\overline{0}}$. 
\begin{Lem}\label{l'integral sur twist non nul}
There exists a normalization of $(\lambda_{\overline{\alpha}})_{\overline{\alpha}\in G}$ such that $$\lambda_{\overline{0}}(\theta_{\overline{0}})=\lambda_{\overline{0}}(\theta_{\overline{0}}^{-1})=1.$$ 
\end{Lem}
\begin{proof}
The proof is thanks to Lemma \ref{value of the integral in degree 0}.
\end{proof}
\begin{The} \label{dl chinh1}
\begin{equation}\label{Invariant of 3-manifolds}
\mathcal{J}(M, \omega)=\bigotimes_{j=1}^{n}\tr^{\overline{\alpha}_j}\left(\mathcal{J}_{L}^{\omega}\right)
\end{equation}
is a topological invariant of the pairs $(M, \omega)$ where $n$ is the number of the components of the surgery link $L$.
\end{The}
\begin{rmq}
Usual quantum surgery invariants are renormalized thanks to the signature. There is no need of renormalisation here thanks to Lemma \ref{l'integral sur twist non nul}. 
\end{rmq}
We use a result on the equivalence of $3$-manifolds obtained by surgery along a link to prove Theorem \ref{dl chinh1}, that is the theorem below.
\begin{The}[\cite{kirby78}]
Let $M_1$ and $M_2$ be oriented $3$-manifolds and $f :\ M_1 \rightarrow M_2$ be an orientation preserving diffeomorphism. Any two surgery presentations $L_1$ and $L_2$ of $M_1$ and $M_2$, respectively can be connected by a sequence of handle-slides, blow-up moves and blow-down moves such that the induced diffeomorphism between $M_1 = S_{L_1}^3$ and $M_2 = S_{L_2}^3$ is isotopic to $f$.
\end{The}
 
\begin{proof}[Proof of Theorem \ref{dl chinh1}]
We need to show that $\mathcal{J}(M, \omega)$ is not change under two Kirby's moves.
In the case of handle slide (the second Kirby's move), we can assume that the algebraic element on the strands are already concentrated as illustrated in the first component of Figure \ref{Fig3} where $x\in \mathcal{U}_{\overline{\alpha}},\ y \in \mathcal{U}_{\overline{\beta}}$ are given by the Fourier transform discrete (see Equation \eqref{gtri J sau bd fourier}).
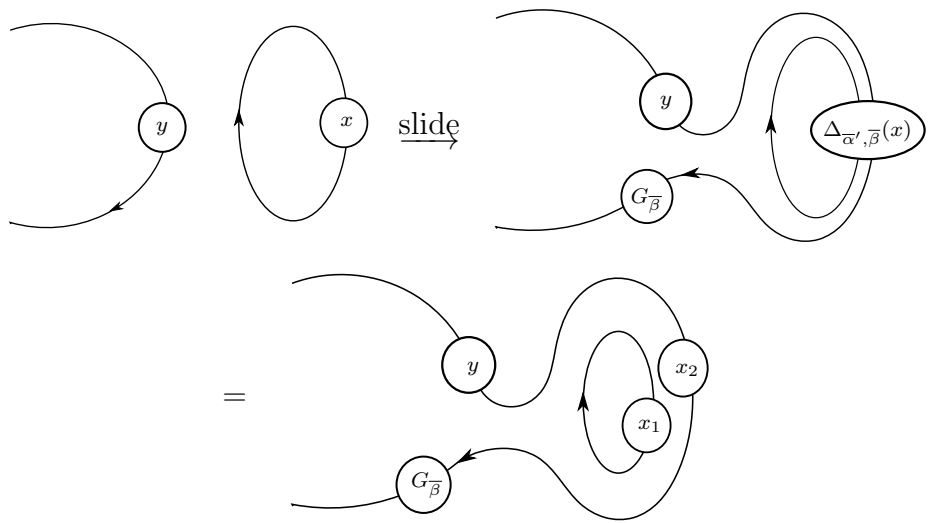
\begin{figure}
$$\epsh{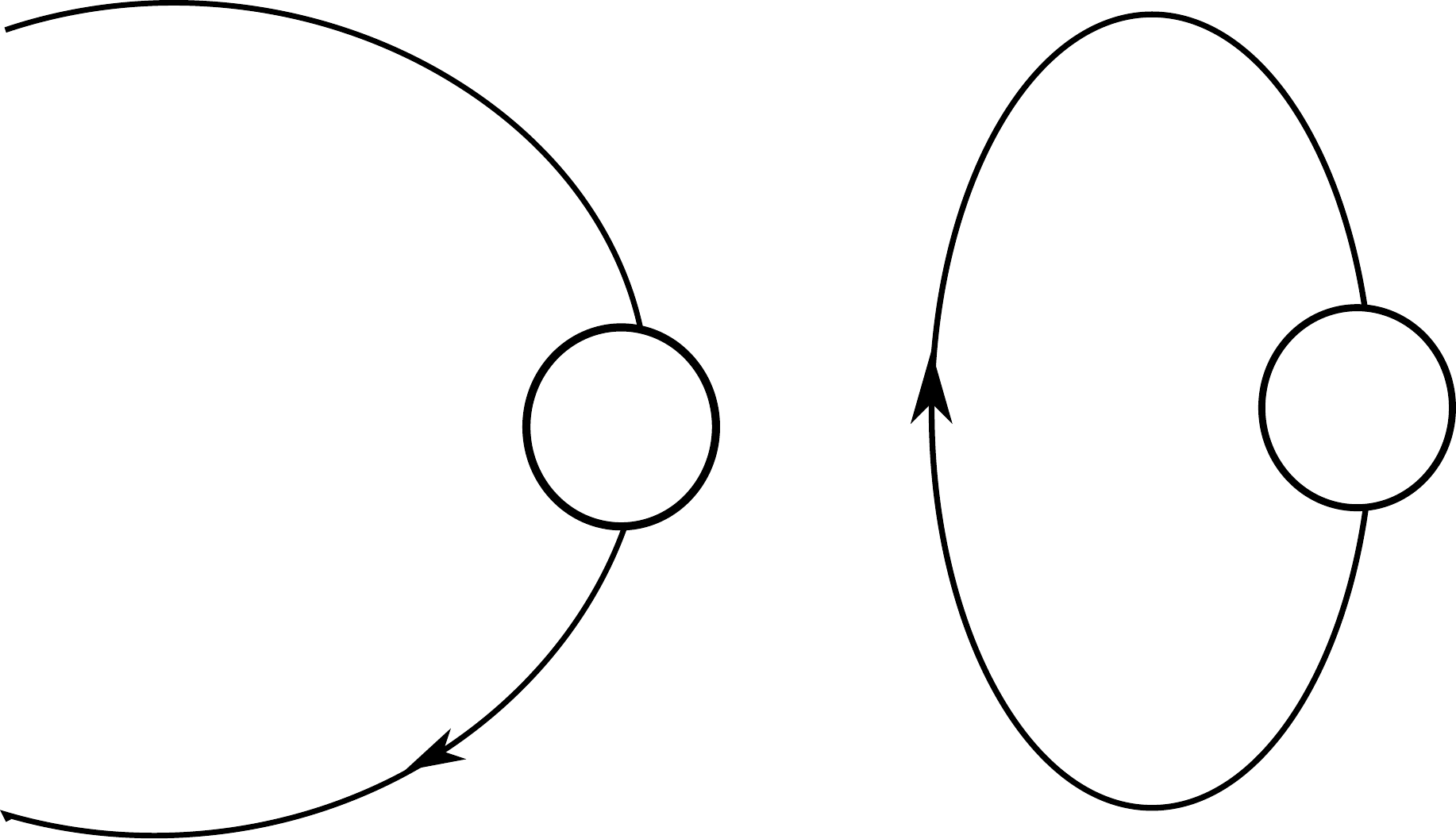}{15ex} \put(-79,1){\ms{y}}  \put(-10,3){\ms{x}}  \quad \underrightarrow{\text{slide}} \quad \epsh{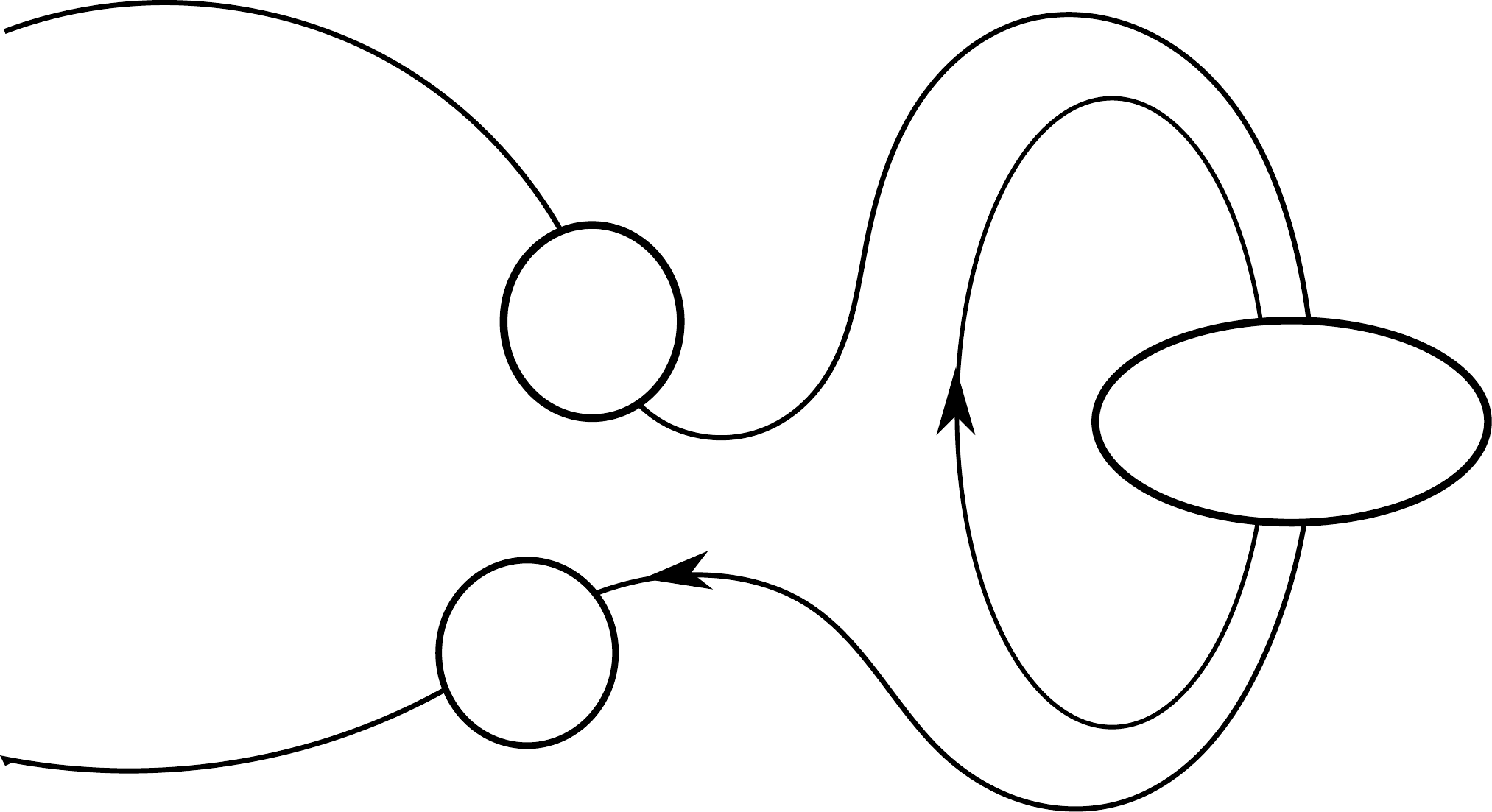}{17ex} \put(-99,11){\ms{y}} \put(-38,-1){\ms{\Delta_{\overline{\alpha}',\overline{\beta}}(x)}} \put(-110,-26){\ms{G_{\overline{\beta}}}} $$\\
$$= \quad \epsh{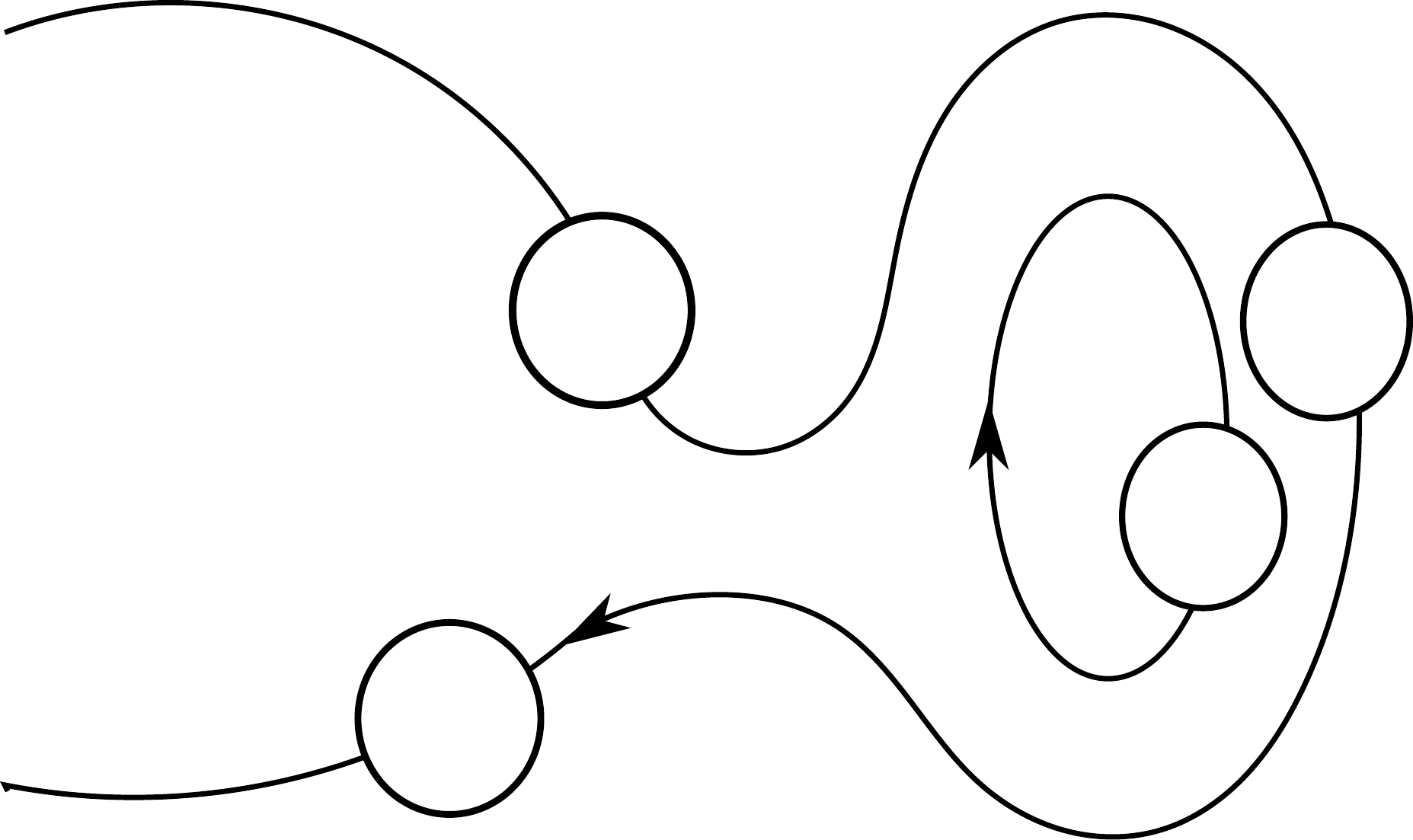}{18ex}\put(-90,13){\ms{y}} \put(-26,-9){\ms{x_1}} \put(-12,12){\ms{x_2}}\put(-111,-31){\ms{G_{\overline{\beta}}}}  $$
\caption{The second Kirby's move}
	\label{Fig3}
\end{figure}
The associated invariant of $3$-manifolds will be computed by $$\tr^{\overline{\alpha}}(x)y=\lambda_{\overline{\alpha}}(G_{\overline{\alpha}}x)y \quad x\in \mathcal{U}_{\overline{\alpha}}, y \in \mathcal{U}_{\overline{\beta}}.$$
After sliding, by the commutativity of the Fourier transform with the coproduct in Lemma \ref{coproduct and Fourier transform} and by the property of the element $\mathcal{R}$-matrix we replace $x$ by $\Delta_{\overline{\alpha}^{'}, \overline{\beta}}(x)=x_{1}\otimes x_{2}$ for $x\in \mathcal{U}_{\overline{\alpha}},\ x_1 \in \mathcal{U}_{\overline{\alpha}^{'}},\ x_2 \in \mathcal{U}_{\overline{\beta}}$ as in the second and third component of Figure \ref{Fig3}. Note that the relation between the homology classes of the meridians is $m_{x_1}+m_y=m_x$, i.e. $\omega(m_{x_1})+\omega(m_y)=\omega(m_x)\Leftrightarrow \overline{\alpha}^{'}+\overline{\beta}=\overline{\alpha}$. The invariant is determined by
$$\tr^{\overline{\alpha}^{'}}(x_{1})yx_{2}G_{\overline{\beta}}=\lambda_{\overline{\alpha}^{'}}(G_{\overline{\alpha}^{'}}x_{1})yx_{2}G_{\overline{\beta}}.$$
Furthermore, the definition of the right $G$-integral $\{\lambda_{\overline{\alpha}}\}_{\overline{\alpha}\in G}$ implies that $$\left(\lambda_{\overline{\alpha}^{'}}\otimes \Id_{\mathcal{U}_{\overline{\beta}}}\right)\Delta_{\overline{\alpha}^{'}, \overline{\beta}}(xG_{\overline{\alpha}})=\lambda_{\overline{\alpha}}(xG_{\overline{\alpha}})1_{\overline{\beta}}$$
then
\begin{equation*}
\lambda_{\overline{\alpha}^{'}}(x_{1}G_{\overline{\alpha}^{'}})x_{2}G_{\overline{\beta}}=\lambda_{\overline{\alpha}}(xG_{\overline{\alpha}})1_{\overline{\beta}}
\end{equation*}
and finally 
\begin{equation*}
\lambda_{\overline{\alpha}^{'}}(x_{1}G_{\overline{\alpha}^{'}})yx_{2}G_{\overline{\beta}}=\lambda_{\overline{\alpha}}(xG_{\overline{\alpha}})y,
\end{equation*}
i.e. $\mathcal{J}(M, \omega)$ is not change under the second Kirby's move. Changing the orientation of a component changes $\mathcal{J}_L$ by applying an antipode (see \cite{Ohtsuki02}), Proposition \ref{xd trace} and Remark \ref{antipode and Fourier transform} imply $\mathcal{J}(M, \omega)$ does not depend on the orientation. For the first move, the blowing up and blowing down, it is easy to see that $\omega(m)=\overline{0}$ for $m$ the meridian of $\pm 1$-framed loops and the two $\pm 1$-framed loops evaluate as $\lambda_{\overline{0}}(\theta_{\overline{0}})$ and $\lambda_{\overline{0}}(\theta_{\overline{0}}^{-1})$, respectively.
\end{proof}
Recall that the Hopf algebra $\mathcal{U}_{\overline{0}}$ has a PBW basis $\{f_1^{i}f_3^{\rho}f_2^{\delta}e_1^{i'}e_3^{\rho'}e_2^{\delta'}k_1^{j_1}k_2^{j_2}\sigma^m:\ 0\leq \rho, \delta, \rho', \delta', m\leq 1,\ 0\leq i, i', j_1, j_2\leq \ell-1\}$.
To prove Lemma \ref{l'integral sur twist non nul} we need the proposition below.
\begin{Pro}\label{proposition determine the integral}
The linear form $\lambda_{\overline{0}}:\ \mathcal{U}_{\overline{0}}\rightarrow \mathbb{C}$ determined by
\begin{equation}\label{formula of the integral}
\lambda_{\overline{0}}(f_1^{i}f_3^{\rho}f_2^{\delta}e_1^{i'}e_3^{\rho'}e_2^{\delta'}k_1^{j_1}k_2^{j_2}\sigma^m)=\eta\delta_{\ell-1}^{i}\delta_{1}^{\rho}\delta_{1}^{\sigma}\delta_{\ell-1}^{i'}\delta_{1}^{\rho'}\delta_{1}^{\sigma'}\delta_{0}^{j_1}\delta_{\ell-2}^{j_2}\delta_{0}^{m}
\end{equation}
is a right integral of $\mathcal{U}_{\overline{0}}$ 
where $\eta\in \mathbb{C}^*$ is a constant and $\delta_{j}^{i}$ is Kronecker symbol.
\end{Pro}
\begin{proof}
See in Appendix \ref{appendix1}.
\end{proof}
By Equation \eqref{definition of the integral} we have the remark.
\begin{rmq}
For $\overline{\alpha}=(\overline{\alpha}_1, \overline{\alpha}_2)\in \mathbb{C}/\mathbb{Z}\times \mathbb{C}/\mathbb{Z}$ then
\begin{equation*}
\lambda_{\overline{\alpha}}(f_1^{i}f_3^{\rho}f_2^{\delta}e_1^{i'}e_3^{\rho'}e_2^{\delta'}k_1^{j_1}k_2^{j_2}\sigma^m)=\eta\xi^{\ell (\alpha_1+\alpha_2)}\delta_{\ell-1}^{i}\delta_{1}^{\rho}\delta_{1}^{\sigma}\delta_{\ell-1}^{i'}\delta_{1}^{\rho'}\delta_{1}^{\sigma'}\delta_{0}^{j_1}\delta_{\ell-2}^{j_2}\delta_{0}^{m}
\end{equation*}
is a right $G$-integral for the Hopf $G$-coalgebra $\{\mathcal{U}_{\overline{\alpha}}\}_{\overline{\alpha} \in G}$.
\end{rmq}
By using Proposition \ref{proposition determine the integral} one gets the lemma.
\begin{Lem}\label{value of the integral in degree 0}
We have
\begin{equation*}
\lambda_{\overline{0}}(\theta_{\overline{0}})=\lambda_{\overline{0}}(\theta_{\overline{0}}^{-1})=\frac{\{1\}^{\ell +1}(1-\xi)^{\ell-1}}{\ell (\ell-1)}\eta.
\end{equation*}
\end{Lem}
\begin{proof}
See in Appendix \ref{appendix2}.
\end{proof}
\newpage

\appendix
\section{Proof of Proposition \ref{proposition determine the integral} and Lemma \ref{value of the integral in degree 0}}
\subsection{Proof of Proposition \ref{proposition determine the integral}}\label{appendix1}
It is necessary to check $\lambda_{\overline{0}}$ satisfies the condition \eqref{definition of the integral}, i.e.
\begin{equation}\label{definition of the integral in degree 0}
(\lambda_{\overline{0}}\otimes \Id_{\mathcal{U}_{\overline{0}}})\Delta(x)=\lambda_{\overline{0}}(x)1
\end{equation} 
for all $x\in \mathcal{U}_{\overline{0}}$. We check Equation \eqref{definition of the integral in degree 0} for the elements in PBW basis. This equation holds true for all elements $f_1^{i}f_3^{\rho}f_2^{\delta}e_1^{i'}e_3^{\rho'}e_2^{\delta'}k_1^{j_1}k_2^{j_2}\sigma^m$ in which $(i, \rho, \delta, i', \rho', \delta', j_1, j_2, m)\neq (\ell-1, 1, 1, \ell-1, 1, 1, 0, \ell-2, 0)$. For $(i, \rho, \delta, i', \rho', \delta', j_1, j_2, m)= (\ell-1, 1, 1, \ell-1, 1, 1, 0, \ell-2, 0)$ we have the right hand side of Equation \eqref{definition of the integral in degree 0} at $w=f_1^{\ell-1}f_3f_2e_1^{\ell-1}e_3e_2k_2^{\ell-2}$ is equal to $\eta 1$. The left hand side of Equation \eqref{definition of the integral in degree 0} at $w$ is computed as follows. First, one has
\begin{align*}
\Delta(e_3)&=e_3\otimes 1+k_1^{-1}k_2^{-1}\sigma\otimes e_3+(\xi-\xi^{-1})e_2k_1^{-1}\otimes e_1,\\
\Delta(f_3)&=f_3\otimes k_1 k_2+\sigma\otimes f_3+(\xi^{-1}-\xi)f_1\sigma\otimes k_1f_2
\end{align*}
and one can write 
\begin{align*}
\Delta(e_1)^{\ell-1}&=e_1^{\ell-1}\otimes 1+k_1^{\ell-1}\otimes e_1^{\ell-1}+\sum_{u,v<\ell -1}c_{uv}e_1^{u}k_1^{-v}\otimes e_1^{v},\\
\Delta(f_1)^{\ell-1}&=f_1^{\ell-1}\otimes k_1^{\ell-1}+1\otimes f_1^{\ell-1}+\sum_{u',v'<\ell -1}c'_{u'v'}f_1^{u'}\otimes k_1^{u'}f_1^{v'}
\end{align*}
where $c_{uv}, c'_{u'v'}$ are the coefficients in $\mathbb{C}$ and the powers of $e_1, f_1$ and $k_1$ are less then $\ell -1$. \\
Then we have the decomposition 
\begin{align*}
&\Delta(w)=\Delta(f_1)^{\ell-1}\Delta(f_3)\Delta(f_2)\Delta(e_1)^{\ell-1}\Delta(e_3)\Delta(e_2)\Delta(k_2^{\ell-2})\\
&=(f_1^{\ell-1}\otimes k_1^{\ell-1})(f_3\otimes k_1 k_2)(f_2\otimes k_2)(e_1^{\ell-1}\otimes 1)(e_3\otimes 1)(e_2\otimes 1)(k_2^{\ell-2}\otimes k_2^{\ell-2})\\
&\quad +\sum c_{i,\rho,\delta,j,\rho',\delta',j_1, j_2,m}^{i',\rho_1,\delta_1,j',\rho_1',\delta_1',j_1', j_2'}f_1^{i}f_3^{\rho}f_2^{\delta}e_1^{j}e_3^{\rho'}e_2^{\delta'}k_1^{j_1}k_2^{j_2}\sigma^m \otimes f_1^{i'}f_3^{\rho_1}f_2^{\delta_1}e_1^{j'}e_3^{\rho_1'}e_2^{\delta_1'}k_1^{j_1'}k_2^{j_2'}
\end{align*} 
where the terms in the sum satisfy $(i, \rho, \delta, j, \rho', \delta', j_1, j_2, m)\neq (\ell-1, 1, 1, \ell-1, 1, 1, 0, \ell-2, 0)$. By Equation \eqref{formula of the integral} and $k_i^{\ell}=1$ for $i=1,2$ the decomposition above implies that 
\begin{align*}
(\lambda_{\overline{0}}\otimes \Id_{\mathcal{U}_{\overline{0}}})\Delta(w)&=
(\lambda_{\overline{0}}\otimes \Id_{\mathcal{U}_{\overline{0}}})((f_1^{\ell-1}\otimes k_1^{\ell-1})(f_3\otimes k_1 k_2)(f_2\otimes k_2).\\
&\qquad\qquad\qquad \quad (e_1^{\ell-1}\otimes 1)(e_3\otimes 1)(e_2\otimes 1)(k_2^{\ell-2}\otimes k_2^{\ell-2}))\\
&=(\lambda_{\overline{0}}\otimes \Id_{\mathcal{U}_{\overline{0}}})(f_1^{\ell-1}f_3f_2e_1^{\ell-1}e_3e_2k_2^{\ell-2}\otimes k_1^{\ell-1}k_1 k_2k_2k_2^{\ell-2}),
\end{align*}
i.e. 
\begin{align*}
(\lambda_{\overline{0}}\otimes \Id_{\mathcal{U}_{\overline{0}}})\Delta(w)&=(\lambda_{\overline{0}}\otimes \Id_{\mathcal{U}_{\overline{0}}})(f_1^{\ell-1}f_3f_2e_1^{\ell-1}e_3e_2k_2^{\ell-2}\otimes 1)\\
&=\lambda_{\overline{0}}(f_1^{\ell-1}f_3f_2e_1^{\ell-1}e_3e_2k_2^{\ell-2})1\\
&=\eta 1.
\end{align*}
Thus the linear form $\lambda_{\overline{0}}$ is a right integral of $\mathcal{U}_{\overline{0}}$.
\subsection{Proof of Lemma \ref{value of the integral in degree 0}}\label{appendix2}
Firstly, we represent the decomposition of $\theta_{\overline{0}}^{-1}$ in a PBW basis of $\mathcal{U}_{\overline{0}}$. By Equation \eqref{twist not sigma} the ribbon element $\theta_{\overline{0}}$ of $\mathcal{U}_{\overline{0}}$ is determined by
\begin{equation*}
\theta_{\overline{0}} =\phi_{0}^{\sigma}.(m \circ \tau^{s} \circ (\Id \otimes S_{\overline{0}})(\mathcal{R}^{\overline{0}}))^{-1},
\end{equation*}
i.e. 
\begin{equation}\label{inverse of twist degree 0}
\theta_{\overline{0}}^{-1} =m \circ \tau^{s} \circ (\Id \otimes S_{\overline{0}})(\mathcal{R}^{\overline{0}}).(\phi_{0}^{\sigma})^{-1}.
\end{equation}
In Equation \eqref{inverse of twist degree 0} the terms are determined by
\begin{equation*}
(\phi_{0}^{\sigma})^{-1}=\phi_{0}^{-1}\sigma^{-1}=k_2^2\sigma
\end{equation*}
and
\begin{multline*}
\mathcal{R}^{\overline{0}}=\frac{1}{\ell^2}R_1\sum_{i, i_1, i_2, j_1, j_2=0}^{\ell-1}\sum_{\rho, \delta=0}^{1}\frac{\{1\}^i(-\{1\})^{\rho+\delta}}{(i)_{\xi}!(\rho)_{\xi}!(\delta)_{\xi}!}\xi^{i_1j_2+i_2j_1-2i_1i_2}.\\
e_1^{i}e_3^{\rho}e_2^{\delta}k_1^{i_1}k_2^{j_1}\sigma^{\rho+\delta}\otimes f_1^{i}f_3^{\rho}f_2^{\delta}k_1^{i_2}k_2^{j_2}
\end{multline*}
where $R_1=\frac{1}{2}\left( 1\otimes 1+\sigma \otimes 1 + 1 \otimes\sigma-\sigma \otimes \sigma \right)=\frac{1}{2}\sum_{m,n=0}^1(-1)^{mn}\sigma^m\otimes\sigma^n$, i.e.
\begin{multline*}
\mathcal{R}^{\overline{0}}=\frac{1}{2\ell^2}\sum_{i, i_1, i_2, j_1, j_2=0}^{\ell-1}\sum_{m,n,\rho, \delta=0}^{1}(-1)^{mn}\frac{\{1\}^i(-\{1\})^{\rho+\delta}}{(i)_{\xi}!(\rho)_{\xi}!(\delta)_{\xi}!}\xi^{i_1j_2+i_2j_1-2i_1i_2}.\\
\sigma^m e_1^{i}e_3^{\rho}e_2^{\delta}k_1^{i_1}k_2^{j_1}\sigma^{\rho+\delta}\otimes \sigma^n f_1^{i}f_3^{\rho}f_2^{\delta}k_1^{i_2}k_2^{j_2}.
\end{multline*}
Since 
\begin{align*}
S_{\overline{0}}&(\sigma^n f_1^{i}f_3^{\rho}f_2^{\delta}k_1^{i_2}k_2^{j_2})\\
&=S_{\overline{0}}(k_2^{j_2})S_{\overline{0}}(k_1^{i_2})S_{\overline{0}}(f_2^{\delta})S_{\overline{0}}(f_3^{\rho})S_{\overline{0}}(f_1^{i})S_{\overline{0}}(\sigma^n)\\
&=k_2^{-j_2}k_1^{-i_2}(-1)^{\delta}\sigma^{\delta}f_2^{\delta}k_2^{-\delta}\sigma^{\rho}\left((-1)^{\rho}\xi^{-2\rho}f_3^{\rho}+(\xi^{-2\rho}-1)f_2^{\rho}f_1^{\rho}\right) k_1^{-\rho}k_2^{-\rho}(-f_1k_1^{-1})^i\sigma^n\\
&=(-1)^{\delta+i}k_2^{-j_2}k_1^{-i_2}\sigma^{\delta}f_2^{\delta}k_2^{-\delta}\sigma^{\rho}\left((-1)^{\rho}\xi^{-2\rho}f_3^{\rho}+(\xi^{-2\rho}-1)f_2^{\rho}f_1^{\rho}\right) k_1^{-\rho}k_2^{-\rho}\xi^{i(i-1)}f_1^i k_1^{-i}\sigma^n
\end{align*}
where in the second equality we used $$S_{\overline{0}}(f_3^{\rho})=\sigma^{\rho}\left((-1)^{\rho}\xi^{-2\rho}f_3^{\rho}+(\xi^{-2\rho}-1)f_2^{\rho}f_1^{\rho}\right) k_1^{-\rho}k_2^{-\rho},$$
then
\begin{align*}
&(\Id \otimes S_{\overline{0}})(\sigma^m e_1^{i}e_3^{\rho}e_2^{\delta}k_1^{i_1}k_2^{j_1}\sigma^{\rho+\delta}\otimes\sigma^n f_1^{i}f_3^{\rho}f_2^{\delta}k_1^{i_2}k_2^{j_2})\\
&=(-1)^{\delta+\rho+i}\xi^{i(i-1)-2\rho}\sigma^m e_1^{i}e_3^{\rho}e_2^{\delta}k_1^{i_1}k_2^{j_1}\sigma^{\rho+\delta} \otimes k_2^{-j_2}k_1^{-i_2}\sigma^{\delta}f_2^{\delta}k_2^{-\delta}
\sigma^{\rho}f_3^{\rho} k_1^{-\rho}k_2^{-\rho}f_1^i k_1^{-i}\sigma^n\\
&+(-1)^{\delta+i}(\xi^{-2\rho}-1)\xi^{i(i-1)}\sigma^m e_1^{i}e_3^{\rho}e_2^{\delta}k_1^{i_1}k_2^{j_1}\sigma^{\rho+\delta} \otimes k_2^{-j_2}k_1^{-i_2}\sigma^{\delta}f_2^{\delta}k_2^{-\delta}
\sigma^{\rho}f_2^{\rho}f_1^{\rho} k_1^{-\rho}k_2^{-\rho}f_1^i k_1^{-i}\sigma^n.
\end{align*}
We have
{\allowdisplaybreaks
\begin{align*}
& m \circ \tau^{s}\circ (\Id \otimes S_{\overline{0}})(\sigma^m e_1^{i}e_3^{\rho}e_2^{\delta}k_1^{i_1}k_2^{j_1}\sigma^{\rho+\delta}\otimes\sigma^n f_1^{i}f_3^{\rho}f_2^{\delta}k_1^{i_2}k_2^{j_2})\\
&=(-1)^{\delta+\rho+i}\xi^{i(i-1)-2\rho}(-1)^{\rho+\delta}k_2^{-j_2}k_1^{-i_2}\sigma^{\delta}f_2^{\delta}k_2^{-\delta}\sigma^{\rho}f_3^{\rho}k_1^{-\rho}k_2^{-\rho}f_1^i k_1^{-i}\sigma^{m+n} e_1^{i}e_3^{\rho}e_2^{\delta}k_1^{i_1}k_2^{j_1}\sigma^{\rho+\delta}\\
& +(-1)^{\delta+i+\delta+\rho}(\xi^{-2\rho}-1)\xi^{i(i-1)}k_2^{-j_2}k_1^{-i_2}\sigma^{\delta}f_2^{\delta}k_2^{-\delta}
\sigma^{\rho}f_2^{\rho}f_1^{\rho} k_1^{-\rho}k_2^{-\rho}f_1^i k_1^{-i}\sigma^{m+n}e_1^{i}e_3^{\rho}e_2^{\delta}k_1^{i_1}k_2^{j_1}\sigma^{\rho+\delta}\\
&=(-1)^{i}\xi^{i(i-1)-2\rho}k_2^{-j_2}k_1^{-i_2}\sigma^{\delta}f_2^{\delta}k_2^{-\delta}\sigma^{\rho}f_3^{\rho}k_1^{-\rho}k_2^{-\rho}f_1^i k_1^{-i}\sigma^{m+n}e_1^{i}e_3^{\rho}e_2^{\delta}k_1^{i_1}k_2^{j_1}\sigma^{\rho+\delta}\\
&+ (-1)^{i+\rho}(\xi^{-2\rho}-1)\xi^{i(i-1)}k_2^{-j_2}k_1^{-i_2}\sigma^{\delta}f_2^{\delta}k_2^{-\delta}
\sigma^{\rho}f_2^{\rho}f_1^{\rho} k_1^{-\rho}k_2^{-\rho}f_1^i k_1^{-i}\sigma^{m+n}e_1^{i}e_3^{\rho}e_2^{\delta}k_1^{i_1}k_2^{j_1}\sigma^{\rho+\delta}\\
&=X_1+X_2.
\end{align*}}
Since $k_i f_j=\xi^{-a_{ij}}f_j k_i,\ k_i e_j=\xi^{a_{ij}}e_j k_i,\ k_if_3=\xi^{-(a_{i1}+a_{i2})}f_3k_i,\ k_ie_3=\xi^{a_{i1}+a_{i2}}e_3k_i$ for $i,j=0, 1$ and $\sigma x=(-1)^{\deg{x}}x\sigma$ then 
{\allowdisplaybreaks
\begin{multline*}
X_1
=(-1)^{i}\xi^{i(i-1)-2\rho}\xi^{-(j_2+\delta)\rho-(j_2+\delta+\rho)i+(j_2+\delta+\rho)i+(j_2+\delta+\rho)\rho}.\\
\xi^{-i_2\delta+i_2\rho+2(i_2+\rho)i-2(i+i_2+\rho)i-(i+i_2+\rho)\rho+(i+i_2+\rho)\delta}.\\
\sigma^{\delta}f_2^{\delta}\sigma^{\rho}f_3^{\rho}f_1^i \sigma^{m+n}e_1^{i}e_3^{\rho}e_2^{\delta}k_1^{i_1-i_2-i-\rho}k_2^{j_1-j_2-\delta-\rho}\sigma^{\rho+\delta}\\
=(-1)^{i}\xi^{-i-i^2-i\rho+i\delta+\rho\delta-2\rho}
\sigma^{\delta}f_2^{\delta}\sigma^{\rho}f_3^{\rho}f_1^i\sigma^{m+n} e_1^{i}e_3^{\rho}e_2^{\delta}k_1^{i_1-i_2-i-\rho}k_2^{j_1-j_2-\delta-\rho}\sigma^{\rho+\delta}\\
=(-1)^{i+\delta\delta+(\rho+\delta)\rho+(\rho+\delta+m+n)(\rho+\delta)}\xi^{-i-i^2-i\rho+i\delta+\rho\delta-2\rho}.\\
f_2^{\delta}f_3^{\rho}f_1^i e_1^{i}e_3^{\rho}e_2^{\delta}k_1^{i_1-i_2-i-\rho}k_2^{j_1-j_2-\delta-\rho}\sigma^{2(\rho+\delta)+m+n}\\
=(-1)^{i+\delta\rho+(m+n)(\delta+\rho)}\xi^{-i-i^2-i\rho+i\delta+\rho\delta-2\rho}
f_2^{\delta}f_3^{\rho}f_1^i e_1^{i}e_3^{\rho}e_2^{\delta}k_1^{i_1-i_2-i-\rho}k_2^{j_1-j_2-\delta-\rho}\sigma^{m+n}
\end{multline*}}
and
\begin{multline*}
X_2=(-1)^{i+\rho+\rho\delta}(\xi^{-2\rho}-1)\xi^{i(i-1)+i\rho}k_2^{-j_2}k_1^{-i_2}\sigma^{\delta}k_2^{-\delta}
\sigma^{\rho}f_2^{\rho+\delta}f_1^{\rho+i} k_2^{-\rho}k_1^{-i-\rho}.\\
 \sigma^{m+n}e_1^{i}e_3^{\rho}e_2^{\delta}k_1^{i_1}k_2^{j_1}\sigma^{\rho+\delta}\\
=(-1)^{i+\rho+\rho\delta+(m+n)(\rho+\sigma)}(\xi^{-2\rho}-1)\xi^{-i-i^2-i\rho+\rho\delta+i\delta}.\\
f_2^{\rho+\delta}f_1^{\rho+i} e_1^{i}e_3^{\rho}e_2^{\delta}k_1^{i_1-i_2-i-\rho}k_2^{j_1-j_2-\rho-\delta}\sigma^{m+n}.
\end{multline*}
Thus we have 
\begin{multline*}
\theta_{\overline{0}}^{-1} =m \circ \tau^{s} \circ (\Id \otimes S_{\overline{0}})(\mathcal{R}^{\overline{0}})(\phi_{0}^{\sigma})^{-1}\\
= \frac{1}{2\ell^2}\sum_{i, i_1, i_2, j_1, j_2=0}^{\ell-1}\sum_{m,n,\rho, \delta=0}^{1}(-1)^{mn}\frac{\{1\}^i(-\{1\})^{\rho+\delta}}{(i)_{\xi}!(\rho)_{\xi}!(\delta)_{\xi}!}\xi^{i_1j_2+i_2j_1-2i_1i_2}(X_1+X_2)k_2^2\sigma.
\end{multline*}
Since 
\begin{multline*}
X_2k_2^2\sigma=(-1)^{i+\rho+\rho\delta+(m+n)(\rho+\sigma)}(\xi^{-2\rho}-1)\xi^{-i-i^2-i\rho+\rho\delta+i\delta}.\\
f_2^{\rho+\delta}f_1^{\rho+i} e_1^{i}e_3^{\rho}e_2^{\delta}k_1^{i_1-i_2-i-\rho}k_2^{j_1-j_2-\rho-\delta+2}\sigma^{m+n+1}
\end{multline*}
then by Proposition \ref{proposition determine the integral} implies that $\lambda_{\overline{0}}(X_2k_2^2\sigma)=0$. Hence we have
\begin{multline*}
\lambda_{\overline{0}}(\theta_{\overline{0}}^{-1})=-\frac{1}{2\ell^2} \sum_{i_1, i_2, j_1, j_2=0}^{\ell-1} \sum_{m,n=0}^{1}\frac{\{1\}^{\ell-1}(-\{1\})^{1+1}}{(\ell-1)_{\xi}!(1)_{\xi}!(1)_{\xi}!}(-1)^{mn}\xi^{-1+i_1j_2+i_2j_1-2i_1i_2}\\
\lambda_{\overline{0}}(f_2f_3f_1^{\ell-1} e_1^{\ell-1}e_3e_2k_1^{i_1-i_2-(\ell-1)-1}k_2^{j_1-j_2-1-1+2}\sigma^{m+n+1})\\
=-\frac{1}{2\ell^2}\frac{\{1\}^{\ell+1}\xi^{-1}}{(\ell-1)_{\xi}!} \eta'\sum_{i_1, i_2, j_1, j_2=0}^{\ell-1}\sum_{m,n=0}^{1} \xi^{i_1j_2+i_2j_1-2i_1i_2}.\\
\delta_{i_1-i_2\, \text{mod}\, \ell\mathbb{Z}}^0\delta_{j_1-j_2\, \text{mod}\, \ell\mathbb{Z}}^{\ell-2}\delta_{m+n+1\, \text{mod}\, 2\mathbb{Z}}^{0}
\end{multline*}
where $\eta'=\lambda_{\overline{0}}(f_2f_3f_1^{\ell-1} e_1^{\ell-1}e_3e_2k_2^{\ell-2})$, i.e.
{\allowdisplaybreaks
\begin{align*}
\lambda_{\overline{0}}(\theta_{\overline{0}}^{-1})&=-\frac{1}{2\ell^2}\frac{\{1\}^{\ell+1}\xi^{-1}}{(\ell-1)_{\xi}!}\eta'\sum_{ i_2, j_2=0}^{\ell-1}\xi^{i_2j_2+i_2(j_2+\ell-2)-2i_2i_2}\sum_{m,n=0}^{1}(-1)^{mn}\delta_{m+n+1\, \text{mod}\, 2\mathbb{Z}}^{0}\\
&=-2\frac{1}{2\ell^2}\frac{\{1\}^{\ell+1}\xi^{-1}}{(\ell-1)_{\xi}!}\eta'\sum_{ i_2=0}^{\ell-1}\xi^{-2i_2^2-2i_2}\sum_{ j_2=0}^{\ell-1}\xi^{2i_2j_2}\\
&=-\frac{1}{\ell^2}\frac{\{1\}^{\ell+1}\xi^{-1}}{(\ell-1)_{\xi}!}\eta'\sum_{ i_2=0}^{\ell-1}\xi^{-2i_2^2-2i_2} \ell\delta_{i_2\, \text{mod}\, \ell\mathbb{Z}}^0\\
&=-\frac{\{1\}^{\ell+1}\xi^{-1}}{\ell(\ell-1)_{\xi}!}\eta'.
\end{align*}}
To computer $\lambda_{\overline{0}}(\theta_{\overline{0}})$ we use the equality
\begin{equation*}
\theta_{\overline{0}} =\phi_{0}^{\sigma}.(m \circ \tau^{s} \circ (S_{\overline{0}}^2 \otimes \Id)(\mathcal{R}^{\overline{0}})).
\end{equation*}
Since
{\allowdisplaybreaks
\begin{align*}
&S_{\overline{0}}^2(e_1^i)=\phi_{0}^{\sigma} e_1^i (\phi_{0}^{\sigma})^{-1}=\sigma k_2^{-2}e_1^i k_2^2\sigma =\xi^{2i}e_1^i,\\
&S_{\overline{0}}^2(e_3^\rho)=(-1)^\rho\xi^{2\rho}e_3^\rho,\\
&S_{\overline{0}}^2(e_2^\delta)=(-1)^\delta e_2^\delta,\\
&S_{\overline{0}}^2(k_1^{i_1})=k_1^{i_1},\\
&S_{\overline{0}}^2(k_2^{j_1})=k_2^{j_1},\\
&S_{\overline{0}}^2(\sigma^{\rho+\delta})=\sigma^{\rho+\delta}
\end{align*}}
then
\begin{align*}
S_{\overline{0}}^2(\sigma^m e_1^{i}e_3^{\rho}e_2^{\delta}k_1^{i_1}k_2^{j_1}\sigma^{\rho+\delta})&=S_{\overline{0}}^2(\sigma^m)S_{\overline{0}}^2(e_1^i)S_{\overline{0}}^2(e_3^\rho)S_{\overline{0}}^2(e_2^\delta)S_{\overline{0}}^2(k_1^{i_1})S_{\overline{0}}^2(k_2^{j_1})S_{\overline{0}}^2(\sigma^{\rho+\delta})\\
&=(-1)^{\rho+\delta}\xi^{2(i+\rho)}\sigma^m e_1^ie_3^{\rho}e_2^{\delta}k_1^{i_1}k_2^{j_1}\sigma^{\rho+\delta}.
\end{align*}
It implies that
\begin{align*}
(S_{\overline{0}}^2 \otimes \Id)&(\sigma^m e_1^{i}e_3^{\rho}e_2^{\delta}k_1^{i_1}k_2^{j_1}\sigma^{\rho+\delta}\otimes\sigma^n f_1^{i}f_3^{\rho}f_2^{\delta}k_1^{i_2}k_2^{j_2})\\
&=(-1)^{\rho+\delta}\xi^{2(i+\rho)}\sigma^m e_1^ie_3^{\rho}e_2^{\delta}k_1^{i_1}k_2^{j_1}\sigma^{\rho+\delta}\otimes\sigma^n f_1^{i}f_3^{\rho}f_2^{\delta}k_1^{i_2}k_2^{j_2}
\end{align*}
then 
\begin{align*}
m \circ \tau^{s}&\circ (S_{\overline{0}}^2 \otimes \Id)(\sigma^m e_1^{i}e_3^{\rho}e_2^{\delta}k_1^{i_1}k_2^{j_1}\sigma^{\rho+\delta}\otimes\sigma^n f_1^{i}f_3^{\rho}f_2^{\delta}k_1^{i_2}k_2^{j_2})\\
&=(-1)^{\rho+\delta+\rho+\delta}\xi^{2(i+\rho)}\sigma^n f_1^{i}f_3^{\rho}f_2^{\delta}k_1^{i_2}k_2^{j_2}\sigma^m e_1^{i}e_3^{\rho}e_2^{\delta}k_1^{i_1}k_2^{j_1}\sigma^{\rho+\delta}\\
&=\xi^{2(i+\rho)}\xi^{i_2(2i+\rho-\delta)}\xi^{-j_2(i+\rho)}\sigma^n f_1^{i}f_3^{\rho}f_2^{\delta}\sigma^m e_1^{i}e_3^{\rho}e_2^{\delta}k_1^{i_1+i_2}k_2^{j_1+j_2}\sigma^{\rho+\delta}\\
&=(-1)^{(2n+m)(\rho+\delta)}\xi^{2(i+\rho)+i_2(2i+\rho-\delta)-j_2(i+\rho)}f_1^{i}f_3^{\rho}f_2^{\delta}e_1^{i}e_3^{\rho}e_2^{\delta}k_1^{i_1+i_2}k_2^{j_1+j_2}\sigma^{m+m+\rho+\delta}\\
&=(-1)^{m(\rho+\delta)}\xi^{2(i+\rho)+i_2(2i+\rho-\delta)-j_2(i+\rho)}f_1^{i}f_3^{\rho}f_2^{\delta}e_1^{i}e_3^{\rho}e_2^{\delta}k_1^{i_1+i_2}k_2^{j_1+j_2}\sigma^{m+m+\rho+\delta}.
\end{align*}
So we have
\begin{multline*}
\theta_{\overline{0}} = \sigma k_2^{-2} \frac{1}{2\ell^2}\sum_{i, i_1, i_2, j_1, j_2=0}^{\ell-1}\sum_{m,n,\rho, \delta=0}^{1}(-1)^{mn}\frac{\{1\}^i(-\{1\})^{\rho+\delta}}{(i)_{\xi}!(\rho)_{\xi}!(\delta)_{\xi}!}\xi^{i_1j_2+i_2j_1-2i_1i_2}.\\ 
(-1)^{m(\rho+\delta)}\xi^{2(i+\rho)+i_2(2i+\rho-\delta)-j_2(i+\rho)}f_1^{i}f_3^{\rho}f_2^{\delta}e_1^{i}e_3^{\rho}e_2^{\delta}k_1^{i_1+i_2}k_2^{j_1+j_2}\sigma^{m+m+\rho+\delta}\\
=\frac{1}{2\ell^2}\sum_{i, i_1, i_2, j_1, j_2=0}^{\ell-1}\sum_{m,n,\rho, \delta=0}^{1}\frac{\{1\}^i(-\{1\})^{\rho+\delta}}{(i)_{\xi}!(\rho)_{\xi}!(\delta)_{\xi}!}\xi^{2(i+\rho)+i_2(2i+\rho-\delta)-j_2(i+\rho)}.\\
(-1)^{m(\rho+\delta)+mn}\xi^{i_1j_2+i_2j_1-2i_1i_2}f_1^{i}f_3^{\rho}f_2^{\delta}e_1^{i}e_3^{\rho}e_2^{\delta}k_1^{i_1+i_2}k_2^{j_1+j_2-2}\sigma^{m+n+\rho+\delta+1}.
\end{multline*}
By Proposition \ref{proposition determine the integral} one has
\begin{multline*}
\lambda_{\overline{0}}(\theta_{\overline{0}})=\frac{1}{2\ell^2}\sum_{ i_1, i_2, j_1, j_2=0}^{\ell-1}\sum_{m,n=0}^1\frac{\{1\}^{\ell-1}(-\{1\})^{1+1}}{(\ell-1)_{\xi}!}(-1)^{mn}\xi^{-2i_2+i_1j_2+i_2j_1-2i_1i_2}\eta.\\
\delta_{i_1+i_2\, \text{mod}\, \ell\mathbb{Z}}^0\delta_{j_1+j_2-2\, \text{mod}\, \ell\mathbb{Z}}^0\delta_{m+n+1\, \text{mod}\,  2\mathbb{Z}}^0\\
=
\frac{1}{2\ell^2}\frac{\{1\}^{\ell+1}}{(\ell-1)_{\xi}!}\eta\sum_{i_1, i_2, j_1, j_2=0}^{\ell-1}\xi^{-2i_2+i_1j_2+i_2j_1-2i_1i_2}\delta_{i_1+i_2 \, \text{mod}\, \ell\mathbb{Z}}^0\delta_{j_1+j_2-2\, \text{mod}\, \ell\mathbb{Z}}^0.\\
\sum_{m,n=0}^1(-1)^{mn}\delta_{m+n+1\, \text{mod}\,  2\mathbb{Z}}^0
\end{multline*}
where $\eta=\lambda_{\overline{0}}(f_1^{\ell-1}f_3f_2e_1^{\ell-1}e_3e_2k_2^{\ell-2})$,
i.e.
{\allowdisplaybreaks
\begin{align*}
\lambda_{\overline{0}}(\theta_{\overline{0}})&=\frac{1}{\ell^2}\frac{\{1\}^{\ell+1}}{(\ell-1)_{\xi}!}\eta\sum_{i_1, j_1=0}^{\ell-1}\xi^{-2(\ell-i_1)+i_1(\ell-j_1+2)+(\ell-i_1)j_1-2i_1(\ell-i_1)}\\
&=\frac{1}{\ell^2}\frac{\{1\}^{\ell+1}}{(\ell-1)_{\xi}!}\eta\sum_{i_1=0}^{\ell-1}\xi^{2i_1^2+4i_1}\sum_{j_1=0}^{\ell-1}\xi^{-2i_1j_1}\\
&=\frac{1}{\ell^2}\frac{\{1\}^{\ell+1}}{(\ell-1)_{\xi}!}\eta\sum_{i_1=0}^{\ell-1}\xi^{2i_1^2+4i_1}\ell \delta_{i_1\, \text{mod}\, \ell\mathbb{Z}}^0\\
&=\frac{\{1\}^{\ell+1}}{\ell(\ell-1)_{\xi}!}\eta.
\end{align*}}
On the other hand
\begin{align*}
f_2f_3f_1^2&=\xi^{-2}f_2f_1^2f_3=\xi^{-2}((\xi+\xi^{-1})f_1f_2f_1-f_1^2f_2)f_3\\
&=\xi^{-2}(\xi^{-1}f_1(f_2f_1-\xi f_1f_2)+\xi f_1(f_3+\xi f_1 f_2))f_3\\
&=\xi^{-2}(\xi^{-1}f_1f_3+\xi f_1f_3+\xi^2 f_1^2f_2) f_3\\
&=f_1^2f_2 f_3
\end{align*}
where in the first equality one used the relation $f_3f_1=\xi^{-1}f_1f_3$ and the second Serre relation \eqref{relation serre 1}. Note that $\ell$ is odd  $f_2f_3f_1^{\ell-1}=f_1^{\ell-1}f_2f_3$. Since $f_2f_3=-\xi f_3f_2$ we get
\begin{align*}
\eta'&=\lambda_{\overline{0}}(f_2f_3f_1^{\ell-1} e_1^{\ell-1}e_3e_2k_2^{\ell-2})\\
&=-\xi\lambda_{\overline{0}}(f_1^{\ell-1}f_3f_2e_1^{\ell-1}e_3e_2k_2^{\ell-2})\\
&=-\xi \eta.
\end{align*}
Thus we have $$\lambda_{\overline{0}}(\theta_{\overline{0}}^{-1})=-\frac{\{1\}^{\ell+1}\xi^{-1}}{\ell(\ell-1)_{\xi}!}\eta'=\frac{\{1\}^{\ell+1}}{\ell(\ell-1)_{\xi}!}\eta=\lambda_{\overline{0}}(\theta_{\overline{0}}).$$
Since $(\ell-1)_{\xi}!=\prod_{i=1}^{\ell-1}\frac{1-\xi^i}{1-\xi}=\frac{\ell-1}{(1-\xi)^{\ell-1}}$ then $\lambda_{\overline{0}}(\theta_{\overline{0}})=\frac{\{1\}^{\ell+1}(1-\xi)^{\ell-1}}{\ell(\ell-1)}\eta$.

\bibliographystyle{plain} 
\bibliography{TKhao}

\end{document}